\documentclass[smallextended,final,nospthms]{svjour3}
\usepackage{mathptmx}
\usepackage{amsmath,amssymb,amsfonts,amsthm,amstext,verbatim}
\usepackage{enumerate,color,graphicx,stmaryrd,psfrag} 
\usepackage{url}
\usepackage{mdwlist}   
\usepackage{mathtools} 
\usepackage{wasysym}   

\RequirePackage[colorlinks,citecolor=blue,urlcolor=blue]{hyperref}

\newcommand\sC{{\mathcal C}}
\newcommand\sG{{\mathcal G}}

\newcommand\oo{\infty}

\newcommand\De{\Delta}

\newcommand\rad{\text{\rm rad}}
\renewcommand\b{\beta}

\newcommand\sT{{\mathcal T}}
\newcommand\sN{{\mathcal N}}
\newcommand\g{\gamma}
\renewcommand\a{\alpha}
\newcommand\ga{\gamma}
\newcommand\pd{\partial}

\newcommand\es{\varnothing}
\newcommand\Om{\Omega}
\newcommand\om{\omega}
\newcommand\Ga{\Gamma}
\newcommand\La{\Lambda}
\newcommand\Si{\Sigma}

\newcommand\sD{{\mathcal D}}

\newcommand\de{\delta}
\newcommand\resp{respectively}
\newcommand\lra{\leftrightarrow}
\newcommand\xlra{\xleftrightarrow}
\renewcommand\th{\theta}
\newcommand\bxp{box-crossing property}
\newcommand\BXP{\mbox{\rm BXP}}
\newcommand\bp{\mathbf{p}}

\newcommand\tri{{\triangle}}
\newcommand\hex{{\hexagon}}

\newcommand\Cv{C_{\text{\rm v}}}
\newcommand\Ch{C_{\text{\rm h}}}

\newcommand\bH{{\mathbf H}}
\newcommand\bR{{\mathbf R}}
\newcommand\bL{{\mathbf L}}
\newcommand\bX{{\mathbf X}}
\newcommand\bM{{\mathbf M}}
\newcommand\lest{\le_{\text{\rm st}}}
\newcommand\bigmid{\,\big|\,}

\newcommand\sP{{\mathcal P}}
\newcommand\sR{{\mathcal R}}

\newcommand\wt{\widetilde}

\newcommand\eps{\epsilon}

\newcommand\simh{\sim_{\text{\rm h}}}
\newcommand\simv{\sim_{\text{\rm v}}}
\newcommand\lan{\langle}
\newcommand\ran{\rangle}

\newcommand{\RR}{\mathbb{R}}     
\newcommand{\NN}{\mathbb{N}}     
\newcommand{\ZZ}{\mathbb{Z}}     

\newcommand{\PP}{\mathbb{P}}     
\newcommand{\comp}{\circ}        
\newcommand{\Dom}{\mathcal{D}}   
\newcommand{\B}{\mathcal{B}}     

\newcommand{\col}{\sC}
\newcommand{\Ann}{\mathcal{A}}

\newcommand\cd{c_d}

\newcommand\prob{black} 

\def\mik{1}
\newcommand\cpsfrag[2]{\ifnum\mik=1\psfrag{#1}{#2}\fi}

\newenvironment{romlist}{\begin{list}{\rm(\roman{mycount})}%
   {\usecounter{mycount}\leftmargin=1cm\labelwidth=1cm\itemsep 0pt}}{\end{list}}

\newenvironment{letlist}{\begin{list}{\rm(\alph{mycount})}%
   {\usecounter{mycount}\leftmargin=1cm\labelwidth=2cm\itemsep 0pt}}{\end{list}}

\newtheorem{thm}{Theorem}
\newtheorem{lemma}[thm]{Lemma}
\newtheorem{prop}[thm]{Proposition}
\newtheorem{cor}[thm]{Corollary}

\newtheorem{conj}[thm]{Conjecture}
\newtheorem{definition}[thm]{Definition}

\newtheorem{rem}[thm]{Remark}

\numberwithin{equation}{section}
\numberwithin{thm}{section}
\numberwithin{figure}{section}

\newcounter{mycount}

\newcommand{\Arm}{A}
\newcommand{\logeqv}{\approx}
\newcommand{\eqv}{\asymp}

\newcommand{\balpha}{\boldsymbol\alpha}
\newcommand{\bbeta}{\boldsymbol\beta}
\newcommand{\di}{\Diamond}
\newcommand\la{\lambda}
\newcommand\bac{bounded-angles property}
\newcommand\BAC{\mbox{\rm BAP}}
\newcommand\SGP{\mbox{\rm SGP}}
\newcommand\sgp{square-grid property}
 
\newcommand\para{parallel} 
\newcommand\stt{star--triangle transformation}
\newcommand\ZZO{\NN_0}
\newcommand\bef{\xi}
\newcommand\gest{\ge_{\text{\rm st}}}

\newcommand\nxlra[1]{\mathrel{\mathpalette\concel{\xleftrightarrow{#1}}}}
\newcommand\dxlra[1]{\xlra{#1}\!\!\!{}^*\,}
\newcommand\concel[2]{\ooalign{$\hfil#1\mkern0mu/\hfil$\crcr$#1#2$}}
\newcommand\ext{\text{\rm ext}}

\begin{document}
\title{Bond percolation on isoradial graphs:\\
Criticality and universality}
\titlerunning{Bond percolation on isoradial graphs}

\author{Geoffrey R. Grimmett \and Ioan Manolescu}
\authorrunning{Grimmett and Manolescu}

\institute{G. R. Grimmett \at Statistical Laboratory, Centre for
Mathematical Sciences, Cambridge University, Wilberforce Road,
Cambridge CB3 0WB, UK.
              \email{\url{g.r.grimmett@statslab.cam.ac.uk}},           
URL: {\url{http://www.statslab.cam.ac.uk/~grg/}}
\and
           I. Manolescu \at
            Statistical Laboratory, Centre for
Mathematical Sciences, Cambridge University, Wilberforce Road,
Cambridge CB3 0WB, UK. \emph{Present address\/}: 
Section
de Math\'ematiques, Universit\'e de Gen\`eve, 2--4 rue de Li\`evre,
1211 Gen\`eve 4, Switzerland. 
\email{\url{i.manolescu@statslab.cam.ac.uk}}
}

\date{1 April 2012, revised 6 March 2013}

\maketitle

\begin{abstract}
In an investigation of percolation on isoradial graphs, we prove the
criticality of canonical bond percolation  on  isoradial embeddings
of planar graphs, thus extending celebrated earlier results for homogeneous
and inhomogeneous square, triangular, and other lattices. 
This is achieved via the \stt, by transporting the box-crossing property  
across the family of  isoradial graphs. 
As a consequence, we obtain the universality of these models \emph{at}
the critical point, in the sense that the 
one-arm and $2j$-alternating-arm critical exponents (and therefore also
the connectivity and volume exponents) are constant
across the family of such percolation processes. The isoradial graphs in question are those
that satisfy certain weak conditions on their embedding
 and on their track system.
This class of graphs includes, for example, isoradial
embeddings of periodic graphs, and graphs derived from 
rhombic Penrose tilings.  
\end{abstract}

\section{Introduction}\label{sec:intro}

Two-dimensional disordered systems, when critical, are presumed to have properties
of  conformality and universality. Rigorous evidence
for this classic paradigm from conformal field theory includes the
recent analysis of the critical Ising model by Chelkak and Smirnov
\cite{Chelkak-Smirnov2,Smirnov1} on a family of graphs known as
`isoradial'. On the one hand, such isoradial graphs are especially harmonious
in a theory of discrete holomorphic functions (introduced by
Duffin, see 
\cite{Chelkak-Smirnov3,Duff,Merc}), and
on the other they are well adapted to transformations of star--triangle type
(explained by Kenyon \cite{Ken02}). These two properties resonate
with the intertwined concepts of conformality and universality. 

It is shown here that, for a broad class $\sG$ of isoradial graphs, the associated
bond percolation model is critical, and furthermore certain of its critical exponents are
constant across $\sG$. 
The class $\sG$ includes many specific models studied earlier, but is 
much more extensive than this restricted class.
In earlier related work of Kesten \cite{Kesten80,Kesten_book}, 
Wierman \cite{Wierman81} and others, it
has been assumed that the graph under study has certain invariances
under, for example,  translation, rotation and/or reflection. 
Such assumptions of regularity play no role in the current work, where the key assumptions are
that the graph under study is isoradial, and satisfies two weak conditions
on its embedding and on its so-called track system,
namely  the so-called \bac\ and the \sgp. 
The first of these is a condition of non-degeneracy,
and the second requires the graph to contain a square grid within
its track system. (Formal definitions are deferred to Sections \ref{sec:isog} and
\ref{sec:track}.)

In advance of the formalities of Section \ref{sec:def}, we explain the term `isoradial'.
Let $G$ be a planar graph embedded in the plane $\RR^2$. It is called \emph{isoradial}
if there exists $r>0$ such that, for every face $F$ of $G$, the vertices of $F$ lie on a circle 
of radius $r$ with centre in the interior of $F$. Note that isoradiality is a property of
the planar embedding of $G$ rather than of $G$ itself. By rescaling the embedding of $G$,
we may assume $r=1$.

It was noted by Duffin \cite{Duff}
that isoradial graphs are in two--one correspondence with rhombic tilings of the plane
(the name `isoradial' was coined later by Kenyon). While details of this correspondence
are deferred to  Section \ref{sec:def}, we highlight one fact here. Let $G=(V,E)$ be isoradial. 
An edge $e \in E$ lies in two faces, and therefore two circumcircles. 
As illustrated in Figure \ref{fig:isoradial}, $e$ subtends the same angle $\th_e \in (0,\pi)$
at the centres of these circumcircles, and we define $p_e \in (0,1)$ by
\begin{equation}\label{G100}
\frac{p_e}{1-p_e} = \frac{ \sin(\frac13[\pi - \th_e])}{\sin(\frac13 \th_e)}. 
\end{equation}
We consider bond percolation on $G$ with edge-probabilities $\bp=(p_e: e \in E)$.
Our main results are that (subject to two assumptions on $G$) these percolation
models are critical, and their critical exponents (\emph{at} the critical point)
are universal across the class of such models. More precisely,
the exponents $\rho$, $\eta$, $\delta$ and the $2j$-alternating arm
exponents $\rho_{2j}$ are constant across the class (assuming they exist). 
See \cite{Grimmett_Percolation} for an account of
the general theory of percolation.

\begin{figure}[htb]
 \centering
    \includegraphics[width=0.6\textwidth]{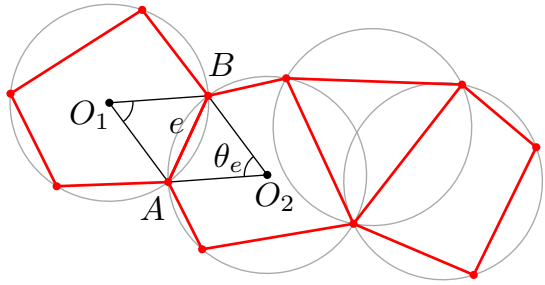}
  \caption{Part of an isoradial graph. Each face is inscribed in a circle of radius $1$. 
  With the edge $e$, we associate the angle $\theta_e$. 
This and later figures are best viewed in colour.}
  \label{fig:isoradial}
\end{figure}

Our methods
work only \emph{at} the critical point, and we are unable
to extend such universality to exponents \emph{near} the critical point, such as 
$\gamma$ and $\beta$. Neither have we any progress to report on the
\emph{existence} of critical exponents for these models.
Essentially the only model for which existence has been proved is site
percolation on the triangular lattice (see \cite{Smirnov-Werner}), and our 
numerical knowledge
of the exponents is based in this case on the proof by Smirnov \cite{Smirnov}
of the convergence in the scaling limit
to SLE$_6$ (see also the more recent works \cite{CamNew,WW_park_city}). 

The principal tool in the current work is the \stt, otherwise known amongst physicists
as the Yang--Baxter equation. This was discovered by Kennelly \cite{Ken} in 1899
in the context of electrical networks, and was adapted in 1944 by Onsager \cite{OnsI}
to the Ising model in conjunction with Kramers--Wannier duality. It is
a key element in the work of Baxter \cite{Baxter_book} on exactly solvable models
in statistical mechanics
(see \cite{McCoy_book,Perk-AY} for scientific and historical accounts).
Sykes and Essam \cite{Sykes_Essam} used the \stt\ 
to predict the critical surfaces of inhomogeneous bond percolation on
triangular and hexagonal lattices, and it is also a tool in the study of
the random-cluster model \cite{Grimmett_RCM}, and
of the dimer model \cite{Ken02}.

The \stt\ is especially well matched to isoradiality. The exchange of a star
with a triangle is equivalent to a local move in the associated rhombic tiling.
In the case of percolation, this move replaces one triple of edge-parameters by another triple
given according to \eqref{G100}, and these
are precisely the triples that satisfy the \stt.
This general observation is valid for several models of statistical physics,
including the Ising and Potts models and the random-cluster model
(see \cite{Ken02}).   Note, however, that the \stt\ contributes
also to the analysis of models on non-isoradial graphs
(see \cite{Baxter_book}).

The \stt\ is used here to move between models. 
Using a coupling of probability measures on different isoradial
graphs, one may show as in \cite{GM1,GM2} that the \stt\ preserves
open connections in the percolation model, and may therefore be used to transport
the so-called `\bxp' from one model to another. The \emph{\bxp} is the statement that
any rectangle in $\RR^2$ of given aspect-ratio is crossed by an open
path, with probability bounded away from $0$ uniformly in the size, 
position, and rotation of the box. This property has played an important
role in the theory of criticality and scaling in two-dimensional percolation,
as described in Kesten's book \cite{Kesten_book}, and more recently
in \cite{WW_park_city}. 

This is a continuation 
and extension of the work begun in \cite{GM1,GM2}. The main idea is
to use the \stt\ to transport the \bxp\ between graphs. By 
the  Russo--Seymour--Welsh (RSW) 
lemma of \cite{Russo,Seymour-Welsh}, homogeneous percolation on the square
lattice (with $p=\frac12$) has the \bxp, and this 
property may be transported, by repeated applications
of  \stt s, to the isoradial
graphs considered here.  Once we know that an isoradial graph and its 
(isoradial) dual graph have
the \bxp, criticality follows as in \cite[Props 4.1--4.2]{GM1}. The above universality
follows by related arguments.

We return briefly to the \stt\ and the Ising model. 
The Ising model differs from percolation and the random-cluster model in the following
significant regard. Whereas the Ising model may be transformed thus for any
triple of edge-interactions, the percolation and random-cluster models
may be transformed only when the interaction-triple satisfies a certain equation
connected to the critical surface in the translation-invariant lattice model 
(see \cite[Thm 1.1]{GM1}).
For this reason, the \stt\ enables a deeper understanding of the Ising model
than of its cousins. Its application
to the Ising model has been developed in a series of papers
by Baxter and others under the title `$Z$-invariant
Ising model' beginning with \cite{Bax86} and continued in later works.
We do not attempt a full bibliography here, but mention the
paper \cite{AuP07} by Au-Yang and Perk,  the recent papers
of Chelkak and Smirnov \cite{Chelkak-Smirnov2,Smirnov1}, as well
as the dimer analysis \cite{BdTB} of Boutillier and de Tili\`ere.
 
This paper is organized as follows. 
Isoradial graphs and the \bxp\ are summarised in
Section \ref{sec:def}, and our main results
concerning criticality and universality
are stated in Section \ref{sec:main}.
Section \ref{sec:isoradial} contains a fuller description of 
isoradiality and its connections to
rhombic tilings. 
The \stt\ is described in Section \ref{sec:stt}. 
The \bxp, Theorem \ref{main},   is then proved in two steps:
first in Section \ref{sec:proof1} for the special
case of isoradial square lattices, and then in Section \ref{sec:proof2} 
where it is explained how this case may be extended to
general isoradial graphs satisfying the given conditions.
Theorem \ref{main3}, concerning universality, is proved in
Section \ref{sec:arm-exps}.

The set of integers is denoted by $\ZZ$, of natural numbers by $\NN$, 
and of non-negative integers by $\ZZO$.

\section{Isoradiality and the \bxp}\label{sec:def}

\subsection{Isoradial graphs}\label{sec:isog}

Let $G=(V,E)$ be a planar graph embedded in the plane $\RR^2$
(it is assumed that the edges are embedded as straight-line segments,
with intersections only at vertices). 
It is called \emph{isoradial}
if, for every bounded face $F$ of $G$, the vertices of $F$ lie on a circle 
of (circum)radius $1$ with centre in the interior of $F$. 
In the absence of a contrary assumption,
we shall assume that isoradial graphs are infinite
with bounded faces. 
While isoradiality is a property of
the planar embedding of $G$ rather than of $G$ itself, we shall sometimes speak of
an isoradial graph.

Let $G=(V,E)$ be isoradial.
Each edge $e = \lan A, B\ran$ of $G$ lies in two faces, 
with circumcentres $O_1$ and $O_2$. 
Since the two circles have equal radii, the quadrilateral $AO_1BO_2$ is a rhombus. 
Therefore, the angles $AO_1B$ and $BO_2A$ are equal, and we write
$\th_e \in (0,\pi)$ for their common value. See Figure 
\ref{fig:isoradial}.

\begin{definition}\label{def:bac}
  Let $\eps>0$. The isoradial graph $G$ is said to have the \emph{\bac}
  \BAC$(\eps)$ if 
  \begin{equation}\label{G310}
    \th_e \in [\eps, \pi-\eps], \qquad e \in E.
  \end{equation}
  It is said to have, simply, the \emph{\bac} if it satisfies \BAC$(\eps)$ for
  some $\eps>0$.
\end{definition}

All isoradial graphs of this paper will be assumed to 
have the \bac. 
The area of the rhombus $AO_1BO_2$
equals $\sin \th_e$ and, under \BAC$(\eps)$,
\begin{equation}\label{rhombic-area}
  \sin\eps \le \text{Area}(AO_1BO_2)  \le 1.
\end{equation}

We  shall study isoradial graphs satisfying a further property called the `\sgp'.  
We defer a discussion of the \sgp\ until Section  \ref{sec:track}, 
since it necessitates a fuller account of isoradial graphs
and their relationship to rhombic tilings of the plane and the associated `track systems'
introduced by de Bruijn \cite{deB1,deB2}. 
Although one may construct isoradial graphs that do not have the \sgp,
the property is satisfied by many graphs arising from the rhombic tilings that we have encountered in 
the literature on plane tilings.

From the many examples of isoradial graphs with the \sgp,
we select for illustration the set of isoradial embeddings of the square lattice, 
and more generally isoradial embeddings of any connected periodic graph. 
It is not the automorphism group of the graph that is relevant here,
but rather the geometry of its track system. For example, isoradial graphs arising from
rhombic  Penrose tilings have the \sgp. 
For definitions and further discussion, see Section \ref{sec:isoradial}. 

\subsection{Percolation on isoradial graphs}
Bond percolation on a graph $G=(V,E)$ has sample space $\Om:=\{0,1\}^E$,
to which is assigned a product measure $\PP_\bp$ with $\bp=(p_e: e\in E) \in [0,1]^E$. 
When $G$ is isoradial, there is a canonical product measure, denoted $\PP_G$, associated with its embedding,
namely that with $p_e = p_{\th_e}$ where
$\th_e$ is given in the last section and illustrated in Figure \ref{fig:isoradial}, and
\begin{equation}\label{G101} 
\frac{p_\th}{1-p_\th} = \frac{ \sin(\frac13[\pi - \th])}{\sin(\frac13\th)}. 
\end{equation}
Note that $p_\th+p_{\pi-\th}=1$, and  that
$G$ has the \bac\ \BAC$(\eps)$ if and only if 
\begin{equation}\label{G226}
p_{\pi-\eps} \le p_e \le p_\eps, \qquad e \in E.
\end{equation}

The graph $G$ has a dual graph $G^*=(V^*,E^*)$.  Since
$G$ is isoradial, so is $G^*$. 
This fact is discussed in greater depth in
Section \ref{sec:isoradial}, but it is clear from Figure \ref{fig:isoradial}
that $G^*$ may be embedded in $\RR^2$ with vertices at the
centres of circumcircles, and edges
between circumcentres of abutting faces. 
Let $e^*\in E^*$ be the dual edge (in this embedding) crossing the primal
edge $e\in E$. Then  $\th_{e^*} = \pi-\th_e$, so that $p_e + p_{e^*} = 1$
by \eqref{G101}. In conclusion, the canonical measure $\PP_{G^*}$ is
dual to the primal measure $\PP_G$.  
By \eqref{G226}, 
\begin{equation}\label{G455}
\mbox{$G^*$ satisfies \BAC$(\eps)$ if and only if
$G$ satisfies \BAC$(\eps)$.}
\end{equation}
See, for example, \cite[Sect.\ 11.2]{Grimmett_Percolation} for an account of graphical duality.

Here is some basic notation. Let $G=(V,E)$ be a graph, not necessarily
isoradial or even planar, and let $\om \in \Om := \{0,1\}^E$.
An edge $e$ is called \emph{open} (or \emph{$\om$-open}) if $\om(e)=1$, and \emph{closed}
otherwise. A path of $G$ is called \emph{open} if all its edges are open.
For $u,v \in V$,
we say $u$ \emph{is connected to} $v$ (in $\omega$), 
written $u \lra v$ (or $u\xleftrightarrow {G,\omega} v$), 
if $G$ contains an open path from $u$ and $v$;
if they are not connected, we write $u \nxlra{\ \ G, \om} v$.
An \emph{open cluster} of $\om$ is a maximal set of pairwise-connected vertices. 
Let $C_v =\{u\in V: u \lra v\}$ denote the open cluster containing the vertex $v$, and write  
$v \lra \oo$ if $|C_v|=\oo$. 

For $\om\in\Om$ and $e \in E$, let $\om^*(e^*)=1-\om(e)$,
so that $e^*$ is open in the dual graph $G^*$ (written \emph{open$^*$})
if $e$ is closed in the primal graph $G$.

\subsection{The \bxp}\label{sec:thebxp}

Let $G=(V,E)$ be a countably infinite, connected graph embedded in the plane, and
let $\PP$ be a probability measure on $\Om:=\{0,1\}^E$.  The `\bxp' is concerned
with the probabilities of open crossings of domains in $\RR^2$. This
has proved to be a very useful property indeed for the study of infinite
open clusters in $G$; see, for example, \cite{Grimmett_Graphs,GM2,Kesten_book}.

A \emph{(planar) domain} $\sD$ is an open, simply connected subset of $\RR^2$ which,
for simplicity, we assume to be bounded by a Jordan curve $\pd\sD$.
Most domains of this paper are the interiors of polygons. 
Let $\sD$ be a domain, and let $A$, $B$, $C$, $D$ 
be distinct points on its boundary in anticlockwise order. 
Let $\om \in\Om$.
We say that $\sD$ has an open crossing from $DA$ to $BC$ if $G$ contains
an open path using only edges  intersecting $\sD$ which, 
when viewed as an arc in $\RR^2$,  intersects $\pd\sD$ 
exactly twice, once between 
$D$ and $A$  
(including the endpoints) and once between $B$ and $C$. 

A \emph{rectangular domain} is a set $\B = f((0, x) \times (0, y)) \subseteq \RR^2$, 
where $x,y >0$ and $f: \RR^2 \to \RR^2$ comprises a rotation and a translation.
The \emph{aspect-ratio} of this rectangle is $\max\{x/y,y/x\}$.
We say  $\B$ \emph{has open crossings} in a configuration $\om\in\Om$
if it has open crossings both from $f(\{0\} \times [0, y])$ to $f(\{x\} \times [0, y])$
and  from $f([0,x] \times \{0\})$ to $f([0,x] \times \{y\})$.

\begin{definition}
A probability measure $\PP$ on $\Om$ is said to have the \emph{\bxp} 
if, for  $\rho > 0$, there exist $l_0 = l_0(\rho) > 0$ and $\de=\delta(\rho) > 0$ 
such that, for all $l > l_0$ and all rectangular domains $\B$ with side-lengths
$l$ and $\rho l$,
\begin{equation}
  \PP(\B \text{ has open crossings}) \geq \delta. \label{isoradial_bxp} 
\end{equation}
An isoradial graph $G$ is said to possess the \bxp\ 
if $\PP_G$ possesses it. 
\end{definition}

In a standard application of the Harris--FKG inequality
(see \cite[Sect.\ 2.2]{Grimmett_Percolation}), it suffices for the \bxp\ 
to consider boxes with aspect-ratio $2$, and moreover only such boxes
with horizontal/vertical orientation (see, for example,
\cite[Prop.\ 3.2]{G-three} and \cite[Prop.\ 3.1]{GM1}). 
If \eqref{isoradial_bxp} holds for this restricted class of
boxes with $\rho=2$ and $\de=\de(2)$, 
we say that $G$ satisfies \BXP$(l_0,\de)$. All graphs considered here
are isoradial with circumradius $1$, and for such graphs
one may take $l_0=3$. We thus abbreviate \BXP$(3,\de)$ to  \BXP$(\de)$.

It was proved by Russo \cite{Russo} and Seymour--Welsh 
\cite{Seymour-Welsh} that
the isotropic embedding of the square lattice (with $p=\frac12$)
has the \bxp, and more generally
in \cite{GM1} that certain inhomogeneous embeddings of the square, triangular, and hexagonal
lattices have the property.

\section{Main results}\label{sec:main}

Let $\sG$ be the class of isoradial graphs with the \bac\ and the \sgp\ (we recall that
the \sgp\ is formulated in Section \ref{sec:track}).
The main technical result of this paper is the following. 
Criticality and universality will follow.

\begin{thm}\label{main}
  For $G \in \sG$, $\PP_G$  possesses the \bxp. 
\end{thm}

A more precise statement holds. In discussing the \sgp\ in Section \ref{sec:track},
we will introduce a more specific property denoted \SGP$(I)$ for $I\in\NN$.
We shall show that, for $\eps>0$ and $I\in \NN$, there exists $\de=\de(\eps,I)>0$ such that:
\begin{equation}\label{G2231}
\text{if $G$ satisfies \BAC$(\eps)$ and \SGP$(I)$, $\PP_G$ satisfies \BXP$(\de)$.}
\end{equation}

Let $G=(V,E)$ be a graph, and let $\PP$ be a product 
measure on $\{0,1\}^E$ with intensities $(p_e: e\in E)$.
For $\de \in \RR$, we 
write $\PP^\de$ for the percolation measure with intensities
$p_e^\de := (0\vee (p_e+\delta))\wedge 1$. 
[As usual, $x\vee y = \max\{x,y\}$ and $x \wedge y = \min\{x,y\}$.]
If $G$ is embedded in $\RR^2$, the \emph{radius} $\rad(C_v)$
of the open cluster at $v \in V$
is the supremum of $k \ge 0$ such that $C_v$ contains a vertex
outside the box $v + (-k,k)^2 \subseteq \RR^2$.  

The proofs of the following Theorems \ref{criticality} and \ref{main3}
rely heavily on the \bxp\ of Theorem \ref{main}. 

\begin{thm}[Criticality]\label{criticality}
  Let $G=(V,E) \in \sG$, and let $\nu >0$. 
  \begin{letlist}
  \item
    There exist $a,b,c,d > 0$ such that, for   $v\in V$,
    \begin{align*}
      ak^{-b} \le  \PP_G \bigl(\rad (C_v) \geq k\bigr) \leq ck^{-d}, 
      \qquad k \ge 1. 
    \end{align*}
  \item
    There exists, $\PP_G$-a.s.,  no infinite open cluster. 
  \item
    There exist $f,g>0$ such that, for   $v\in V$,
    $$
    \PP_{G}^{-\nu}(|C_v|\ge k) \le f e^{-gk}, \qquad k \ge 0.
    $$
  \item 
    There exists $h > 0$ such that, for  $v\in V$, 
    $$
    \PP_G^\nu(v \lra\oo)>h.
    $$
  \item 
    There exists, $\PP_G^\nu$-a.s.,  exactly one infinite open cluster. 
  \end{letlist}
\end{thm}

More precisely, if $G \in \sG$ satisfies \BXP$(\eps)$ and \SGP$(I)$,
the claims of the theorem hold with constants that depend only on $\eps$, $I$,
and not further on $G$. 
The proof of Theorem \ref{criticality} is summarised at the end of this section.

Turning to critical exponents and universality,
we write $f(t) \eqv g(t)$ as $t\to t_0 \in [0,\oo]$ 
if there exist strictly positive constants $A$, $B$ such that
\begin{equation}\label{asympt}
A g(t) \le f(t) \le Bg(t)
\end{equation}
in some neighbourhood of $t_0$ (or for all large $t$ in the case $t_0=\oo$).
For functions $f^u(t)$, $g^u(t)$ indexed by $u\in U$,
we say that $f^u \eqv g^u$ \emph{uniformly in} $u$ if \eqref{asympt}  holds with
constants $A$, $B$ not depending on $u$.
We write $f(t) \logeqv g(t)$ if
$\log f(t) / \log g(t) \to 1$, and $f^u \logeqv g^u$ \emph{uniformly in} $u$ if
the convergence is uniform in $u$. 

The critical exponents of interest here are those denoted 
conventionally as $\rho$, $\eta$, $\delta$,
and the alternating arm-exponents $\rho_{2j}$. 
We begin by defining the so-called \emph{arm-events}.
Let $B_n$ denote the box $[-n,n]^2$ of $\RR^2$, with
boundary $\pd B_n$.
For $N<n$, let $\Ann(N,n)$ be the \emph{annulus}
 $[-n,n]^2 \setminus (-N,N)^2$
with inner radius $N$ and outer radius $n$.
The \emph{inner} (\resp, \emph{outer}) \emph{boundary} of the annulus is $\pd B_N$ (\resp, $\pd B_n$).
For $u\in \RR^2$, write $\Ann^u(N,n)$ for the translate $\Ann(N,n)+u$.
A \emph{primal} (\resp, \emph{dual}) \emph{crossing} of $\Ann(N,n)$ 
is an open (\resp, open$^*$) path 
whose intersection with $\Ann(N,n)$ is an arc with 
an endpoint in each boundary of the annulus.
Primal crossings are said to have colour $1$, and dual 
crossings colour $0$.

Let $k \in\NN$.
A sequence $\sigma \in \{0,1\}^k$ is called a \emph{colour sequence}
of length $k$. For such $\sigma$, the arm-event
$\Arm_\sigma(N,n)$ is the event that there exist
$k$ vertex-disjoint crossings 
$\gamma_1, \dots, \g_i,\dots,\gamma_k$ of $\Ann(N,n)$ 
with colours $\sigma_i$ taken in anticlockwise order.
The corresponding event on the translated
annulus $\Ann^u(N,n)$ is denoted $\Arm_\sigma^u(N,n)$ and is said to be 
`centred at $u$'. As in \cite{GM2}, the value of $N$ is
largely immaterial to what follows, but  $N=N(\sigma)$
is taken sufficiently large
that the events $\Arm_\sigma(N,n)$ are non-empty for $n \ge N$.

A colour sequence $\sigma$ is called \emph{monochromatic} if either 
$\sigma = (1,1,\dots,1)$ or $\sigma = (0,0,\dots,0)$,
and \emph{bichromatic} otherwise.
It is called \emph{alternating} if it has even length 
and either $\sigma =(1,0,1,0, \dots)$ or  $\sigma =(0,1,0,1, \dots)$.
When $\sigma=(1)$, $\Arm_\sigma(N,n)$ is called the \emph{one-arm-event} 
and denoted $\Arm_1(N,n)$. When $\sigma$ is alternating with
length $k=2j$, the corresponding event is denoted $\Arm_{2j}(N,n)$. 

Let $G \in\sG$ be an isoradial graph with vertex-set $V$, and let $C_v$ be the
open cluster  of the vertex $v \in V$, under the canonical measure $\PP_G$.
We concentrate here on the following exponents given in terms of $\PP_G$, with limits
that are uniform in the choice of $v$:
\begin{letlist}
\item
volume exponent: $\PP_G(|C_v| = n) \logeqv n^{-1-1/\delta}$ as $n \to\oo$,
\item
connectivity exponent: $\PP_G(v \lra w) \logeqv |w-v|^{-\eta}$ as $|w-v|\to\oo$,
\item
one-arm exponent: $\PP_G[\Arm_{1}^v(N,n)] \logeqv n^{-\rho_{1}}$ as $n \to\oo$,
\item
$2j$-alternating-arms exponents: $\PP_G[\Arm_{\sigma}^v(N,n)] \logeqv n^{-\rho_{2j}}$
as $n\to\oo$,
for each alternating colour sequence $\sigma$ of length $2j$, with
$j \ge 1$.
\end{letlist}
It is believed that the above uniformly asymptotic relations hold for 
suitable exponent-values, and indeed with $\logeqv$ 
replaced by the stronger relation $\eqv$. 

The conventional one-arm exponent $\rho$ is given by $\rho=1/\rho_1$,
as in \cite[Sect.\ 9.1]{Grimmett_Percolation}. Parts (c) and (d) above
are parts of the following  more extensive conjecture.

\begin{conj}\label{def_arm_exp}
Let $G$ be an isoradial graph with the \bac, and
let $k \in \NN$ and $\sigma\in\{0,1\}^k$.
\begin{letlist}
\item 
There exists $\rho(\sigma,G) >0$ such that
\begin{align*}
  \PP_G[\Arm_{\sigma}^u(N,n)]  \logeqv n^{-\rho(\sigma,G)} 
  \qquad \text{ as } n \to \infty, 
\end{align*}
uniformly in $u\in\RR^2$.
\item The exponent  $\rho(\sigma,G)$ does not depend on the choice of $G\in\sG$.
\end{letlist}
\end{conj}

Essentially the only two-dimensional
percolation process for which the arm-exponents $\rho_\sigma$ are proved to exist 
(and, furthermore, many of their values known explicitly) 
is site percolation on the triangular lattice (see \cite{BN,Smirnov,Smirnov-Werner}).
In this special case (not belonging to the class of models
considered in this paper), the  
$\rho_\sigma$ are constant for all bichromatic colour sequences of given length
(see \cite{ADA}), and the monochromatic arm-exponents have been studied in \cite{BN}.

A critical exponent $\pi$ is said to \emph{exist} for a graph $G \in  \sG$ 
if the appropriate asymptotic relation holds. 
It is called $\sG$-\emph{invariant} if it exists for all $G \in \sG$ and
its value is independent of the choice of $G$.

Our universality theorem is presented next. 
Part (a) amounts to a verification of Conjecture
\ref{def_arm_exp}(b)
for isoradial graphs $G\in\sG$, and for colour
sequences which are either of length one or alternating. 

\begin{thm}[Universality]
\label{main3}
\mbox{}
\begin{letlist}
\item
Let $\pi \in\{\rho\}\cup\{\rho_{2j}: j \ge 1\}$.
If $\pi$ exists for some $G \in  \sG$, then it is $\sG$-invariant.
\item
If either $\rho$ or $\eta$ exists for some $G \in \sG$, then $\rho$, $\eta$, $\de$ are
$\sG$-invariant and satisfy the scaling relations $\eta\rho=2$ and $2\rho = \de+1$.
\end{letlist}
\end{thm}

Kesten showed in \cite{Kesten87} 
(see also \cite{Nolin}) that certain properties of a \emph{critical} percolation process 
imply properties of the \emph{near-critical} process, when the underlying graph has
a sufficiently rich automorphism group. In particular, knowledge of
certain critical exponents \emph{at criticality} implies knowledge of exponents 
\emph{away from criticality}. Only certain special isoradial graphs have sufficient
homogeneity for such arguments to hold without new ideas of substance.
Therefore, further discussion is omitted, and the reader is referred instead to 
\cite[Sect.\ 1.4]{GM2}. 

Finally, we make some comments on the proofs.
There are two principal steps in the proof of Theorem \ref{main}.  
Firstly, using a technique involving star--triangle transformations, 
the \bxp\ is transported from the homogeneous square lattice to 
an arbitrary isoradial embedding of the square lattice (with the \bac). 
Secondly, the \sgp\ is used to transport the \bxp\ 
to general isoradial graphs. 
This method may be used also to show the invariance
of certain arm exponents across the class of such isoradial graphs,
as in Theorem \ref{main3}. The basic approach is that of \cite{GM1,GM2}, but
the geometrical constructions used here
differ in substantial regards from those papers.

\begin{proof}[Proof of Theorem \ref{criticality}]
  By \eqref{G455}, \eqref{G456}, and Theorem \ref{main}, 
  both $\PP_G$ and $\PP_{G^*}$ have the \bxp.
  The claims then follow as in \cite[Props 4.1, 4.2]{GM1} (see also \cite[Remark 4.3]{GM1}), 
  and the details are omitted. It suffices to check that the conditions of the Remark 
  hold under the \bac. The constants in the theorem may be tracked through the proofs,
  and are found to depend only on the values of $\eps$ and $I$.
\end{proof}

\section{Isoradial graphs and rhombic tilings}\label{sec:isoradial}

\subsection{Rhombic tilings}

A \emph{rhombic tiling} is
a planar graph embedded in $\RR^2$ such that every face
is a rhombus of side-length $1$. 
Rhombic tilings have featured prominently in the
theory of planar tilings, both periodic and aperiodic. 
A famous example is the aperiodic rhombic tiling
of Penrose \cite{Pen78}, and the generalizations of
de Bruijn \cite{deB1,deB2} and others. The reader is referred to \cite{GS,Sen}
for general accounts of the theory of tiling.  

There is a two--one correspondence 
between isoradial graphs and
rhombic tilings of the plane, which we review next. 
Let $G=(V,E)$ be an isoradial graph. The \emph{diamond graph} $G^\di$
is defined as follows. The vertex-set of
$G^\di$ is $V^\di :=V \cup C$ where $C$ is the set of circumcentres of faces of $G$;
elements of $V$ shall be  called \emph{primal} vertices, and elements of $C$ \emph{dual} vertices.
Edges are placed between pairs $v\in V$,  $c\in C$
if and only if
$c$ is the centre of a circumcircle of a face containing $v$. 
Thus $G^\di$ is bipartite. Since $G$ is isoradial, the diamond graph
$G^\di$ is a rhombic tiling, and is illustrated in Figure \ref{fig:iso_sq}.

\begin{figure}[htb]
 \centering
    \includegraphics[width=0.6\textwidth]{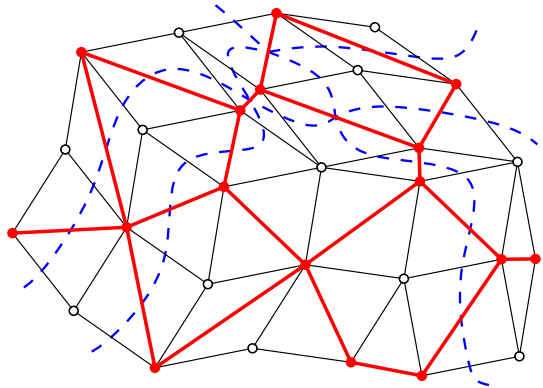}
  \caption{The isoradial graph $G$ is drawn in red,
and the associated diamond graph $G^\di$  in black. 
    The primal vertices of $G^\di$ are those of $G$; 
    the dual vertices are centres of faces of $G$.
    A track is a doubly infinite sequence of adjacent rhombi
    sharing a common vector, and 
    may be represented by a path, drawn in blue.
    Two tracks meet in an edge of $G$ lying in some face of $G^\di$.}
  \label{fig:iso_sq}
\end{figure}

From the diamond graph $G^\di$ may be found both $G$ and its
planar dual $G^*$. Write $V_1$ and $V_2$ for the two sets of vertices in the bipartite $G^\di$.
For $i=1,2$, let $G_i$ be the graph with vertex-set $V_i$,
two points of which are joined  by an edge if and only if they lie in the same
face of $G^\di$. One of the graphs $G_1$, $G_2$ is $G$ and the other
is its dual $G^*$.
It follows in particular that $G^*$ is isoradial. Let $e\in E$
and let $e^*$ denote its dual edge.
The pair $e$, $e^*$ are diagonals of the same rhombus of $G^\di$
and are thus perpendicular.

The above construction may be applied to any rhombic tiling $T$ to obtain
a primal/dual pair of isoradial graphs.

\subsection{Track systems}\label{sec:track}

Rhombic tilings have attracted much interest,
especially since the discovery by
Penrose \cite{Pen74,Pen78} of his 
celebrated aperiodic tiling.
Penrose's rhombic tiling
was elaborated by de Bruijn \cite{deB1,deB2}, who developed the following
representation in terms of `ribbons' or `(train) tracks'.
Let $G=(V,E)$ be isoradial. An edge $e_0$ of $G^\di$ belongs to two rhombi
$r_0$, $r_1$ of $G^\di$. 
Write $e_{-1}$ (\resp, $e_1$) for  the  edge of $r_0$ (\resp, $r_1$) opposite  $e_0$,
so that $e_{-1},e_0,e_1$ are parallel unit-line-segments. 
The edge $e_{-1}$ (\resp, $e_1$) belongs to a further rhombus $r_{-1}$ (\resp, $r_2$)
that is distinct from $r_0$ (\resp, $r_1$).
By iteration of this procedure, we obtain a doubly-infinite sequence of 
rhombi $(r_i :i \in \ZZ)$ such that the intersections $(r_i \cap r_{i+1}: i \in \ZZ)$
are distinct, parallel unit-line-segments. 
We call such a sequence a \emph{(train) track}. We write $\sT(G)$ for the
set of tracks of $G$, and note that $\sT(G) = \sT(G^*)$.
The track construction is illustrated in Figure \ref{fig:iso_sq}.

A track $(r_i: i\in \ZZ)$  is sometimes illustrated as an arc joining the midpoints of
the line-segments $r_i \cap r_{i+1}$ in sequence. The set 
$\sT$ may therefore be represented as a family of doubly-infinite
arcs which, taken together with the intersections 
of arcs, defines a graph. We shall denote this graph by $\sT$ also. 
A vertex $v$ of $G^\di$ is said to be \emph{adjacent} to a track $(r_i: i\in \ZZ)$
if it is a vertex of one of the rhombi $r_i$. 
 
It was pointed out by de Bruijn, and is easily checked, that
the rhombi in  a track are distinct. Furthermore, two distinct tracks
may have no more than one rhombus in common.
Since each rhombus belongs to exactly two tracks, 
it is the unique intersection of these two tracks. 

Kenyon and Schlenker \cite{KenS} have showed a converse theorem. 
Let $Q$ be an infinite planar graph embedded in the plane with the
property that every face has four sides.
One may define the tracks of $Q$ by an adaptation of the above definition: 
a track exits a face
across the edge opposite to its entry.
Then $Q$ may be deformed continuously into a rhombic tiling if and only if
(i) no track intersects itself, and (ii) no two tracks intersect more
than once.

A  track $t$ is said to be \emph{oriented} if it is endowed with
a direction. As an oriented track
$t$ is followed in its given direction, it crosses sides of rhombi which are parallel.
Viewed as vectors from right to left, these sides constitute a unit vector $\tau(t)$ of $\RR^2$
called the \emph{transverse vector} of $t$. The transverse vector makes an angle
with the $x$-axis called the \emph{transverse angle} of $t$, with value
in the interval $[0,2\pi)$.

Let $I \in \NN$. We say that an isoradial graph $G$ has the \emph{\sgp} \SGP$(I)$ if 
its track-set $\sT$ may be partitioned into three sets $\sT =S \cup T_1 \cup T_2$ 
satisfying the following.
\begin{letlist}
\item For $k=1,2$, $T_k$ is a set $(t_k^i: i \in \ZZ)$ of distinct 
  non-intersecting tracks indexed by $\ZZ$.
\item For $k=1,2$ and $s \in \sT\setminus T_k$, every track of $T_k$ intersects $s$, 
and these intersections occur in their lexicographic order. 
\item For $k=1,2$, $i\in \ZZ$, and $s \in  T_{3-k}$,
  the number of track-intersections on $s$ between its intersections
  with $t_k^i$ and $t_k^{i+1}$ is strictly less than $I$. 
\end{letlist}
Two tracks belonging to the same $T_k$ are said to be \emph{\para}.
We refer to $T_1\cup T_2$ as a \emph{square grid} of $G$,
assumed implicitly to satisfy (c) above.
A square grid is a subset of tracks with the topology of the square lattice
(and satisfying (c)). 

Since the \sgp\ pertains to the diamond graph $G^\di$ rather than to $G$ itself,
\begin{equation}\label{G456}
\mbox{$G$ satisfies \SGP$(I)$ if and only if $G^*$ satisfies \SGP$(I)$.}
\end{equation}

An isoradial graph $G$ is said to have the \sgp\  (\SGP) if it satisfies \SGP$(I)$ for some
$I\in\NN$. As before, $\sG$ denotes the set of all isoradial graphs with
the \bac\ and the \sgp.  More specifically, we write $\sG(\eps,I)$
for the set of $G$ satisfying \BAC$(\eps)$ and \SGP$(I)$.

Let $G \in \sG$ have square grid $T_1 \cup T_2$. It may be seen by
the \bac\ that, for $k=1,2$, every 
$x \in \RR^2$ lies either in some track of $T_k$ or in the region of $\RR^2$ `between' two
consecutive elements of $T_k$.

\subsection{Examples}\label{sec:ex}

Here are three families of isoradial graphs with the \sgp, and one without.

\subsubsection{Isoradial square lattices}
An isoradial embedding of the square lattice is called an \emph{isoradial square lattice}
The track-system of such a graph is simply a square grid, and \emph{vice versa}.

\subsubsection{Periodic graphs}

A planar graph $H$, embedded in $\RR^2$, is said to be \emph{periodic} if there exist
distinct non-zero vectors $\tau_1,\tau_2\in\RR^2$ such that 
$H$ is invariant under shifts by either $\tau_i$.
Let $G$ be an isoradial embedding of a periodic 
connected graph $H$
(the embedding itself need not be periodic).
The track system $\sT$ of $G$ (viewed as a set of arcs) is determined by the structure of $H$. 
Since $H$ is periodic, so is $\sT$ (viewed as a graph). 
Therefore, $\sT$ may be embedded homeomorphically
into $\RR^2$ in a periodic manner. After re-scaling,
we may assume that $\sT$ is invariant under any unit shift of $\RR^2$ in the 
direction of a coordinate vector. 
In fact, $\sT$ may be thought of as the lifting to the universal cover of a track-system
on a torus.

As observed in \cite[Sect.\ 5.2]{KenS}, any
oriented track $t$ has an asymptotic 
angle $\a(t) \in S^1$, and in addition the reversed
track has direction $\pi+\a(t)$.
Let $t \in \sT$, viewed as a subset of $\RR^2$. 
There exists $(a,b)\in\ZZ^2$, $(a,b) \ne (0,0)$, such that $t$ is invariant under the shift $\tau_{a,b}:
z\mapsto z+(a,b)$.  We have that $\tan \a(t) =b/a$. By periodicity,
the set of all angles (modulo $\pi$) of $\sT$ is finite, 
and we write it as $\{\a_1,\a_2,\dots,\a_m\}$
with $m \ge 1$.

Let $T_k$ be the set of tracks with asymptotic angle
(modulo $\pi$)  $\a_k$. By periodicity,
each $T_k$ is a set of tracks indexed by $\ZZ$, and may be ordered according to
their crossings of the line with polar coordinates $\th = \th_0$ with $\th_0 \ne \a_k$ for all $k$.
Since tracks
$t_k\in T_k$, $t_l \in T_l$ (with $k \ne l$) have different asymptotic angles, they must intersect.

It remains to show that any $t,t' \in T_k$ do not intersect 
(whence, in particular, $m \ge 2$). 
Suppose the converse,
that there exist $k\in\{1,2,\dots,m\}$ and $t,t'\in T_k$
such that $t$ and $t'$ intersect at some point $J\in\RR^2$. 
Since $t$ and $t'$ have the same 
angle $\a_k$, there exists $(a,b)\in \ZZ^2$ such that
$t$ and $t'$ are invariant under $\tau_{a,b}$. 
Therefore, they intersect at $J+n(a,b)$
for all $n \in \ZZ$, in contradiction of the fact that 
they may have at most one intersection.

For any distinct pair $T_k$, $T_l$,
part (c) of the \sgp\ holds by periodicity.

We have proved not only that $G$ has the \sgp, but the stronger fact that
its track-set may be partitioned into $m$ classes of parallel tracks.

\subsubsection{Rhombic tilings via multigrids}

The following `multigrid' construction was introduced and studied 
by de Bruijn \cite{deB1,deB2,deB3}.
A \emph{grid} is a set of parallel lines in $\RR^2$ with some common perpendicular 
unit-vector $v$.
A \emph{multigrid} is a family of grids with pairwise non-parallel perpendiculars. Suppose there
are $m \ge 2$ grids, with perpendiculars $v_1, v_2, \dots, v_m$. The  $k$th grid is given
in terms of a set $C_k=\{c_k^i: i \in \ZZ\}$ of reals, 
specifically as the set of all $z \in \RR^2$
with $z \cdot v_k = c^i_k$ as $i$ ranges over $\ZZ$.  It is assumed that the $c_k^i$ are
strictly increasing in $i$, with $c_k^i/i \to 1$ as $i \to \pm\oo$.

With the lines of the $k$th grid, duly oriented,
we associate a unit vector $w_k$. It is explained in \cite{deB3}
how, under certain conditions on the $C_k$, $v_k$, $w_k$, one may `dualize'
the multigrid to obtain a  rhombic tiling of $\RR^2$. The track-set of the ensuing tiling
is a homeomorphism of the multigrid with transverse vectors $w_k$. 
Under the additional assumption that the
differences
$|c_k^{i+1} - c_k^i|$ are uniformly bounded away from $0$ and $\oo$, all such tilings have both 
the \bac\ and the \sgp. The results of this paper apply to the associated isoradial graphs.

Penrose's rhombic tiling may be obtained thus with $m=5$, 
the $v_k$ being vectors forming a regular pentagon, with $w_k = v_k$,
and $C_k = \{i+\g_k: i \in \ZZ\}$ with an appropriate vector $(\g_k)$. 
Other choices of the parameters yield a broader class
of aperiodic rhombic tilings of the plane. See \cite{deB1,deB2}.
Percolation on Penrose tilings has been considered in \cite{Hof}.

\subsubsection{A track-system with no square grid}

Figure \ref{fig:nosg} is an illustration of a track-system without
the \sgp. 

\begin{figure}[htb]
 \centering
    \includegraphics[width=0.4\textwidth]{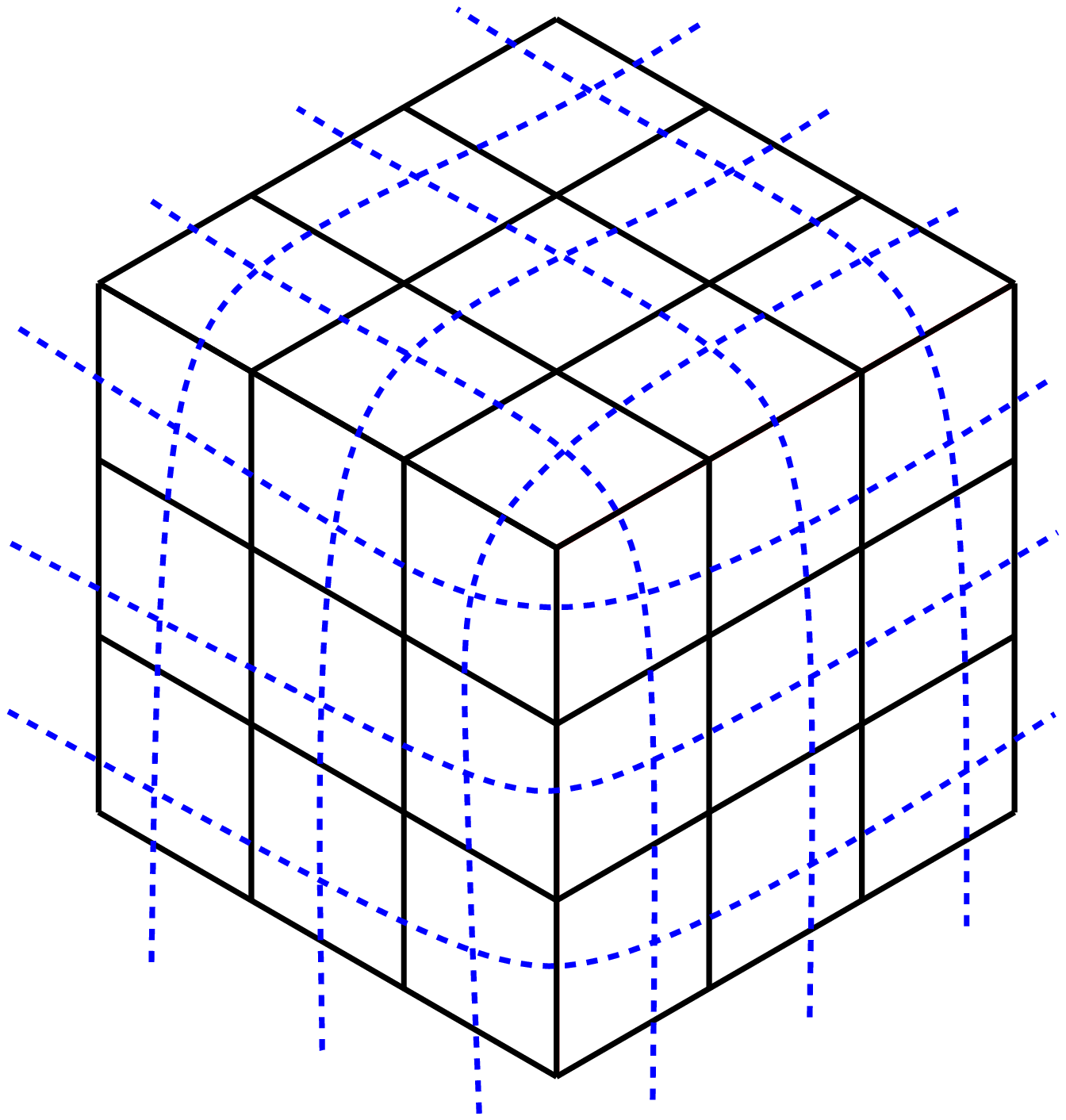}
\raisebox{12pt}{\includegraphics[width=0.4\textwidth]{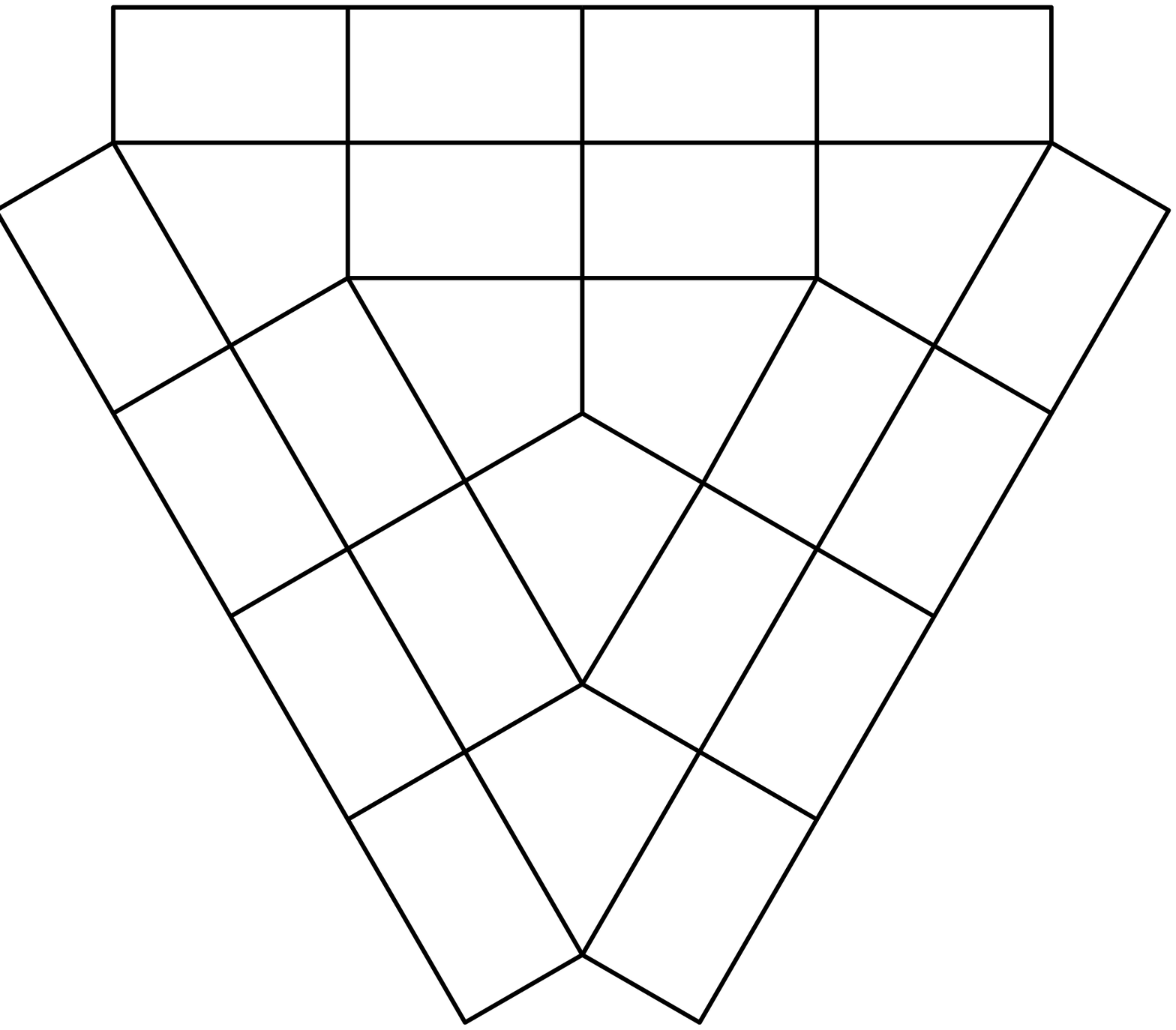}}
  \caption{Part of a rhombic tiling without
the \sgp, and one of the two corresponding isoradial graphs.}
  \label{fig:nosg}
\end{figure}

\subsection{Equivalence of metrics}\label{sec:metrics}
Let $G$ be an isoradial graph.
It will be convenient to use both the Euclidean metric $|\cdot|$ 
and the graph-metric  $d^\di$ on $G^\di$.
For $n \in N$ and $u \in G^\di$ we write $\La_u^\di(n)$ 
for the ball of $d^\di$-radius $n$ centred at $u$:
\begin{align*}
  \La_u^\di(n) =  \{ v \in G^\di : d^\di (u,v) \leq n\}.
\end{align*}

\begin{prop}\label{distance_equiv}
Let $\eps>0$. There exists $\cd=\cd(\eps)>0$ such that, for any
isoradial graph $G=(V,E)$ satisfying \BAC$(\eps)$,
  \begin{align}\label{di_euclid}
    \cd^{-1} |v-v'| \leq  d^\di (v,v') \leq \cd |v-v'| , \qquad  v,v' \in G^\di. 
  \end{align}
\end{prop}

\begin{proof}
  Let $u$, $v$ be distinct vertices of $G^\di$. 
  Since each edge of $G^\di$ has length $1$, 
  $ d^\di(u,v) \geq |u-v|$.
  Conversely, let $S_{uv}$ be the set of all faces of $G^\di$ (viewed as closed sets of $\RR^2$) 
  that intersect the straight-line segment $uv$ of $\RR^2$ joining $u$ to $v$. 
  Since the diameter of any such face is less than $2$,
  every point of the union of $S_{uv}$ is within Euclidean distance $2$ of $uv$.
  By \BAC$(\eps)$ and \eqref{rhombic-area}, 
  every face has area at least $\sin \eps$, 
  and  therefore
  $|S_{uv}| \le 4(|u-v| + 4)/ \sin \eps$.
  Similarly, there exists
  $\de=\de(\eps)>0$ such that $|u-v| \ge \de$.
  The edge-set of elements of $S_{uv}$ contains a path of 
  edges of $G^\di$ from $u$ to $v$, whence
  $$ 
  d^\di(u,v) \le \frac{8}{\sin \eps} (|u-v|+4) \le \frac{8(\de+4)}{\de \sin \eps} |u-v|,
  $$
  as required.
\end{proof}

\subsection{The \bxp\ for graphs in $\sG$}\label{sec:bxp_G}

This section begins with a definition of the rectangular 
domains of an isoradial graph $G \in \sG$,
using the topology of its square grid.

Let  $(t, t')$ be an ordered pair of non-intersecting tracks of  $G$.
A point $x \in \RR^2$ is said to be `strictly between' $t$ and $t'$ if, with these tracks
viewed as arcs of $\RR^2$, there exists an unbounded path of $\RR^2$ from $x$ that intersects
$t$ but not $t'$, and \emph{vice versa}.
A face $F$ of $G^\di$ is said to be \emph{between} $t$ and $t'$ if: either $F$ is
a rhombus of $t$, or every point of $F$ is strictly between $t$ and $t'$.
Note that this usage of `between' is not symmetric: 
there are faces  between $t$ and $t'$ that are not
between $t'$ and $t$.
A vertex or edge of $G^\di$ is said to be `between' $t$ and $t'$ 
if it belongs to some face between $t$ and $t'$. 
The \emph{domain} between $t$ and $t'$ 
is the union of the (closed) faces between $t$ and $t'$.
It is useful to think of a domain as either
a subgraph of $G^\di$, or (unlike the domains
of Section \ref{sec:thebxp}) as a closed region of $\RR^2$. 

Suppose $G \in \sG$ has a square grid $S \cup T$, 
with  $S=(s_j: j\in\ZZ)$ and $T=(t_i : i\in\ZZ)$. 
We call tracks in $S$ (\resp, $T$)  \emph{horizontal} (\resp,  \emph{vertical}). 
For $i_1, i_2, j_1, j_2\in \ZZ$ we define 
$\Dom = \Dom(t_{i_1}, t_{i_2}; s_{j_1}, s_{j_2})$ to be the intersection of 
the domains between $t_{i_1}$ and $t_{i_2}$ and between $s_{j_1}$ and $s_{j_2}$.

We say that $\Dom$ is crossed horizontally 
if $G$ contains an open path $\pi$ such that:
(i) every edge of $\pi$ lies in $\Dom$, and
(ii) the first edge crosses $t_{i_1}$ and the last vertex is adjacent to $t_{i_2}$.
Write  $\Ch(\Dom) = \Ch(t_{i_1}, t_{i_2}; s_{j_1}, s_{j_2})$ 
for the event that $\Dom$ is crossed horizontally, with a similar
definition of the vertical-crossing event $\Cv(\Dom)$. 
See Figure \ref{fig:domain} for an illustration of the above notions. 

\begin{figure}[htb]
  \begin{center}
    \cpsfrag{si}{$s_{j_1}$}
    \cpsfrag{sii}{$s_{j_2}$}
    \cpsfrag{ti}{$t_{i_1}$}
    \cpsfrag{tii}{$t_{i_2}$}
    \cpsfrag{D}{$\Dom$}
    \includegraphics[width=0.6\textwidth]{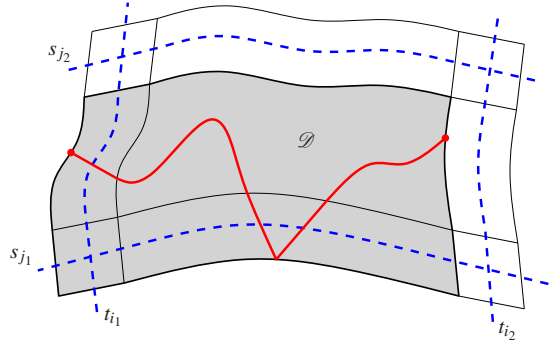}
  \end{center}
  \caption{The shaded domain $\Dom = \Dom(t_{i_1}, t_{i_2}; s_{j_1}, s_{j_2})$ 
    is crossed horizontally.}
  \label{fig:domain}
\end{figure}
 
The purpose of the following proposition is to restate 
the \bxp\ in terms of the 
geometry of the square grid.

\begin{prop}\label{grid_bxp}  
  Let $\eps>0$, $I\in \NN$, and let
  $G \in \sG(\eps,I)$.
  The graph $G$ has the \bxp\ if and only if there exists $\delta > 0$ 
  such that, for  $N\in \NN$ and $i,j \in \ZZ$,
  \begin{equation}\label{graph-theoretical_bxp}
    \PP_G\bigl [\Ch(t_i, t_{i + 2N}; s_j, s_{j + N})\bigr],
    \PP_G\bigl[\Cv(t_i,t_{i + N}; s_j, s_{j + 2N})\bigr]\geq \delta.
  \end{equation}
  Moreover, if \eqref{graph-theoretical_bxp} holds, then
  $G$ satisfies \BXP$(\de')$ with $\de'$ depending
  on $\de$, $\eps$, $I$ and not further on $G$.
\end{prop}

\begin{proof}
  We prove only the final sentence of the proposition.
  The converse (that the \bxp\ implies \eqref{graph-theoretical_bxp} for some $\de>0$)
  holds by similar arguments, and will not be used this paper.  
  Let $G\in\sG(\eps,I)$, and  
  assume \eqref{graph-theoretical_bxp} with $\de>0$.

  Let $N \in \NN$. 
  For $i, j \in \ZZ$, the \emph{cell} $C_{i,j}$ is defined to be the
  domain\linebreak
$\Dom(t_{iN}, t_{(i+1)N}; s_{jN}, s_{(j+1)N})$.
  The cells have disjoint interiors and cover the plane. 
  Two distinct cells $C=C_{i,j}$, $C'=C_{k,l}$ are said to be \emph{adjacent} if 
  $(i,j)$ and $(k,l)$ are adjacent  vertices  of the square lattice,
  in which case we  write $C \sim C'$. More specifically, we write
  $C \simh C'$ (\resp, $C \simv C'$) if $|i-k|=1$ (\resp,  $|j-l|=1$). 
  With the adjacency relation $\sim$, 
  the graph having the set of cells as vertex-set 
  is isomorphic to the square lattice.

  Each cell has perimeter at most $4IN$, 
  and therefore diameter not exceeding $2IN$. 
  A cell contains at least $N^2$ faces of $G^\di$, 
  and thus (by \eqref{rhombic-area}) 
  has total area at least $N^2 \sin \eps$.
  
  For $\mu \in \NN$ with $\mu\geq 2I$, let 
  $u=(-\mu N, 0)$ and $v = (\mu N, 0)$ viewed as points  in the plane. 
  Let $S^N_{uv}$ be the set of cells that intersect the straight-line segment $uv$ with endpoints $u$, $v$,
and let $U^N_{uv}$ be the union of such cells.
  Let $R$ be the tube $uv + [-2IN,2IN]^2$. 
  Thus $R$ has area $8IN(\mu N + 2IN)$, and $U^N_{uv} \subseteq R$.
  Since each cell has area at least $N^2\sin\eps$,
  the cardinality of $S^N_{uv}$ satisfies
  \begin{equation}\label{G322}
    |S^N_{uv}| \le \frac{8IN(\mu N+2IN)}{N^2\sin\eps} = \frac{8I(\mu+2I)}{\sin\eps}.
  \end{equation}
  
  \begin{figure}[htb]
    \begin{center}
      \includegraphics[width=0.8\textwidth]{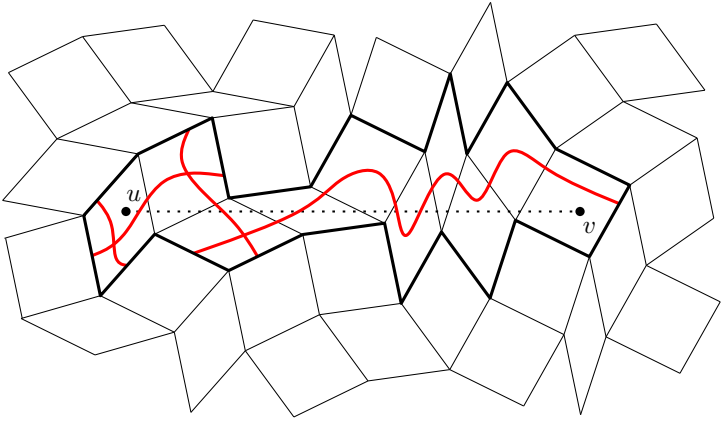}
    \end{center}
    \caption{The region $U^N_{uv}$ is outlined in bold, and  
      contains a chain of cells joining $u$ to $v$. 
      The events $H_k$ are drawn explicitly for the first two contiguous pairs of cells.}
    \label{fig:cell_chain}
  \end{figure}

  There exists a chain of cells $C_1, \ldots, C_K \in S^N_{uv}$ such that 
  $u \in C_1$, $v \in C_K$ and $C_k \sim C_{k+1}$ for $k=1,2,\dots,K-1$; 
  see Figure \ref{fig:cell_chain}.  
  Let $k\in\{1,2,\dots,K-1\}$, and assume $C_k \simh C_{k+1}$.
  Let $H_k$ be the event that $C_k$ and $C_{k+1}$ are crossed vertically,
  and $C_k \cup C_{k+1}$ is crossed horizontally.
  A similar definition
  holds when $C_k \simv C_{k+1}$, with vertical and horizontal interchanged.
  By \eqref{graph-theoretical_bxp} and the Harris--FKG inequality, 
  $\PP_G(H_k) \geq \de^3$.
  
  By the Harris--FKG inequality,  the fact that $K \le |S^N_{u,v}|$, and \eqref{G322}, 
  \begin{equation}\label{G324}
    \PP_G \left(\bigcap_{k = 1}^{K-1}H_k \right) \geq \delta^{3K} 
    \ge \de^{24I(\mu+2I)/\sin\eps}.
  \end{equation}
  If the event on the left side occurs, 
  the rectangle 
  $$
  S_{\mu,N} := \bigl[-(\mu - 2I)N, (\mu - 2I) N \bigr]\times [-2IN,2IN]
  $$ 
  of $\RR^2$ is crossed horizontally.
  
  Let $R_k=[-k,k] \times[-\frac12k, \frac12 k]$ where $k \ge 8I$.
  Pick $N$ such that $4IN \le k \le 8IN$, so that $R_k$ is `higher'
  and `shorter' than $S_{10 I,N}$. By \eqref{G324} with $\mu=10 I$,
  \begin{equation}\label{bxp_equiv_bound}
    \PP_G \left( R_k  \text{ is crossed horizontally}\right) 
    \ge \de'',
  \end{equation}
  where
  $\de'' = \de^{288I^2/\sin\eps}$.
  Smaller values of $k$ are handled by adjusting $\de''$ accordingly.
  
  The same argument is valid for translates and rotations of the 
  line-segment $uv$, and the proof is complete.
\end{proof}

\subsection{Isoradial square lattices}\label{sec:iso-sl}

An \emph{isoradial square lattice} is  an isoradial embedding of the square lattice $\ZZ^2$.
Isoradial square lattices, and only these graphs, have a square grid as track-system.

Let $G$ be an isoradial square lattice.
The diamond graph $G^\di$  possesses two families of \para\ tracks, 
namely the 
horizontal tracks $(s_j: j \in \ZZ)$ and the vertical tracks $(t_i: i \in \ZZ)$.
The graph $G^\di$, and hence the pair $(G,G^*)$ also, 
may be characterized in terms of two vectors of angles linked
to the transverse vectors. First, we orient $s_0$ in  an arbitrary way (interpreted as
`rightwards'). As we proceed
in the given direction along $s_0$, the crossing tracks $t_i$ are numbered in increasing sequence,
and are oriented from right to left (interpreted as `upwards'). 
Similarly, as we proceed along $t_0$, the crossing tracks
$s_j$ are numbered in increasing sequence and oriented from left to right. 
Using the notation of Section \ref{sec:track}, the transverse vector 
$\tau(s_j)$ has some transverse angle $\b_j$, and
similarly $\tau(t_i)$ has some transverse angle  $\g_i$. 
Rather than working with the $\g_i$, we work instead with $\a_i := \g_i-\pi$
as illustrated in  Figure \ref{fig:isoradial_square}. 
Write $\balpha=(\a_i: i\in\ZZ)$ and $\bbeta=(\b_j: j\in\ZZ)$, and note
that $\a_i \in[-\pi,\pi)$, $\b_j \in [0,2\pi)$.
We will generally assume that $G$ is rotated in such a way that $\alpha_0 =0$, 
so that $\b_j \in [0, \pi]$ and $\b_j-\pi \leq \a_i \leq \b_j$
for $i,j \in \ZZ$.  

The vertex of $G^\di$ adjacent to the four tracks 
$t_{i-1}$, $t_{i}$, $s_{j-1}$, $s_{j}$ is denoted $v_{i,j}$.
If not otherwise stated, 
we shall assume that the tracks are labelled in such a way that 
the vertex $v_{0,0}$ is a \emph{primal} vertex of $G^\di$. 

\begin{figure}[htb]
  \centering
  \includegraphics[width=0.6\textwidth]{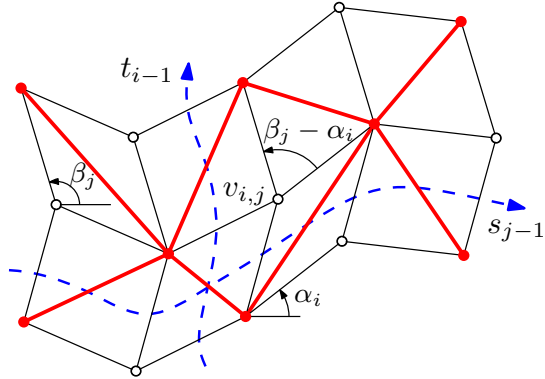}
  \caption{An isoradial square lattice (in red) with the associated diamond graph.
    The diamond graph is isomorphic to $\ZZ^2$, 
    and its embedding is characterized by two sequences  
    $\balpha$, $\bbeta$ of angles.}
  \label{fig:isoradial_square}
\end{figure}

Tracks $t_i$, $s_j$ intersect in a rhombus of $G^\di$ with 
sides $\tau(t_i)$, $\tau(s_j)$, $-\tau(t_i)$, $-\tau(s_j)$ in clockwise order,
and thus its internal angles are $\b_j-\a_i$ and $\pi-(\b_j-\a_i)$. 
Thus, $G$ satisfies the \bac\ \BAC$(\eps)$ if and only if
\begin{equation}\label{bounded_angles}
  \b_j - \a_i \in [\eps , \pi - \eps], \qquad  i,j \in \ZZ. 
\end{equation}
Conversely, for two vectors $\balpha$, $\bbeta$ satisfying \eqref{bounded_angles},
we may construct  the diamond graph denoted
$G_{\balpha, \bbeta}^\di$ as in Figure \ref{fig:isoradial_square}. 
This gives rise to an isoradial square lattice denoted $G_{\balpha, \bbeta}$ 
(and its dual)  satisfying \BAC$(\eps)$. 
We write $\PP_{\balpha, \bbeta}$ 
for the canonical measure of $G_{\balpha, \bbeta}$.

We introduce now some notation to be used later. 
For a set $W$ of vertices of $G^\di$, we define the  
\emph{height} of $W$ by
$$
h(W) = \sup \bigl\{j : \exists i \text{ with } v_{i,j} \in W\bigr\}.
$$
This definition extends in an obvious way to sets of edges.

In Section \ref{sec:sttiso} is described an operation of so-called 
`track-exchange' on isoradial square lattices. 
This introduces a potential for confusion between the \emph{label}
and the \emph{level} of  a track. 
In the $G_{\balpha,\bbeta}$ above, we say that $s_j$ is (initially) 
at \emph{level} $j$. 
The level of $s_j$ may change under track-exchange, but $v_{i,j}$
shall always refer to the vertex between levels $j-1$ and $j$ in
the new graph.

Due to this potential confusion, 
we may use a different notation for domains in square lattices
than for general graphs. 
For $M_1,M_2,N_1,N_2\in \ZZ$ with $M_1 \le M_2$, $N_1 \le N_2$,
let $B(M_1,M_2;N_1,N_2)$ be the subgraph of $G$ induced by
the subset of vertices lying in
$\{v_{i,j} : M_1 \leq i \leq M_2,\ N_1 \leq j \leq N_2\}$.
For $M, N \in \NN$, we use the abbreviated notation
$B(M, N)=B(-M,M;0,N)$. 
A \emph{horizontal crossing} of $B=B(M_1,M_2;N_1,N_2)$ is an open path of $B$ 
linking some vertex $v_{M_1, n_1}$ to some vertex $v_{M_2, n_2}$;
a \emph{vertical crossing} links some $v_{m_1,N_1}$ to some $v_{m_2,N_2}$.
We write $\Ch[B]$ (\resp, $\Cv[B]$) for the event that 
a box $B$ contains a horizontal (\resp, vertical) crossing. 
For a vertex $v=v_{i,j}$ of $G$, we write 
$B^{v}$ for the translate $\{v_{r,s}: v_{r-i,s-j} \in B\}$.

When applied to $G$, we have that  
\[ 
B(M_1,M_2;N_1,N_2) = \Dom(t_{M_1},t_{M_2};s_{N_1},s_{N_2}), 
\]
since $s_{N_1}$ and $s_{N_2}$
are the tracks at levels $N_1$ and $N_2$ \resp. 
As mentioned before, the latter will not always be the case.
Use of the notation $B$ emphasizes that domains are 
defined in terms of tracks at specific levels,
rather than of tracks with specific labels. 

The following lemma will be used in Section \ref{sec:proof1}.
\begin{lemma}\label{alternate_bxp}
  Let $G=(V,E)$ be an isoradial square lattice satisfying the \bac\ \BAC$(\eps)$ and the following.
  \begin{letlist}
  \item For $\rho \ge 1$, there exists $\eta(\rho)>0$ such that 
    \begin{equation*}
      \PP_G\bigl(\Ch[B^v(\lfloor \rho N \rfloor, N)]\bigr) \geq \eta(\rho),
      \qquad N \in \NN, \ v\in V.
    \end{equation*}
  \item There exist  $\rho_0, \eta_0 > 0$ such that 
    \begin{equation*}
      \PP_G\bigl(\Cv[B^v(N, \lfloor \rho_0 N \rfloor)]\bigr) \geq \eta_0,
      \qquad N \ge \rho_0^{-1}, \ v\in V.
    \end{equation*}
  \end{letlist}
  Then there exists $\de=\de(\rho_0,\eta_0, \eta(1), \eta(2\rho_0^{-1}), \eps)>0$ 
  such that $G$ has the \bxp\ \BXP$(\de)$.
\end{lemma}

\begin{proof}[Outline proof]
  Assume (a) and (b) hold. Just as in the proof of
  \cite[Prop.\ 3.1]{GM1} (see also \cite[Remark 3.2]{GM1}), 
  the crossing probabilities of
  boxes of $G$ with aspect-ratio $2$ and horizontal/vertical
  orientations are bounded away from $0$ by a constant that depends
  only on the aspect-ratios of the boxes illustrated in
  \cite[Fig.\ 3.1]{GM1}.
  (Here, the boxes in question are those of $G$ viewed as an isoradial square lattice, that is,
  boxes of the form $B(\cdot;\cdot)$ defined before the lemma.)
  Therefore, the hypothesis of Proposition \ref{grid_bxp} 
  holds with suitable constants, and the claim follows
  from its conclusion.
\end{proof}

\section{The \stt}\label{sec:stt}

We review the basic action of the \stt, and show its harmony
with isoradial embeddings. It is shown in Section \ref{sec:xch} how a sequence of \stt s
may be used to exchange two tracks of an isoradial square lattice.

\subsection{Star--triangle transformation}\label{sec:stt0}

The following material is standard but is included for completeness.
For proofs and details see, for example, \cite{GM1}. 

Consider the triangle $\De=(V,E)$ 
and the star $\De'=(V',E')$ of Figure \ref{fig:star_triangle_transformation}. 
Let $\bp=(p_0,p_1,p_2) \in [0,1)^3$ be a triplet of parameters.
Write $\Om=\{0,1\}^E$ with associated product probability measure $\PP_\bp^\tri$
with intensities $p_i$ (as in the left diagram of Figure \ref{fig:star_triangle_transformation}),
and $\Om'=\{0,1\}^{E'}$ with associated measure $\PP_{1-\bp}^\hex$, 
with intensities $1-p_i$ (as in the right diagram of Figure \ref{fig:star_triangle_transformation}).
Let $\omega\in\Om$ and $\om'\in\Om'$.
For each graph we may consider open connections between its vertices,
and we abuse notation by writing, for example, $x \xleftrightarrow{\De,\om} y$ for the \emph{indicator
function} of the event that $x$ and $y$ are connected 
in $\De$ by an open path of $\om$. 
Thus connections in $\De$ are described by the family 
$( x \xleftrightarrow{\De , \omega} y: x,y \in V)$ of random variables, and similarly for $\De'$.

\begin{figure}[htb]
 \centering
    \includegraphics{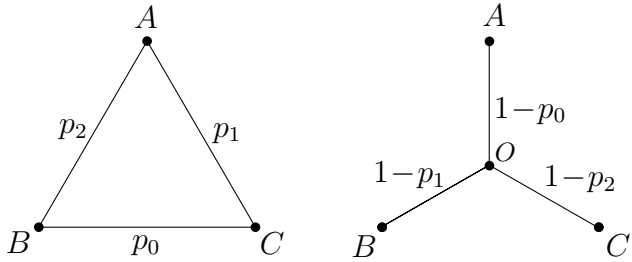}
  \caption{The star--triangle transformation}
  \label{fig:star_triangle_transformation}
\end{figure}

\begin{prop}[Star--triangle transformation] \label{simple_star_triangle}
  Let $\bp\in[0,1)^3$ be such that
  \begin{align}
    p_0 +p_1 + p_2 - p_0 p_1 p_2 = 1. \label{self_dual}
  \end{align}
  The families
  $$
  \left( x \xleftrightarrow{\De, \omega} y : x,y = A,B,C\right), \quad
  \left( x \xleftrightarrow{\De', \om'} y :x,y = A,B,C\right),
  $$
  have the same law. 
\end{prop}

\begin{figure}[htb]
  \begin{center}
    \includegraphics[width=1.0\textwidth]{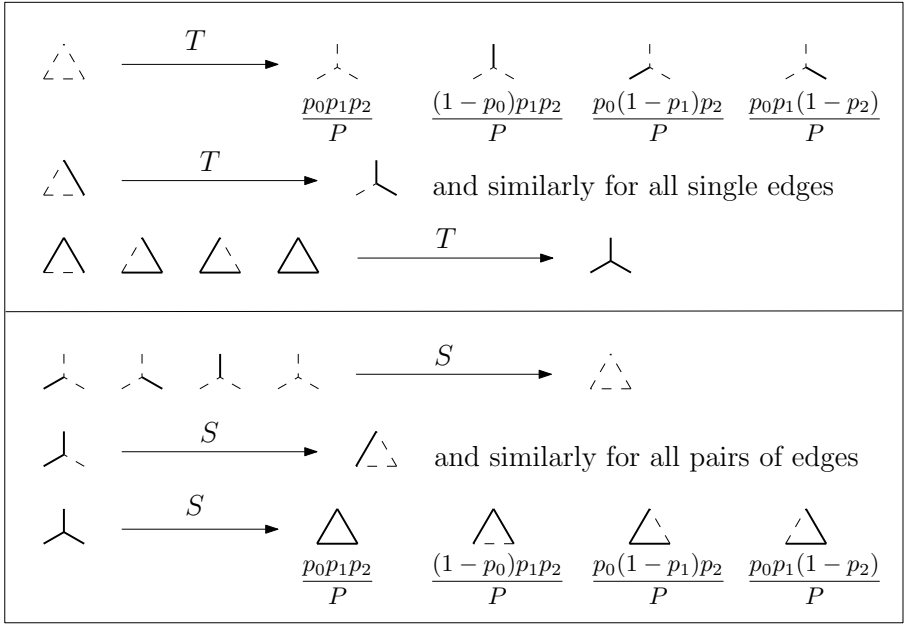}
  \end{center}
  \caption{The random maps $T$ and $S$. Note
that $P:=(1-p_0)(1-p_1)(1-p_2)$.}
  \label{fig:simple_transformation_coupling}
\end{figure}

Next we explore couplings of the two measures.
Let $\bp\in[0,1)^3$ satisfy \eqref{self_dual}, 
and let $\Om$ (\resp, $\Om'$) have associated measure $\PP_\bp^\tri$ (\resp, $\PP_{1-\bp}^\hex$) as above. 
There exist random mappings $T:\Om \to \Om'$ and $S: \Om'\to\Om$ 
such that $T(\om)$ has law $\PP_{1-\bp}^\hex$, and 
$S(\om')$ has law $\PP_\bp^\tri$. 
Such mappings are given in Figure \ref{fig:simple_transformation_coupling},
and we shall not specify them more formally here. 
Note from the figure that $T(\om)$ 
is deterministic for seven of the eight elements of $\Om$; 
only in the eighth case does $T(\om)$ involve further randomness. 
Similarly, $S(\om')$ is deterministic except for one special $\om'$.
Each probability in the figure is well defined since 
$P := (1-p_0)(1-p_1)(1-p_2)>0$.

\begin{prop}[Star--triangle coupling]\label{prop:st-coupling}
  Let $\bp \in [0,1)^3$ satisfy \eqref{self_dual}
  and let $S$ and $T$ be as in Figure \ref{fig:simple_transformation_coupling}.
  With $\omega$ and $\om'$ sampled as above,
  \begin{letlist}
  \item
    $T(\omega)$ has the same law as $\om'$, 
  \item $S(\om')$ has the same law as $\omega$,
  \item for  $x,y \in \{ A,B,C \}$, $x \xleftrightarrow{\De,\omega} y$
    if and only if  $x \xleftrightarrow{\De',T(\omega)} y$,
  \item for  $x,y \in \{ A,B,C \}$, $x \xleftrightarrow{\De',\omega'} y$
    if and only if  $x \xleftrightarrow{\De,S(\omega')} y$.
  \end{letlist}
\end{prop}

\subsection{The \stt\ for isoradial graphs}\label{sec:sttiso}

Let $G=(V,E)$ be an isoradial graph, and let $\De$ be a triangle of $G$ with vertices $A$, $B$, $C$.
Seen as a transformation between graphs, the \stt\ changes $\De$ into a star $\De'$ with a
new central vertex $O\in\RR^2$. It turns out that $O$ may be
chosen in such a way that the new graph, denoted $G'$,  is isoradial
also. The right way of seeing this is via the
diamond graph $G^\di$, as illustrated in Figure \ref{fig:isoradial_star_triangle}.
This construction has its roots in the $Z$-invariant Ising model of
Baxter \cite{Baxter_book,Bax86}, studied in the context of isoradial graphs by
Mercat \cite{Merc}, Kenyon \cite{Ken02}, and  Costa-Santos \cite{Costa}
(see also \cite{BdTB}).

\begin{figure}[htb]
  \begin{center}
    \includegraphics[width=0.8\textwidth]{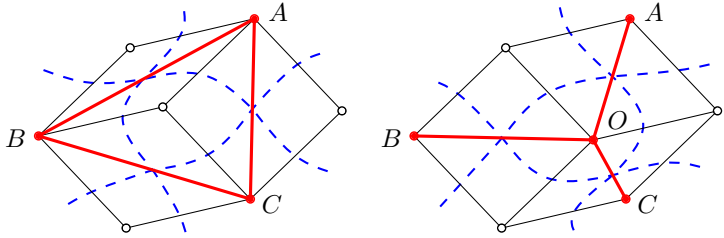}
  \end{center}
  \caption{The triangle on the left is replaced by the star on the right. The new vertex $O$
is the circumcentre of the three dual vertices of the surrounding hexagon of $G^\di$.}
  \label{fig:isoradial_star_triangle}
\end{figure}

The triangle $\De$ comprises the diagonals of three rhombi of $G^\di$. These rhombi form the
interior of a \emph{hexagon} with primary vertices $A$, $B$, $C$ and three further dual
vertices. Let $O$ be the circumcentre of these dual vertices. Three new rhombi
are formed from the hexagon augmented by $O$ (as shown). 
The star $\De'$ has edges $AO$, $BO$, $CO$, and the
ensuing graph is isoradial (since it stems from a rhombic tiling).

By an examination of the angles in the figure,
the canonical measure on $\De'$ is that obtained from $\De$ by the \stt\  of Section \ref{sec:stt0}.
That is, the \stt\ maps $\PP_G$ to $\PP_{G'}$.
Furthermore, for $\eps>0$,
\begin{equation}\label{G622}
\mbox{$G$ satisfies  \BAC$(\eps)$
if and only if $G'$ satisfies \BAC$(\eps)$.}
\end{equation}

We shall sometimes view the \stt\ as acting on the rhombic tiling $G^\di$
rather than on $G$, and thereby it acts simultaneously on
$G$ and its dual $G^*$.  

The \stt\ of Figure \ref{fig:isoradial_star_triangle} is said to 
act on the \emph{track-triangle} formed by the
tracks on the left side, and to 
\emph{slide} one of the tracks illustrated there \emph{over the intersection} of the other two,
thus forming the track-triangle on the right side.

A \stt\ maps an open
path of $G$  to an open path of $G'$. We shall not spell this out in detail,
but  recall the ideas from \cite{GM1}. Let $\pi$ be an open path
of $G$ that intersects some hexagon $H$ of $G^\di$, and consider the \stt\  $\sigma$
acting in $H$. Since $\sigma$ preserves open connections within $H$, it maps $\pi$ to
some $\sigma(\pi)$ containing an open path. A minor complication arises if $H$
contains a star of $G$ and $\pi$ ends at the centre of this star. In this case,
the endpoint of $\sigma(\pi)$ is a vertex of the resulting triangle.

\subsection{Track-exchange in an isoradial square lattice}\label{sec:xch}

Let $G$ be an isoradial square lattice. The tracks of $G$ are to be viewed as doubly-infinite
sequences of rhombi with a common vector. In this section, we describe a  procedure
for interchanging two consecutive parallel tracks.

Consider a vertical strip $G = G_{\balpha, \bbeta}$ of the square lattice, 
where $\balpha = (\a_i: -M \leq i \leq N)$ and $\bbeta = (\b_j :j \in \ZZ)$
are vectors of angles satisfying \BAC$(\eps)$, \eqref{bounded_angles}. 
Thus every \emph{finite} face of $G$ has circumradius $1$.
(There are also \emph{infinite} faces to the left and right of the strip.)
There are two types of tracks in $G$, the finite horizontal tracks $(s_j)$, 
and the infinite vertical tracks $(t_i)$. 
We explain next how to exchange two adjacent horizontal tracks by a sequence of \stt s,
employing a process that is implicit in \cite{Ken02}. 
Track $s_j$ has transverse angle $\beta_j$, as illustrated in Figure 
\ref{fig:isoradial_square},
and the `exchange' of two tracks may be interpreted as the 
interchange of their transverse angles.

We  write $\Sigma_j$ for the operation that exchanges the tracks at levels
${j-1}$ and $j$. When applied to $G$, $\Si_j$ exchanges $s_{j-1}$ and $s_j$,
and we describe $\Si_j$ by reference to $G^\di$. 
If $\beta_j = \beta_{j-1}$, there is nothing to do, and $\Si_j$ interchanges the labels
of the tracks without changing the transverse angles.
Assume $\beta_j > \beta_{j-1}$.
We insert a new rhombus on the left side of the strip formed of $s_{j-1}$ and $s_j$, 
marked in green in Figure \ref{fig:line_exchange}.
This creates a hexagon in $G^\di$, containing either a triangle or a star of $G$.
The \stt\ is applied within this hexagon, thereby moving the new rhombus to the right. 
By repeated \stt s, we  `slide' the new rhombus
along the two tracks from left to right. 
When it reaches the right side, it is removed.
In the new graph, the original tracks $s_{j-1}$ and $s_j$ have been exchanged 
(or, more precisely, the transverse angles of the tracks at levels $j-1$ and $j$ have been interchanged).
Let $\Si_j$ be the transformation thus described, and say
that $\Sigma_j$ `goes from left to right' when $\beta_j > \beta_{j-1}$.
If $\beta_j < \beta_{j-1}$, we construct $\Si_j$ `from right to left'.

\begin{figure}[htb]
  \begin{center}
    \includegraphics[width=0.9\textwidth]{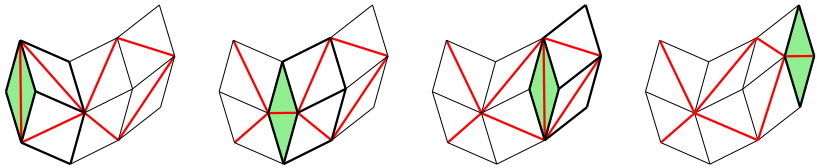}
  \end{center}
  \caption{A new rhombus is introduced on the left (marked in green).
    This is then `slid' along the pair of tracks by a sequence of \stt s,
    until it reaches the right side where it is removed.}
  \label{fig:line_exchange}
\end{figure}

Viewed as an operation on graphs, $\Sigma_j$ replaces an isoradial graph $G$ by another
isoradial graph $\Si_j(G)$. 
It operates on configurations also, as follows.
Let $\omega$ be an edge-configuration of  $G$, and assign a random
state to the new `green' edge with the distribution appropriate
to the isoradial embedding. 
The star--triangle transformations used in $\Sigma_j$ 
are independent applications of the kernels $T$ and $S$ of 
Figure \ref{fig:simple_transformation_coupling}.
The ensuing configuration on $\Si_j(G)$ is written $\Sigma_j(\omega)$. 
Thus $\Si_j$ is a random operator on $\om$, with randomness stemming from
the extra edge and the \stt s. Note that
$\Sigma_j$ is not a local transformation, 
in that the state of an edge in $\Sigma_j(G)$ 
depends on the states of certain distant edges.

\begin{figure}[htb]
  \begin{center}
    \includegraphics[width=1.0\textwidth]{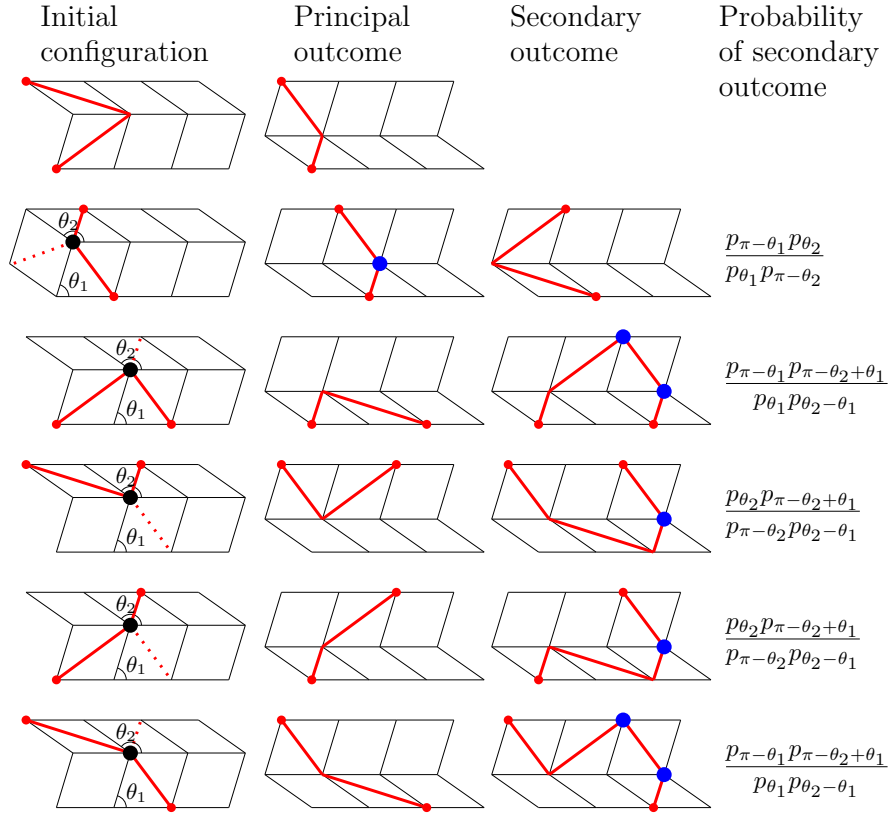}
  \end{center}
  \caption{The six possible ways in which $\gamma$ 
    may intersect the strip in two edges between height $j-1$ and $j+1$, 
    and the corresponding actions of $\Sigma_j$. 
    In five cases, the resulting configuration can be non-deterministic. 
    If the dotted edge is closed, 
    the resulting configuration is in the second column. 
    If it is open,
    the resulting configuration is that of the third column with the given probability
    (recall from \eqref{G101} that $p_{\pi-\a} = 1-p_\a$).
    The movement of black vertices can cause the height increases
    marked in blue. The tracks $s_k$ are drawn as horizontal for simplicity,
    and $\theta_1=\b'-\a_{m}$, $\theta_2 = \b - \a_m$, where $v_{m,j}$ denotes the black vertex,
    and $\b'$/$\b$ is the transverse angle of the lower/upper track. }
  \label{fig:path_transfo}
\end{figure}

Let $\sigma_j$ denote the transposition of the $(j-1)$th and $j$th terms of a sequence.
We may write
\begin{align*}
  \Sigma_j( G_{\balpha, \bbeta},\PP_{\balpha,\bbeta})  = 
  (G_{\balpha, \sigma_j\bbeta},\PP_{\balpha,\sigma_j\bbeta}).
\end{align*}
When applying the $\Si_j$ in sequence, we distinguish between the \emph{label} $s_j$
of  a track and its \emph{level}. 
Thus, $\Si_j$ interchanges the tracks currently at levels $j-1$ and $j$.

\begin{figure}[htb]
  \begin{center}
    \includegraphics[width=0.8\textwidth]{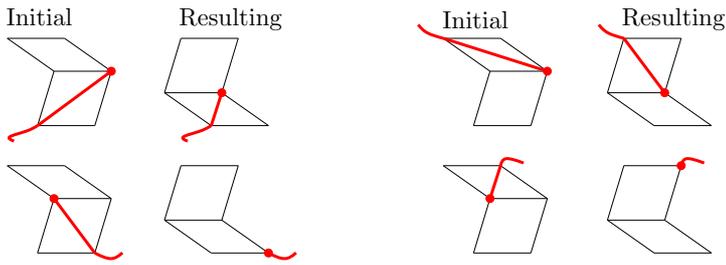}
  \end{center}
  \caption{If an endpoint of $\gamma$ lies between the two tracks, 
    the corresponding edge is sometimes contracted to a single point. }
  \label{fig:path_transfo_ep}
\end{figure}

We consider next the transportation of open paths.
Let $\omega$ be a configuration on $G_{\balpha, \bbeta}$, and let $\gamma$ be an $\omega$-open path.
The action of a \stt\ on $\g$ is discussed 
in detail in \cite[Sect.\ 2.3]{GM1}.
The transformation $\Sigma_j$ comprises three steps: 
the addition of an edge to $G_{\balpha,\bbeta}$, a series of star--triangle transformations, and the removal of an edge. 
The first step does not change $\gamma$, and the effect of the second step is
discussed in Section \ref{sec:sttiso} and   the following paragraphs. 
If the removed edge
is in the image of the path $\g$ at the moment of removal, we say that $\Si_j$ 
\emph{breaks} $\gamma$. 
Thus, $\Si_j(\gamma)$ is an open path of $\Si_j(G)$ whenever
$\Si_j$ does not  break $\gamma$.
In applying the $\Si_j$, we shall choose the strip-width $M+N$ 
sufficiently large that 
open paths of the requisite type do not reach the boundary, 
and are therefore not broken. 

Finally, we summarise in Figures \ref{fig:path_transfo}--\ref{fig:path_transfo_ep}
the action of $\Sigma_j$ on the path $\gamma$, with $M$ and $N$ chosen sufficiently large. 
Consider two tracks 
$s'$, $s$ at respective levels $j-1$ and $j$, with transverse angles $\b'$ and $\b$.
Edges of $\gamma$ lying outside levels ${j-1}$ and $j$ are unchanged by $\Si_j$. 
The intersection of $\gamma$ with these two tracks forms a set of open sub-paths
of length either $1$ or $2$; there are four 
possible types of length $1$, and six of length $2$. 
We do not describe this in detail, but refer the reader to the figures, which are
drawn for the case  $\b > \b'$.
The path $\g$ may cross the tracks in more than one of the diagrams
on the left of  Figure \ref{fig:path_transfo}, and the image path contains
an appropriate subset of the edges in the listed outcomes. 
Note that, if the intersections of $\gamma$ with $s$ and $s'$ 
are at distance at least $2$ from the lateral boundaries, then 
$\Si_j$ does not break $\gamma$,

In the special case when $\b = \b'$, 
$\Sigma_j$ interchanges the labels of $s'$ and $s$ but
alters neither embedding nor configuration. 
In this degenerate case, we set $\Sigma_j(\gamma) = \gamma$, and note that
Figure \ref{fig:path_transfo} remains accurate. 

\section{Proof of Theorem \ref{main}: Isoradial square lattices}\label{sec:proof1}

\subsection{Outline of proof}\label{sec:outline}
The proof for isoradial square lattices
is based on Proposition \ref{bxp_transport}, following.
For
$\bef \in [0,2\pi)$, we write $G_{\balpha, \bef}$ for the isoradial square lattice 
generated by the angle-sequence $\balpha$ and the constant sequence $(\bef)$. 

\begin{prop}\label{bxp_transport}
  Let $\de,\eps>0$. There exists $\de'=\de'(\de,\eps)>0$ such that the
  following holds.  
  Let $G_{\balpha, \bbeta}$ be an isoradial square lattice 
  satisfying \BAC$(\eps)$,
  and let $\bef \in [0,2\pi)$ be such that $\balpha$ and the 
  constant sequence  $(\bef)$ 
  satisfy \BAC$(\eps)$, \eqref{bounded_angles}.
  If $G_{\balpha, \bef}$ satisfies \BXP$(\de)$, 
  then $G_{\balpha, \bbeta}$ satisfies \BXP$(\de')$.
\end{prop}

\begin{cor}\label{bxp_square}
Let $\eps>0$. There exists $\de=\de(\eps)>0$ such that every
isoradial square lattice satisfying \BAC$(\eps)$ 
has the \bxp\ \BXP$(\de)$. 
\end{cor}

Since $\sG(\eps, 1)$ is  the set of isoradial square lattices satisfying \BAC$(\eps)$,
the corollary is equivalent to \eqref{G2231} with $I=1$.

\begin{proof}[Proof of Corollary \ref{bxp_square}]
  Let $\eps>0$ and let $G_{\balpha,\bbeta}$ satisfy \BAC$(\eps)$.
  
  First, assume that one of the two sequences $\balpha$, $\bbeta$ is constant. 
  Without loss of generality we may take $\balpha$ to be constant, 
  and by rotation of the graph, we shall assume $\balpha \equiv 0$. 
  There exists $\de>0$ such that 
  the homogeneous square lattice $G_{0,\pi/2}$ satisfies \BXP$(\de)$
  (see, for example, \cite[Sect.\ 1.7]{Grimmett_Percolation}).
  By Proposition \ref{bxp_transport} with $\bef = \frac12\pi$, 
  $G_{\balpha,\bbeta}$ satisfies \BXP$(\de')$ for some $\de'=\de'(\de,\eps)>0$.

  Consider now the case of general $\balpha$, $\bbeta$. 
  By the above, $G_{\balpha,\b_0}$ satisfies \BXP$(\de')$.
  By Proposition \ref{bxp_transport} with $\bef = \b_0$, 
  $G_{\balpha,\bbeta}$ satisfies \BXP$(\de'')$ for some $\de''=\de''(\de',\eps)>0$.
\end{proof}

By Lemma \ref{alternate_bxp}, Proposition \ref{bxp_transport} follows from 
the forthcoming Propositions 
\ref{horizontal_transport} and \ref{vertical_transport}, dealing \resp\ 
with horizontal and vertical crossings. 

The proofs are outlined in the remainder of this section. 
The basic idea is that 
pieces of $G_{\balpha, \bbeta}$ and $G_{\balpha, \bef}$ 
may be glued together along a horizontal track.
Track-exchanges may then be performed repeatedly
in order to swap parts of $G_{\balpha, \bef}$ and $G_{\balpha, \bbeta}$
while maintaining the existence of certain open paths.

Let $\balpha$, $\bbeta$, and $\bef$ be as in Proposition \ref{bxp_transport}.
Fix $N$, to be chosen separately in the two proofs, and define
\begin{equation*}
  \wt{\b}_j =
  \begin{cases}
    \bef &\text{ if } j < N , \\
    \b_{j-N} &\text{ if } j \geq N.
  \end{cases}
\end{equation*}
We refer to the part of $G = G_{\balpha, \wt\bbeta}$  above height $N$  as
the \emph{irregular block}, and that with height between $0$ and $N$
as the \emph{regular} block. The regular
block may be viewed as part of $G_{\balpha, \bef}$, 
and the irregular block as part of $G_{\balpha, \bbeta}$.
We will only be interested in the graph above height $0$.

In Section \ref{sec:horizontal}, we explain how to
transport \emph{horizontal box-crossings} from the regular block to the irregular block.
Consider an open horizontal crossing of a wide rectangle in the regular block of $G$. 
A sequence of track-exchanges is made from the top to the bottom of the regular block
in such a way that the regular block moves upwards. 
See Figure \ref{fig:horizontal_sigmas}.

The evolution of the open crossing is observed throughout the transformations. 
Its endpoints are pinned to the lowest track, and therefore
do not change. 
The open path may itself drift upwards, 
and the core of the proof of Proposition 
\ref{horizontal_transport} lies in a probabilistic control of its rate of drift.

\emph{Vertical box-crossings} are studied in Section \ref{sec:vertical}.
This time, the tracks of the regular block are moved upwards by a process of track-exchange, 
as illustrated in Figure \ref{fig:vertical_sigmas}. 
Consider an open vertical crossing of some rectangle of the regular block.
During the exchanges of tracks, the highest point of the path may drift downwards, 
and as before one needs some control on its rate of drift.
It is key that, after all tracks of the regular block have been moved upwards, 
the irregular block contains an open vertical crossing of some (wider and lower) rectangle. 
This is achieved by showing that the height of the upper endpoint of
the crossing decreases by at most one at each step, and does not decrease 
with some uniformly positive probability. 

The arguments and illustrations of Section \ref{sec:xch}
will be central to the proofs of Propositions \ref{horizontal_transport} and \ref{vertical_transport}. 

Separate considerations of  horizontal and vertical crossings are required 
  since $\balpha$ and $\bbeta$ do not play equal roles. 
  Consider the proof of  Corollary \ref{bxp_square}.
When passing from $G_{0,\pi/2}$ to $G_{\balpha,\bbeta}$ 
  with one of the sequences $\balpha$, $\bbeta$ constant, 
  both and horizontal and vertical box-crossings need be transported. 
 On the other hand,  Proposition \ref{horizontal_transport} suffices in the
second part of the proof, as explained 
in the following remark.  
  
\begin{rem}
  The material in Section \ref{sec:vertical}, 
  and specifically Proposition \ref{vertical_transport},
  may be circumvented by use of \cite[Thm 1.5]{GM1}, where
  the \bxp\ is proved for so-called `highly inhomogeneous' square lattices.
  We do not take this route here since it would reduce the integrity of the current
  proof, and would require the reader to be familiar with a different method 
  of applying the \stt\ to square (and triangular) lattices.
  
  Here is an outline of the alternative approach. 
  In the language of \cite{GM1}, an isoradial square
  lattice $G_{\balpha,\xi}$ satisfying \BAC$(\eps)$ is a highly inhomogeneous square
  lattice satisfying  \cite[eqn (1.5)]{GM1} 
  (with an adjusted value of $\eps$). 
  By \cite[Thm 1.5]{GM1}, such a lattice has the \bxp. 
  By Proposition \ref{horizontal_transport},
  horizontal box-crossings may be transported from $G_{\balpha,\xi}$
  to the more general isoradial square lattice $G_{\balpha, \bbeta}$.
  Similarly, by interchanging the roles of the
  horizontal and vertical tracks of $G_{\balpha,\bbeta}$, we 
  obtain the existence of vertical box-crossings in that lattice. 
  Such crossing probabilities are now combined, using Lemma \ref{grid_bxp},
  to obtain Theorem \ref{main}.
\end{rem}

The following is fixed for the rest of this section.
Let $\eps >0$, and 
let  $\balpha$, $\bbeta$ be sequences of angles  
satisfying \BAC$(\eps)$, \eqref{bounded_angles}.
Let $\bef$ be an angle such that $\balpha$ and $(\bef)$ 
satisfy \BAC$(\eps)$, \eqref{bounded_angles}.
All constants in this section may depend on $\eps$, but 
not further on $\balpha$, $\bbeta$, $\bef$ unless otherwise stated. 

\subsection{Horizontal crossings}\label{sec:horizontal}

\begin{prop}\label{horizontal_transport}
  There exist  $\la,N_0 \in \NN$, depending on $\eps$
  only, such that, 
  for $\rho \in\NN $ and $N \ge N_0$,
  \begin{align}
    &\PP_{\balpha, \bbeta}\bigl(\Ch [B( (\rho - 1) N, \la N )]\bigr) \nonumber\\  
    & \hskip1cm \ge (1-\rho e^{-N}) \PP_{\balpha, \bef}\bigl(\Ch[B(\rho N,N )]\bigr) \nonumber\\
    &\hskip2cm\times\PP_{\balpha, \bef}\bigl(\Cv[B (-\rho N, -(\rho-1)N;0,N )]\bigr)\nonumber\\
    &\hskip2cm\times\PP_{\balpha, \bef}\bigl(\Cv[B ((\rho-1) N, \rho N;0,N )]\bigr).
  \label{G14}\end{align}
\end{prop}

\begin{proof}

Fix  $\rho \in \NN$ with $\rho>1$, 
  and $\la, N_0\in\NN$  to be chosen later, and let $N\ge N_0$. 
  Recall the graph $G = G_{\balpha, \wt\bbeta}$  given in Section \ref{sec:outline}.
%
%
  We work on a vertical strip $\{v_{i,j}: -M \le i \le M\}$  
  of $G$ with width $2M$, where 
  \begin{equation}\label{G221}
    M =  (\rho+ 2\la + 1) N,
  \end{equation}
  and we truncate $\balpha$ to a finite sequence 
  $(\a_i: - M \leq i \leq M-1)$.
  
  We will work with graphs obtained from $G$ by a
  sequential application of the transformations $\Si_j$  of Section \ref{sec:xch}, 
  and to this end we let
  \begin{equation}\label{G222}
    U_k = \Sigma_{k} \circ \Si_{k+1} \circ \dots \circ \Sigma_{N + k - 1},\qquad k \ge 1. 
  \end{equation}
  Note that  $U_k$ moves the track at level $N+k-1$ to level $k-1$, while raising the tracks 
  at levels ${k-1},\dots,{N+k-2}$ by one level 
  each  (see Figure \ref{fig:horizontal_sigmas}).
  We propose to apply $U_1,U_2,\dots, U_{\la N}$ to $G$ in turn,
  thereby moving part of the irregular block beneath the regular block.
  
  \begin{figure}[htb]
    \begin{center}
      \includegraphics[width=0.9\textwidth]{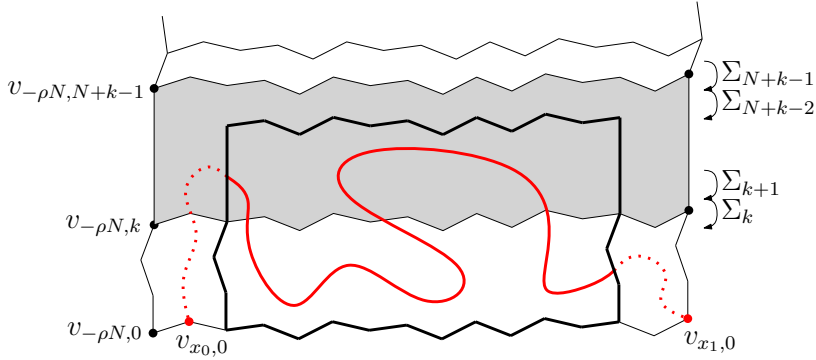}
    \end{center}
    \caption{The transformation $U_k$ raises the (shaded) regular block by one unit,
      and moves the track above by $N$ units downwards.}
    \label{fig:horizontal_sigmas}
  \end{figure}

  Let $E_N$ be the event that there exists an open path of $G$ within
  $B(\rho N, N)$, with endpoints 
  $v_{x_0, 0}$ and $v_{x_1, 0}$ for some 
  $x_0 \in [-\rho N,  - (\rho-1) N]$ and $x_1 \in [(\rho-1) N, \rho N]$.
  By the definition of $\wt{\bbeta}$,  $B(\rho N, N)$ is entirely contained 
  in the regular block of $G$. By the Harris--FKG inequality,
  \begin{align}
    \PP_{\balpha, \wt\bbeta}(E_N) 
    &\ge  \PP_{\balpha, \bef}\bigl(\Ch[B(\rho N,N )]\bigr) \nonumber\\
    &\hskip1cm\times\PP_{\balpha, \bef}\bigl(\Cv[B (-\rho N, -(\rho-1)N;0,N )]\bigr)\nonumber\\
    &\hskip1cm\times\PP_{\balpha, \bef}\bigl(\Cv[B ((\rho-1) N, \rho N;0,N )]\bigr).\label{G105}
  \end{align}
  
  Let $\omega^0$ be a configuration on $G$, 
  chosen according to $\PP_G$.
  For $k \in \NN$, let $G^0=G$ and
  \begin{align*}
    G^k  = U_{k} \circ \dots \circ U_1 (G),\quad
    \omega^k  = U_{k} \circ \dots \circ U_1 (\omega^0).
  \end{align*}
  The family 
  $(\omega^k : k \ge 0)$ is a sequence of configurations 
  on the $G^k$ with 
  associated law denoted $\PP$.
  Note that $\PP$ is given in terms
  of the law of $\om^0$, and of the randomizations contributing to the $U_i$.
  The marginal law of $\om^k$ under $\PP$ is $\PP_{G^k}$. 
  
  Let $\omega^0 \in E_N$, and let $\ga^0$ be  a path 
  in $B(\rho N, N)$ with endpoints 
  $v_{x_0, 0}$ and $v_{x_1, 0}$ for some 
  $x_0 \in [-\rho N,  - (\rho-1) N]$ and $x_1 \in [(\rho-1) N, \rho N]$. 
  Let $\ga^k = U_{k} \circ \dots \circ U_1 (\ga^0)$.
   The path evolves as we apply the $U_k$ sequentially,
  and most of this proof is directed at studying the sequence
  $\g^0,\g^1,\dots,\g^{\la N}$.
  
  First we show that the path is not broken by the track-exchanges.
 For $0 \leq k \leq \la N$, set
  \begin{align*}
    D^k = \bigl\{ v_{x,y} \in (G^k)^\di: |x| \le  (\rho + 1) N + 2k - y, 
    \ 0 \leq y \leq N+k \bigr\}.
  \end{align*}
  The proof of the following elementary lemma is summarised at the end of this section.

  \begin{lemma}\label{D_control}
    For $0 \leq k \leq \la N$, $\ga^k$ is an open path
    contained in $D^k$. 
  \end{lemma}
  
  The set $\{v_{x,0}: x\in \ZZ\}$ of vertices of $G^\di$ is invariant under the $U_k$, whence 
  the endpoints of the $\ga^k$ are constant for all $k$.  
  It follows that the horizontal span of $\ga^{\la N}$ is at least $2(\rho-1)N$. 
  
  If $\ga^{\la N}$ has maximal height not exceeding $\la N$, 
  then it contains a $\om^{\la N}$-open horizontal crossing of $B((\rho - 1)N, \la N)$.
  The graph $G^{\la N}$ agrees with $G_{\balpha, \bbeta}$ within $B((\rho - 1)N, \la N)$,
  so that
  \begin{equation*}
    \PP_{\balpha, \bbeta}\bigl(\Ch[B ((\rho - 1) N, \la N )]\bigr) \geq 
    \PP \bigl( h(\ga^{\la N}) \leq \la N \bigmid E_N \bigr)
    \PP(E_N) .
  \end{equation*}  
  By \eqref{G105}, it suffices to show the existence of 
  $\la, N_0\in \NN$ such that, 
  \begin{align}
    \PP \bigl( h(\ga^{\la N}) \leq \la N \bigmid \omega^0 \bigr)
    \geq 1- \rho e^{-N}, \qquad N \geq N_0,\ \omega^0 \in E_N, 
    \label{growth} 
  \end{align}
  and the rest of the proof is devoted to this. 
  The basic idea is similar to the corresponding step of \cite{GM1}, but the calculation
is more elaborate.
  
  Let $\omega^0 \in E_N$
  and let $\ga^0$ be as above. 
  We observe the evolution of the heights of the images of $\ga^0$ 
  within each column.
  For $n \in \ZZ$ and $0 \leq k \leq \la N$, set 
  \begin{align*}
    \col_n = \{v_{n,y} : y \in \ZZ \},\quad
    h_n^k = \begin{cases} h ( \ga^{k} \cap \col_n ) &\text{if } \ga^k \cap \col_n \ne \es,\\
      -\infty &\text{otherwise}.
    \end{cases}
  \end{align*}
  Thus, $h(\ga^{\la N}) = \sup \{h_n^{\la N}:n \in \ZZ\}$.
  
  The process $(h_n^k: n \in \ZZ)$, $k=0,1,\dots,\la N$,
  has some lateral drift depending on the directions of the track-exchanges $\Sigma_j$.
  We will modify it in order to relate it to the growth process of \cite{GM1}.
  The track above the regular block is transported by $U_{k}$ through the regular block, 
  and thus all $\Sigma_j$ contributing to $U_k$ are in the same direction 
  (either all move from right to left, or from left to right). 
  Let $(d_k :k \geq 0)$ be given by $d_0=0$ and 
  \begin{align*}
    d_{k+1} =
    \begin{cases}
      d_{k}+1 &\text{ if } \beta_k > \bef, \\
      d_{k}   &\text{ if } \beta_k = \bef, \\
      d_{k}-1 &\text{ if } \beta_k < \bef,
    \end{cases}
  \end{align*}
  and set $H_n^k = h_{n + d_k}^k$.
 The term $d_k$ is included to compensate for the 
  asymmetric lateral drift induced by the directions of the track-exchanges 
  (see Figure \ref{fig:path_transfo}).
  Thus $H_n^k$ has a roughly symmetric evolution as $n$ increases. 
  The rest of the proof is devoted to the process 
  $\bH^k = ( H_n^k : n \in \ZZ )$, $k =0,1,\dots, \la N$.

  We introduce some notation to be used in the proof.
  A sequence $\bR = (R_n:n \in \ZZ) \in (\ZZ\cup\{-\oo\})^\ZZ$ is termed a \emph{range}.
  The \emph{height in column} $n$ of $\bR$ is the value $R_n$,  and the  
  \emph{height} of the range is $\sup\{R_n: n \in \ZZ\}$.
  For two ranges $\bR^1$, $\bR^2$, we write $\bR^1 \leq \bR^2$
  if $R^1_n \leq R^2_n$ for  $n\in \ZZ$.
  The \emph{maximum} of a family of ranges is the pointwise supremum sequence.
  The range $\bR$ is called \emph{regular} if
  \begin{align}
    |R_{n+1}-R_n| \leq 1, \qquad  n\in\ZZ. \label{reg_mount} 
  \end{align}
  
  The \emph{mountain} at a point $(n,r)\in \ZZ^2$ is defined to be
  the range $\bM(n,r)=(M(n,r)_l : {l \in \ZZ})$ given by
  \begin{align*}
    M(n,r)_l =
    \begin{cases}
      r - |n - l| + 1 &\text{ for } l \neq n, \\
      r         &\text{ for } l = n. \\
    \end{cases}
  \end{align*}
  Note that mountains have flat tops of width $3$ centred at $(n,r)$, 
  and sides with gradient $\pm 1$.
  The \emph{covering} of a range $\bR$ is the range $C(\bR)$ 
  formed as the union of the mountains of each of its elements:
  $$ 
  C(\bR) = \max \{\bM(n,R_n) : n \in \ZZ\}.
  $$
  We note that $\bR \le C(\bR)$ with sometimes strict inclusion, and also
  that $\bR$ and $C(\bR)$ have the same height.
  If $\bR$ is regular,  the heights of $\bR$ and $C(\bR)$ in any given column
  differ by at most $1$.
  See Figure \ref{fig:growth_process} for an illustration of these notions
  and their role in the evolution of $\bH$.
  We return to the study of $(\bH^k)$. 
  
  \begin{figure}[htb]
    \begin{center}
      \includegraphics[width=1.0\textwidth]{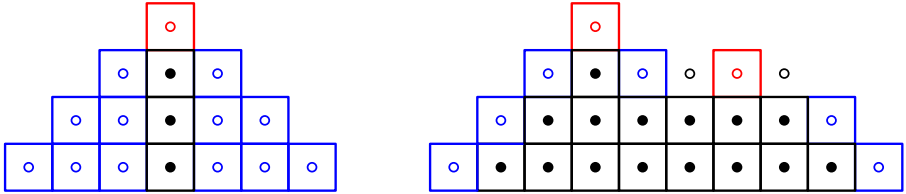}
    \end{center}
    \caption{ \emph{Left}: One step in the evolution of $\bH$. 
      The initial range $\bH^0$ has only one occupied column (black). 
      The blue/black squares form the mountain of the black column. 
      The red square is added at random. 
      \emph{Right}: One step in the evolution of the $\bX^k$ (or $\bH^k$ when regular). 
      The black squares are the configuration at step $k$,
      the blue squares are the additions at time $k+1$ due to the covering, 
      and the red squares are the random additions.}
    \label{fig:growth_process}
  \end{figure}
  
  \begin{lemma} \label{column_growth_control}
    There exists $\eta=\eta(\eps)\in(0,1)$, and 
    a family of independent Bernoulli random variables 
    $(Y_n^k :  n \in \ZZ,\  0 \leq k <\la N)$
    with common parameter $\eta$ 
    such that
    \begin{align}
      H_{n}^{k+1} \leq \max \{ C(\bH^{k})_n , H_{n}^{k} + Y_n^k \},
      \quad \ n \in \ZZ, \ 0 \leq k < \la N.  \label{H_evolution}
    \end{align}
  \end{lemma}
  
  The $(Y_n^k)$ are random variables used in the \stt s, and the probability
  space may be enlarged to accommodate these variables. 
  
  The proof of the lemma is deferred until later in the section. 
  Meanwhile, we continue the proof of \eqref{growth} 
  by following that of \cite[Prop.\ 3.7]{GM1}.
  Let $(Y_n^k)$ be as in Lemma \ref{column_growth_control},
  and let  $\bX^k:=(X_n^k: n \in \ZZ)$, $k= 0,1,\dots, \la N$, 
  be the Markov chain given as follows.
  \begin{letlist} 
  \item The initial value $\bX^0$ is the regular range given by
    $$
    X_n^0 = \begin{cases} 
      N &\text{for } n \in \left[- \rho N, \rho N \right],\\
      N + \rho N - |n| &\text{for }  n \notin \left[- \rho N, \rho N \right].\\
    \end{cases}
    $$ 
  \item For $k \ge 0$, conditional on $\bX^k$, the range $\bX^{k+1}$ is given by
    \begin{align}
      X_n^{k+1} = \max \{ X_{n-1}^k, X_{n}^k + Y_n^k , X_{n+1}^k \}, \qquad n \in \ZZ. \label{X_evolution}
    \end{align}
  \end{letlist}
  
  We show first, by induction,  that $\bX^k \ge \bH^k $ for all $k$. 
  It is immediate that $\bX^0$ is regular, and that $\bX^0\ge\bH^0$.
  Suppose that $\bX^k \geq \bH^k$. By \eqref{X_evolution}, each 
  range $\bX^k$ is regular  and
  \begin{align}
    \bX^{k+1} \geq C(\bX^k) \geq C(\bH^k). \label{XH1}
  \end{align}
  By \eqref{H_evolution},  $H^{k+1}_n > C(\bH^{k})_n$ only if  $Y^k_n = 1$. 
  Since $X^k_n \geq H^k_n$, we have in this case that 
  \begin{align}
    X^{k+1}_n \ge X^k_n + 1 \geq  H^{k}_n + 1 = H^{k+1}_n.
    \label{XH2}
  \end{align}
  By  \eqref{XH1}--\eqref{XH2}, $ \bX^{k+1} \geq \bH^{k+1}$, and the induction
  step is complete.

  The $\bX^k$ are controlled via the following lemma.

  \begin{lemma}[\cite{GM1}]\label{growth_process_domination}
    There exist $\la,N_0 \in \NN$, depending on $\eta$ only, 
    such that
    \[ 
   \PP \left( \max_n X_n^{\la N} \leq \la N \right) \geq  1 - \rho  e^{ - N}, \qquad
    \rho \in \NN,\ N \geq N_0.
    \]
  \end{lemma}
  
  \begin{proof}[Sketch proof]
    This follows that of \cite[Lemma 3.11]{GM1}. A small difference arises through the minor
    change of  the initial value $\bX^0$, but this is covered by the inclusion
    of smaller-order terms in \cite[eqn\ (3.31)]{GM1}.
  \end{proof}
  
  Let $\la$ and $N_0$ be given thus. For $N \geq N_0$ and $\om^0\in E_N$,
  \begin{align*}
    \PP \bigl( h(\g^{\la N}) \leq \la N \bigmid \omega^0  \bigr)
    &\geq \PP \left(\max_n X_n^{\la N} \leq \la N \right) \\
    &\geq 1- \rho e^{-N}. 
  \end{align*}
  This concludes the proof of Proposition \ref{horizontal_transport}.
\end{proof}

\begin{proof}[Proof of Lemma \ref{column_growth_control}]  
  Let $k \geq 0$ and let $\omega$ be a configuration on $G^k$.
  Let $\ga$ be an open path on $G^k$ that visits no vertex within
  distance $2$ of the sides of $G^k$ and with $h(\ga) \leq N + k$.
  We abuse notation slightly by defining $\bH^{k}$ and $\bH^{k+1}$ 
  as in the proof of Proposition \ref{horizontal_transport}
  with $\ga$ and $U_{k+1}(\ga)$ instead of $\ga^k$ and $\ga^{k+1}$, \resp. 
  That is,
  \begin{align*}
    H_n^k =  h ( \ga \cap \col_{n+ d_k}),\quad    
    H_n^{k+1} =  h ( U_{k+1}(\ga) \cap \col_{n+ d_{k+1}}), 
    \qquad  n \in \ZZ.
  \end{align*}
  We will prove that there exists a family of 
  independent Bernoulli random variables $(Y_n : n \in \ZZ)$, 
  independent of $\omega$, 
  with some common parameter $\eta = \eta(\eps) > 0$ 
  to be specified later, such that
  \begin{equation}
    H_{n}^{k+1} \leq \max \{ C(\bH^{k})_n, H_{n}^{k} + Y_{n}\},
    \qquad n \in \ZZ. \label{one_step_growth} 
  \end{equation}
  
  Once this is proved, the i.i.d.\ 
  family $(Y_n^k :  n \in \ZZ,\ 0 \leq k < \la N)$ may be constructed step by step, 
  by applying the above to the pair $\omega^k$, $\ga^k$ for $0 \leq k < \la N$.
  By Lemma \ref{D_control}, the assumptions on $\ga$ are indeed satisfied by each $\ga^k$. 
  By the independence of $(Y_n: n \in \ZZ)$ and $\omega$ above, 
  the family $(Y_n^k : n,k)$ satisfies the conditions of the lemma.

  It remains to prove \eqref{one_step_growth} for fixed $k$. 
  If $\b_k=\bef$, no track-exchange takes place, 
  hence $\bH^{k+1}=\bH^k$ and \eqref{one_step_growth} holds.
  Suppose $\b_k \neq \bef$. 
  Without loss of generality we may suppose $\b_k > \bef$, so that 
  $d_{k+1} = d_{k}+1$ and the track-exchanges in the application
  of $U:=U_{k+1}$ are all  from left to right. 
  To simplify  notation we shall assume $d_k = 0$.  
  
  Equation \eqref{one_step_growth} is proved in two steps. 
  First, we will show that
  \begin{align}
    H^{k+1}_{n} \leq \max \{ H^k_{n - i} - |i| + 1 : i \in \ZZ \}. \label{neighbours_invade}
  \end{align}
  This equation is a weaker version of  \eqref{one_step_growth}
  in which each $Y_n$ is replaced by $1$. 

  We prove \eqref{neighbours_invade} by analysing the individual track-exchanges 
  of which $U$ is composed.
  For $k \leq j \leq N+k$, let
  $ \Psi_j = \Sigma_{j+1} \circ \dots \circ \Sigma_{N+k}$. 
  Thus, $\Psi_{N+k}$ is the identity, 
  $\Psi_k  = U$, and $\Psi_{j-1} = \Sigma_{j} \comp \Psi_{j}$.
  Recall that the diamond graph is bipartite, with the primal and dual
  vertices as vertex-sets.
  A vertex $v_{n,r}$ is said to be contained in a range $\bR$ if $r \leq R_n$. 
  A set of vertices is contained in $\bR$ if every member is thus contained. 
  
 Let the sequence $(\bL^j : j =N+k,N+k-1,\dots, k)$ of ranges be defined recursively as follows.
  First, $\bL^{N+k} = \bH^k$.
  We obtain $\bL^{j-1}$ from $\bL^{j}$ by increasing its height in certain columns:
  for each primal vertex $v_{n,j}$ contained in $\bL^{j}$,
  the heights in columns $n+1$ and $n+2$ increase
  to $j+1$ and $j$, if not already at that height or greater. 
  
  We claim that  
  $\Psi_j(\ga)$ is contained in $\bL^{j}$ for  $N+k \geq j \geq k$, which is to say that
  \begin{align}
    h( \Psi_j(\ga) \cap \col_{n} ) \leq L^{j}_n, \qquad n \in \ZZ. \label{L_induction}
  \end{align}
  The above holds for $j=N+k$ by the definition of $\bL^{N+k}$,
  and we proceed by (decreasing) induction on $j$ as follows. 
  The path $\Psi_{j-1}(\ga)$ is obtained by applying $\Sigma_j$ to $\Psi_j(\ga)$,
  as illustrated in Figure \ref{fig:path_transfo}.
  Possible increases in column heights are marked in blue.
  Since the black vertices in Figure \ref{fig:path_transfo} 
  are contained in $\bL^{j}$, the blue ones are contained in $\bL^{j-1}$.
  This concludes the induction. 

  Therefore, $U(\ga) = \Psi_k(\ga)$ is contained in $\bL^{k}$, and hence
  inequality \eqref{neighbours_invade} follows once we have proved that 
  \begin{align} 
    L^{k}_{n+1} \leq 
    \max \{ H^k_{n-i} - |i| + 1 : i \geq -1 \}. 
    \label{cover3}
  \end{align}
  This we shall do by observing that the sequence $(\bL^{j})$ is, in a certain sense, 
  additive with respect to its initial state. 
  We think of $\bL^{N+k}$ as a union of 
  columns, each of whose evolutions may be followed individually. 
  
  \begin{figure}[htb]
    \begin{center}
      \cpsfrag{si}{}
      \cpsfrag{sii}{}
      \cpsfrag{siii}{}
      \includegraphics[width=1.0\textwidth]{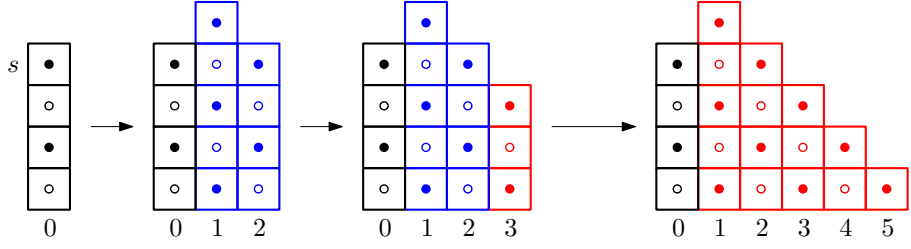}
    \end{center}
    \caption{An illustration of the sequence $\wt \bL(0,s)^j$ beginning with the
initial column $\wt \bL(0,s)^{N+k} = \De(0,s)$.
 This column is unchanged up to and including  $j=s$, and then it evolves as illustrated.}
    \label{fig:L_tilde}
  \end{figure}
  
  Let $r,s\in\ZZ$ be such that $s \le N+k$ and $r+s$ is even, so that $v_{r,s}$ is a primal vertex.
  Let $\De(r,s)$  be the range comprising a single column of height $s$ at position $r$.
  Consider the sequence $(\wt{\bL}(r,s)^{j})$ with the same dynamics as $(\bL^j)$ but with 
  initial state $\wt{\bL}(r,s)^{N+k} = \De(r,s)$.
  The evolution of $\wt{\bL}(r,s)^{j}$ is illustrated in Figure \ref{fig:L_tilde}. 
  We have that $\wt \bL(r,s)^j = \De(r,s)$ for $N+k \geq j \geq s$
  and, for $s > j \geq k$,
  \begin{equation*}
     \wt{L}(r,s)_m^j = 
    \begin{cases}
      - \infty       \quad & \text{if } m < r \text{ or } m > r + s  - j + 1,\\
      s             \quad & \text{if } m = r,\\
      s- (m - r) + 2 \quad & \text{if } r < m \leq r+s  - j + 1.
    \end{cases}
  \end{equation*}
  The range $\bL^j$ is obtained by combining the contributions of the columns of  $\bH^k$, 
  in that
  \begin{equation}
    \bL^j = \max \{ \wt{\bL}(r,H^k_r)^j : r \in \ZZ \}, 
\qquad N+k \geq j \geq k. \label{col_contrib}
  \end{equation}
  A rearrangement of the above with $j=k$ implies \eqref{cover3};
  \eqref{neighbours_invade} follows 
  by extending the maximum in \eqref{cover3} over $i \in \ZZ$.
    
  Let $n \in \ZZ$ be such that
  \begin{align} 
    H^k_{n} + 1 \leq \max \bigl\{ H^k_{n - i} - |i| + 1 : i \in \ZZ\setminus \{0\} \bigl\}. 
    \label{first_case}
  \end{align}
  Then \eqref{neighbours_invade} implies $H^{k+1}_{n} \leq C(\bH^k)_n $,
  whence \eqref{one_step_growth} holds for this particular value of $n$.
  
  It remains to prove \eqref{one_step_growth} when \eqref{first_case} fails.
  Assume $n$ does not satisfy \eqref{first_case}, 
  so that \eqref{neighbours_invade}
  implies $H_n^{k+1} \le H_n^k +1$.
  We shall prove that
  \begin{align}
    H^{k+1}_{n} \leq H^k_{n} + Y_n, \label{sc_domin}
  \end{align}
  where the $Y_n$ are independent Bernoulli random variables 
  with respective parameters 
 \begin{equation}\label{G707}
  \eta_k(n) := 
  \frac{ p_{\pi-\bef + \alpha_n}p_{\pi-\b_k+ \bef} }  { p_{\bef - \alpha_n} p_{\b_k- \bef} } ,
  \end{equation}
  (with $p_\th$  given in \eqref{G101}),  and which are independent of $\omega$.
  
  Let $l = H^k_{n}$. 
  We first analyse the action of 
  $ \Psi_l = \Sigma_{l+1} \circ \dots \circ \Sigma_{N+k}$, 
  and then that of $\Sigma_l$. 

  The vertex $v_{n, l}$ is necessarily primal. 
  Since \eqref{first_case} fails,
  \begin{align*}
    H_{n-i}^k \leq l + |i| - 1 , \qquad i \in \ZZ\setminus \{0\}. 
  \end{align*}
  Since each $v_{n-i, l + |i| - 1}$ is a dual vertex, 
  we have the strengthened inequality
  \begin{align}
    H_{n-i}^k \leq l + |i| - 2 , \qquad i \in \ZZ\setminus \{0\}. \label{second_case2}
  \end{align}
  See Figure \ref{fig:second_case} for an illustration of the 
  environment around $v_{n,l}$.
  
  \begin{figure}[htb]
    \begin{center}
      \cpsfrag{c}{$\col_n$}
      \cpsfrag{cn}{$\col_{n+1}$}
      \includegraphics[width=0.35\textwidth]{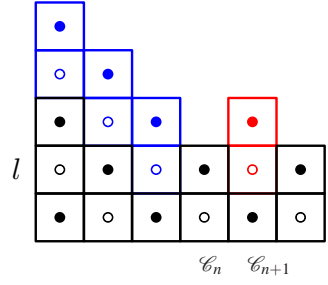}
    \end{center}
    \caption{The environment around $v_{n,l}$.
      By \eqref{second_case2}, the black blocks contain $\bH^k$. 
      The range $\bL^l$, and hence the path $\Psi_l(\ga)$,
      is contained in the aggregate range shown.
      The height in $\col_{n+1}$ increases  only if
      the red block appears when applying $\Sigma_l$.}
    \label{fig:second_case}
  \end{figure}

  By \eqref{L_induction}, and 
  \eqref{second_case2} substituted into \eqref{col_contrib},
  \begin{align}
    h(\Psi_l(\ga)\cap \col_{n-i})\leq L^l_{n-i} & \leq l + i, \qquad i \geq -1.
    \label{second_case3}
  \end{align}
  Note that $\Sigma_l$ is the final track-exchange with the potential to
  add vertices to the path at height $l+1$.
  Hence, $H_n^{k+1} = l + 1$ only if $v_{n+1,l+1}$ is 
contained in $\Psi_{l-1}(\ga)$, or, equivalently,
 only if the height in $\col_{n+1}$ increases to $l+1$ 
  when applying $\Sigma_l$ to $\Psi_l(\ga)$.

  By \eqref{second_case3} with $i = 0,1$, 
  the only cases in which this may happen are 
  those of the third and sixth lines of Figure \ref{fig:path_transfo}
  (with $v_{n,l}$  the black vertex).
  (See Figure \ref{fig:sandpile} for a more detailed illustration of the third case.)
  Moreover, the height in $\col_{n+1}$ increases only if the secondary outcome occurs. 
  In both cases, the secondary outcome occurs with probability $\eta_k(n)$
  if the edge $e = \lan v_{n, l}, v_{n+1, l+1}\ran$ is open, 
  and does not occur if $e$ is closed. 
  We therefore provide ourselves with a Bernoulli random variable
  $Y_n$, with parameter $\eta_k(n)$, for use in the former situation.
  We have that $H^{k+1}_{n} = H^k_{n} + 1$ only if $Y_n = 1$,
  and  \eqref{sc_domin} follows. 

  \begin{figure}[htb]
    \begin{center}
      \cpsfrag{cn}{$\col_n$}
      \cpsfrag{cni}{$\col_{n+1}$}
      \cpsfrag{a}{$\b_k -\a_n$}
      \cpsfrag{bf}{$\bef - \a_n$}
      \includegraphics[width=0.8\textwidth]{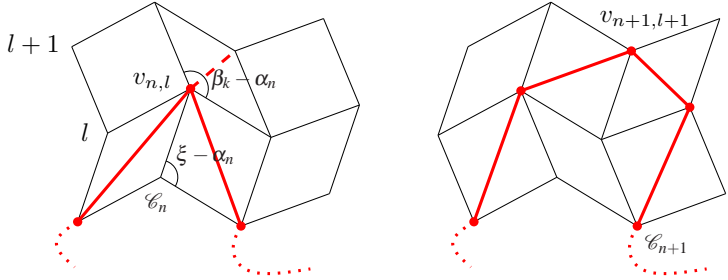}
    \end{center}
    \caption{The third case of Figure \ref{fig:path_transfo}. 
      If the dashed edge in the initial configuration (left) is open
      then, with probability $\eta_k(n)$,
      the resulting configuration is that on the right side.}
    \label{fig:sandpile}
  \end{figure}

  Let $A=\xi-\a_n$ and $B=\b_k-\xi$.  By \eqref{G707},
  \begin{align*}
    \eta_k(n) &= \frac{p_{\pi-A}p_{\pi-B}}{p_A p_B}
    = \frac{\sin(\frac13A)\sin(\frac13 B)}{\sin(\frac13[\pi-A]) \sin(\frac13[\pi-B])}\\
    &= \frac{\cos(\frac13[A-B]) - \cos(\frac13[A+B])}
    {\cos(\frac13[A-B]) - \cos(\frac13[2\pi-A-B])} =: g(A,B).
  \end{align*}
  By assumption, $B>0$, and so by \eqref{bounded_angles},
  \begin{equation}\label{G1234}
\eps\le A\le A+B \le \pi-\eps.
\end{equation}
 There exists $c(\eps)>0$ such that, subject to \eqref{G1234},
  \begin{equation*}
    \cos(\tfrac13[A-B]) \geq \cos(\tfrac13[A + B]) \geq \cos(\tfrac13[2\pi-A-B]) + c(\eps).
  \end{equation*}
  Therefore,
  $$
  \eta:= \sup\bigl\{g(A,B): \eps\le A \le A+B \le\pi-\eps\bigr\}
  $$
  satisfies $\eta<1$, and this concludes the proof of the lemma.
\end{proof}

\begin{proof}[Proof of Lemma \ref{D_control}]  
  We sketch this. 
  Since $B(\rho N, N) \subseteq D^0$, we have that $\ga^0 \subseteq D^0$. 
  It suffices to show that, for  $0 \leq k < \la N$ and 
  $\ga$ an open path in $D^k$, 
  $U_k$ does not break $\ga$ and $U_k(\ga) \subseteq D^{k+1}$.
  
  By considering the individual track-exchanges of which $U_k$ is composed,
  it may be seen that $\Psi_j(\ga)$ is an open path contained in $D^{k+1}$ for all $j$ 
  (with $\Psi_j=\Si_{j+1} \circ \Psi_{j+1}$ as  in the last proof).
  In considering how $\Psi_{j}(\g)$ is obtained from $\Psi_{j+1}(\ga)$,
  it is useful to inspect  the different cases of Figure \ref{fig:path_transfo}, 
  and in particular those involving blue points. 
  The path may be displaced laterally and, during the 
  sequential application of track-exchanges, 
  the drift may be extended laterally as it is propagated downwards. 
 The shapes of the $D^{i}$ have been chosen in such a way that 
 $\Psi_j(\ga)$ is contained in $D^{k+1}$ for all $j$.
  The argument is valid regardless of the direction of $\Sigma_j$. 
\end{proof}

\subsection{Vertical crossings}\label{sec:vertical}

\begin{prop}\label{vertical_transport}
  Let $\delta = \frac{1}{2} p_{\pi-\eps}^4 \in (0,\tfrac12)$.
  There exists $c_N =c_N(\de) >0$ satisfying $c_N\to 1$ as $N \to \infty$ 
  such that
  \begin{align*}
    \PP_{\balpha, \bbeta}\bigl(\Cv[B( 4N ,\delta N)]\bigr)  \geq 
    c_N \PP_{\balpha, \bef} 
    \bigl(\Cv[B (N, N)]\bigr), \qquad N\in\NN.
  \end{align*}
\end{prop}

\begin{proof}
  The notation of Section \ref{sec:outline} will be used.
  Let $N \in \NN$, and, as in Section \ref{sec:horizontal},
we will work with a vertical strip of width $2M$
  of the graph $G = G_{\balpha, \wt\bbeta}$.
  %
  For this proof we take $M = 5N$.
  For $k \in \{0,1, \dots, N -1\}$, set
  \begin{equation}\label{G223}
    V_k =  \Sigma_{2N - k} \circ \dots \circ \Sigma_{N - k + 1}. 
  \end{equation}
  The map  $V_k$ exchanges the track at level $N - k$ with the $N$ tracks immediately above it. 
  The sequential action of $V_0, V_1,\dots, V_{N-1}$ moves the regular block upwards track by track,
  see Figure \ref{fig:vertical_sigmas}. 

  \begin{figure}[htb]
    \begin{center}
      \includegraphics[width=0.8\textwidth]{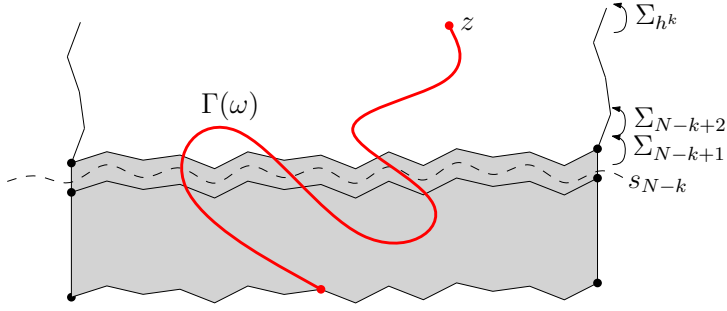}
    \end{center}
    \caption{The transformation $V_k$ moves 
      $s_{N-k}$  upwards by $N$ units.}
    \label{fig:vertical_sigmas}
  \end{figure}
  
  Let $\omega^0$ be a configuration on $G^0:= G_{\balpha, \wt{\bbeta}}$
  chosen according to its canonical measure $\PP_{\balpha, \wt\bbeta}$,
  and let
  \begin{align*}
    G^k &= V_{k-1}\circ \dots \circ V_{0}(G_{\balpha, \wt{\bbeta}}) , \\
    \omega^k &= V_{k-1}\circ \dots \circ V_{0}(\omega^0) ,\\
    D^k &= \bigl\{ v_{x,y} \in (G^k)^\di: |x| \le  N +2k +y,
    \  0 \leq y \leq 2N \bigr\} ,\\
    h^k &= \sup \bigl\{ h \leq N: \exists x_1, x_2 \in \ZZ \text{ with } 
    v_{x_1,0} \xlra{D^k, \om^k} v_{x_2,h} \bigr\} .
  \end{align*}
  That is, $h^k$ is the greatest height
  of an open path of $G^k$ starting in $\{v_{x,0}: x \in \ZZ\}$ 
  and lying in the trapezium $D^k$.
  The law $\PP$ of the sequence  $(\omega^k: k \in \ZZO)$
  is a combination of the law of $\omega^0$
  with those of the \stt s comprising the $V_k$. 
  
  The box $B(N, N)$ is contained in $D^0$,
  and lies entirely  in the regular block of $G^0$. 
  The box $B(4N ,\delta N)$ contains 
  the part of $D^N$ between heights $0$ and $\de N$, 
  and lies in the irregular section of $G^N$
  ($\de<\frac12$ is given in the proposition). 
  Therefore, it suffices to prove the existence of $c_N  =c_N(\de) > 0$ such that $c_N \to 1$ and
  \begin{equation}\label{G1}
    \PP ( h^N \geq \delta N ) \geq  c_N \PP ( h^0 \geq  N ). 
  \end{equation}
  The remainder of this section is devoted to  the proof of \eqref{G1}.

Let $(\De_i: i \in \NN)$ be independent random variables 
with common distribution
\begin{equation}\label{G6014}
P(\De= 0) = 2\de,\quad P(\De=-1)=1-2\de.
\end{equation}
The $\De_i$ are independent of all random  variables used in the construction
of the percolation processes of this paper. 
We set
\begin{equation} \label{G6013}
H^k= H^0 + \sum_{i=1}^k \De_i,
\end{equation}
where $H^0$ is an independent copy of $h^0$, independent of the $\De_i$.
The inequalities $\lest$, $\gest$ refer to stochastic ordering.

 \begin{lemma}\label{G-lem1}
    Let $0 \le k < N$.
  If $h^k \gest H^k$ then $h^{k+1} \gest H^{k+1}$.
\end{lemma}
  
Inequality \eqref{G1} is deduced as follows.
Evidently, $h^0 \gest H^0$ and, by Lemma \ref{G-lem1},
$h^N \gest H^N$.
In particular,
 $$
  \PP ( h^N \geq  \delta N )  \geq P( H^N \geq  \delta N) .
  $$ 
  Since  $h^0$ and $H^0$ have the same distribution,
  \begin{align*}
    \frac{\PP(h^N \geq \delta N)}{\PP(h^0 \geq N)} &\ge
    \frac{P( H^N \geq  \delta N)}{P(H^0 \geq N)}\\
    &\geq P(H^N \geq \delta N \mid H^0 \geq N)
    =: c_N(\delta).
  \end{align*}
  Now, $(H^k)$ is a random walk with mean step-size $ 2 \delta -1$.
  By the law of large numbers,  $c_N \to 1$ as $N\to\oo$.
  In addition, $c_N >0$, and \eqref{G1} follows.
\end{proof}

\begin{proof}[Proof of Lemma \ref{G-lem1}]
  Let $0\le k < N$.  
  We apply $V_k$ to $G^k$,
  and study the effects of the track-exchanges in $V_k$. 
  For $N-k \leq j \leq 2N-k$,
  let  $\Psi_j = \Sigma_{j} \circ \dots \circ \Sigma_{N - k +1}$,
  and let
  $D^k_j$ be the subgraph of $\Psi_j(G^k)^\di$
  induced by vertices $v_{x,y}$ 
  with $0 \leq y \leq 2N$ and
  \begin{equation}
    |x| \leq
    \begin{cases}
      N +2k +y +2 &\text{ if } y \leq j, \\
      N +2k +y +1 &\text { if } y = j+1, \\
      N +2k +y    &\text{ if } y > j+1.
    \end{cases}
  \end{equation}
  The $D^k_j$ increase with $j$, and $D^k \subseteq D_{N-k}^k$,
  $D_{2N-k}^k \subseteq D^{k+1}$.
  
  Let $\om^k_j = \Psi_{j}(\om^k)$ and
  \begin{equation*}
    h^k_j = \sup \bigl\{ h \leq N: \exists x_1, x_2 \in \ZZ \text{ with }  
    v_{x_1,0} \xlra{D^k_j, \om^k_j} v_{x_2,h} \bigr\},
  \end{equation*}
  noting that
  \begin{equation}\label{G6011}
    h^k \le h^k_{N-k}, \quad h^{k+1} \ge h^k_{2N-k}.
  \end{equation}
  First, we prove that, 
  for $N-k \leq j < 2N-k$,
  \begin{alignat}{2}
    h^{k}_{j+1} &\geq h^k_j - 1, \label{non_dec0} \\
    h^{k}_{j+1} &\geq h^k_j  \quad&&\text{if } h_j^k \ne j+1, \label{non_dec1} \\
    \PP (h^{k}_{j+1} \geq h \mid h^k_j =h) &\geq 2\de \quad&&\text{if } h=j+1. \label{non_dec2}
  \end{alignat} 
  
  Fix $j$ such that $N-k \leq j < 2N-k$.
  Let $\g$ be an $\om^k_j$-open  path of $\Psi_j(G^k)$, lying  in $D^k_j$, 
  with one endpoint at height $0$ and the other at height $h^k_j$. 
  
  By consideration of Figure \ref{fig:path_transfo},
  $\Si_{j+1}(\ga)$ is a $\om^k_{j+1}$-open path contained in $D^k_{j+1}$.
  The lower endpoint of $\ga$ is not affected by $\Si_{j+1}$.
  The upper endpoint is affected only if it is at height $j+1$,
  in which case its height decreases by at most $1$ (see Figure \ref{fig:path_transfo_ep}).
  This proves \eqref{non_dec0} and  \eqref{non_dec1}, and 
  we turn to  \eqref{non_dec2}.

  Let $\sP_j$ be the set of  paths $\g$ of
  $\Psi_{j}(G^k)$, contained in $D^k_j$,
  with one endpoint at height $0$, the other endpoint at height $h(\g)$ 
  and all other vertices with heights between $1$ and $h(\g)-1$.     

       
  We perform a preliminary computation.
  Let  $\ga,\ga' \in \sP_j$. 
  We write $\gamma' < \gamma$ if $\gamma' \ne \gamma$, $h(\ga')=h(\ga)$, and
  $\gamma'$ contains no edge strictly to the right of $\gamma$
  within $\{v_{x,y}: |x| \le M,\ 0\le y\le h(\ga)\}$.     
  Note that 
  $$
  h_j^k = \sup\bigl\{h(\ga): \gamma \in \sP_j,\, \gamma \text{ is $\om^k_j$-open}\bigr\},
  $$
  and denote by $\Ga=\Ga(\om^k_j)$ the $\om^k_j$-open
  path of $\sP_j$ that is the minimal element of 
  $\{\ga\in\sP_j: h(\ga) = h_j^k,\ \ga \text{ is $\om^k_j$-open}\}$
  with respect to the  order $<$.
  Thus, $\Ga$ is the leftmost path of $\sP_j$ reaching height $h_j^k$.
    
  We have that
  \begin{equation}\label{G225}
    \{\Ga(\om^k_j) = \gamma\} = \{\gamma \text{ is $\om^k_j$-open}\}
    \cap N_\ga, \qquad \g \in \sP_j,
  \end{equation}
  where $N_\ga$ is the decreasing event that:
  \begin{letlist}
  \item there is no  $\ga' \in \sP_j$ 
    with $h(\ga')>h(\ga)$, all of whose edges 
    not belonging to $\ga$ are $\om^k_j$-open,
  \item there is no $\gamma' < \ga$ with $h(\ga')=h(\ga)$,
    all of whose edges not belonging to $\gamma$ are $\om^k_j$-open.
  \end{letlist}
  Note that $N_\ga$ is independent of the event $\{\gamma \text{ is $\om^k_j$-open}\}$.

  Let $F$ be a set of edges of $\Psi_j(G^k)$,
  disjoint from $\ga$, and let $C_F$ be the event that every edge in $F$ is $\om^k_j$\emph{-closed}. 
  Let $\PP^k_j$ denote
  the marginal law of $\om^k_j$, and  $p_e$ the
  edge-probability of the edge $e$ of $\Psi_j(G^k)$. 
   By \eqref{G225} and the Harris--FKG inequality,
  \begin{align}
    \PP(C_F \mid \Ga = \ga) &= \frac{\PP^k_j(\Ga=\ga \mid C_F)}{\PP^k_j(\Ga=\ga)}
    {\PP^k_j(C_F)}\nonumber\\
    &= \frac{\PP^k_j(N_\ga \mid C_F)}{\PP^k_j(N_\ga)} \PP^k_j(C_F) \nonumber\\
    &\ge \PP^k_j(C_F) = \prod_{f\in F} (1-p_f),\label{G224}
  \end{align}
  where we have
  extended the domain of $\PP$ to include the intermediate subsequence of 
  $\om^k = \om_{N-k}^k,\om_{N-k+1}^k,\dots,\om_{2N-k}^k=\om^{k+1}$. 
  
  Let $\ga\in\sP_j$ with $h(\ga) = j+1$ and suppose $\Ga(\om^k_j) = \gamma$. 
  Without loss of generality, we may suppose  that 
  $\Sigma_{j+1}$, applied to $\Psi_j(G^k)$, 
  goes from left to right; a similar argument holds otherwise. 
    
  Let $z = v_{x, j+1}$ denote the upper endpoint of $\gamma$ 
  and let $z'$ denote the other endpoint of the unique edge of 
  $\gamma$ leading to $z$. 
  Either $z' = v_{x+1, j}$ or $z' = v_{x-1, j}$. 
  In the second case, it is automatic as in Figure \ref{fig:path_transfo_ep}
  that $h(\Si_{j+1}(\ga)) \geq j+1$.
    
  Assume that $z'=v_{x+1,j}$, as illustrated in
  Figure \ref{fig:chance_of_non_decrease}. Let $F=\{e_1,e_2,e_3,e_4\}$
  where  
  \begin{alignat*} {2}
    e_1 &= \lan v_{x,j+1},v_{x-1,j+2}\ran, &\quad  e_2 &= \lan v_{x-1,j+2},v_{x-2,j+1}\ran, \\
    e_3 &= \lan v_{x-2,j+1},v_{x-1,j}\ran, &\quad    e_4 &= \lan v_{x-1,j},v_{x,j+1}\ran,
  \end{alignat*}
  are the edges of the face of $\Psi_{j}(G^k)$ to the left of $z$. 
  By definition of $\sP_j$, $F$ is  disjoint from $\ga$. 
  By studying the three relevant \stt s contributing to $\Sigma_{j+1}$
  as illustrated in Figure \ref{fig:chance_of_non_decrease}, 
  we find as in Figure \ref{fig:simple_transformation_coupling} that
  \begin{align*}
    \PP\bigl(h(\Si_{j+1}(\ga)) \ge j+1 \bigmid \Ga=\ga \bigr) &\geq
    \frac{p_{e_1} p_{e_4}}{(1 - p_{e_1})(1 - p_{e_4})}
    \PP(C_F \mid \Ga=\ga) 
    \\
    &\ge  \frac{p_{e_1} p_{e_4}}{(1 - p_{e_1})(1 - p_{e_4})}
    \PP(C_F),
  \end{align*}
  by \eqref{G224}.
  
  \begin{figure}[htb]
    \begin{center}
      \includegraphics[width=1.0\textwidth]{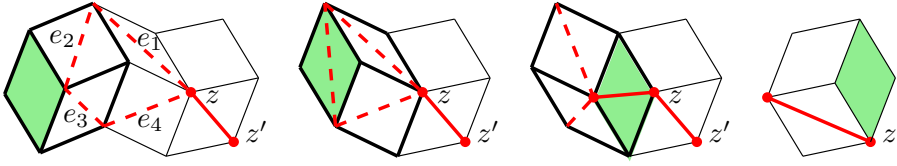}
    \end{center}
    \caption{Three \stt s contributing to $\Sigma_{j+1}$,
      from left to right.
      The dashed edges are closed, the bold edges are open. 
      The first and last passages occur with probability $1$, and 
      the second with probability 
      ${p_{e_1} p_{e_4}}/{(1 - p_{e_1})(1 - p_{e_4})}$.
    }
    \label{fig:chance_of_non_decrease}
  \end{figure}
  
  In summary, we have that
  \begin{align}
    \PP(h^{k}_{j+1} \ge h^k_j \mid \Ga=\ga )
    & \geq \frac{p_{e_1} p_{e_4}}{(1 - p_{e_1})(1 - p_{e_4})} 
    \prod_{f\in F} (1-p_f) \nonumber \\
    &= p_{e_1} p_{e_4}(1 - p_{e_2})(1 - p_{e_3}) \nonumber\\
    &\geq p_{\pi-\eps}^4  = 2\delta, \label{cond_non_decrease}
  \end{align}
  by \eqref{G226}. The proof of \eqref{non_dec2} is complete.

  It remains to show that \eqref{non_dec0}--\eqref{non_dec2} imply the lemma. 
  This may seem obvious, but, since the sequence $(h^k_j)_j$ is not Markovian, 
some technical details are needed.

  Suppose $h^k \gest H^k$.   We shall bound (stochastically) 
the sequence $(h_j^k)_j$ by a Markov chain, as follows.
  Let $(X_j: j=N-k, \dots, 2N-k)$
  be an inhomogeneous Markov chain taking values in $\ZZO$, with transition probabilities given by
  \begin{gather*}
    X_{j+1}=  X_j  \quad \text{if } X_j \ne j+1,\\
    P(X_{j+1}=x \mid X_j = j+1) = \begin{cases}
      2\de &\text{ if } x = j+1,\\
      1-2\de &\text{ if } x = j.
    \end{cases}
  \end{gather*}
  By \eqref{non_dec0}--\eqref{non_dec2}, for all $j$,
  $$
  \PP(h_{j+1}^k \ge x \mid h_j^k = y)
  \ge P(X_{j+1} \ge x \mid X_j = z), \quad x,y,z\in \ZZO,\ z \le y.
  $$
  Let $X_{N-k}= H^k$. By the induction hypothesis,
  $H^k \lest h^k \le h_{N-k}^k$, whence  $X_{2N-k} \lest h_{2N-k}^k$
  as in \cite[Sect.\ 3.3]{GM1} (by \cite[Lemma 3.7]{GM1}, iterated).
  Moreover $X_{2N-k} - X_{N-k} =_{\text{\rm st}}  \De_{k+1}$.
  Therefore,  
  \begin{align*}
    h^{k+1}  \geq h^k_{2N-k} \gest X_{2N-k} =_{\text{\rm st}} H^{k} + \De_{k+1} = H^{k+1},
  \end{align*}
  as claimed.

%
\end{proof}

\section{Proof of Theorem \ref{main}: The general case}\label{sec:proof2}

Let $G\in\sG(\eps,I)$. 
By \SGP$(I)$, there exist two families $(s_j:j\in\ZZ)$ and $(t_i:i\in\ZZ)$
of tracks forming a square grid of $G$. 
A \stt\ is said to act `between $s_k$ and $s_0$' 
if the three faces of $G^\di$ on which it acts are between $s_k$ and $s_0$.
(Recall from  Section \ref{sec:bxp_G} that such faces may belong to 
$s_k$ but not to $s_0$).
A path is said to be between $s_0$ and $s_k$ 
if it comprises  only edges between $s_0$ and $s_k$
(that is, edges belonging to faces between $s_0$ and $s_k$).
A vertex of $G^\di$ is said to be \lq just below' $s_0$ 
if it is adjacent to $s_0$ and between $s_{-1}$ and $s_0$. 

Let $E_N = E_N(G)$ be the event that there exists an open path $\ga$ on $G$ such that:
\begin{letlist}
\item $\g$ is between $s_0$ and $s_N$,
\item the endpoints of $\g$ are just below $s_0$,
\item one endpoint is between $t_{-2N}$ and $t_{- N}$ 
  and the other between $t_{N}$ and $t_{2N}$.
\end{letlist}
There is no further condition on the horizontal extent of $\ga$. 
We claim that there exists $\delta = \delta(\eps,I) > 0$, 
independent of $G$ and $N$, such that
\begin{equation}\label{G620}
  \PP_G (E_N) \geq \delta, \qquad N \ge 1.
\end{equation}
Since $\g$ contains a horizontal crossing of the domain
$\Dom=\Dom(t_{-N}, t_N; s_0, s_N)$, \eqref{G620} implies
$$
\PP_G\bigl[\Ch(t_{-N}, t_N; s_0, s_N)\bigr] \ge \de.
$$
Since $\de$ depends only on $\eps$ and $I$, the corresponding inequality
holds for crossings of translations of $\Dom$, and also with
the roles of the $(s_j)$ and $(t_i)$ reversed.
By Proposition \ref{grid_bxp}, the claim of the theorem 
(together with the stronger statement \eqref{G2231}) follows from \eqref{G620},
and we turn to its proof. 

The method is as follows. 
Consider the graph $G$ between $s_0$ and $s_N$. 
By making a finite sequence of \stt s between $s_N$ and $s_0$, 
we shall move the $s_j$ downwards
in such a way that the section of the resulting graph, 
lying both between $t_{-2N}$ and $t_{2N}$ and between  
the images of $s_0$ and $s_N$, 
forms a box of an isoradial square lattice.
By Corollary \ref{bxp_square}, this box is crossed horizontally 
with probability bounded away from $0$.
The above \stt s are then reversed to obtain a horizontal
crossing of $\Dom$ in the original graph $G$.

Since a finite sequence of \stt s changes $G$ at only finitely many places,
we may retain the track-notation $s_j$, $t_i$ throughout their application.  
We say $s_j,s_{j+1}, \dots, s_{j+k}$ are \emph{adjacent} \emph{between $t_{N_1}$ and  $t_{N_2}$}
if there exists no track-intersection in  the domain $\Dom(t_{N_1}, t_{N_2};s_j, s_{j+k})$
except those on $s_j,s_{j+1}, \dots, s_{j+k}$.
The proof of the next lemma is deferred until later in this section. 
  
\begin{lemma}\label{grid_slide}
  There exists a finite sequence  $(T_k : 1 \leq k \leq K)$ of star--triangle transformations, 
  each acting between $s_N$ and $s_0$, such that, in $T_K\circ \dots \circ T_1(G)$, 
  the tracks $s_0, \dots, s_{N}$ are adjacent between $t_{-2N}$ and $t_{2N}$. 
\end{lemma}
  
Let $(T_k : 1 \leq k  \leq K)$ be given thus, 
and write $G^0=G$ and $G^k = T_k \circ \dots \circ T_1(G^0)$.
Let $S_k$ be the inverse transformation of $T_k$, 
as in Section \ref{sec:stt0}, so that $S_k (G^k)  = G^{k-1}$. 
Since the track notation is retained for each $G^k$, 
the event $E_N$ is defined on each such graph.
By a careful analysis of its action, we may see that $S_k$
preserves $E_N$ for $k = K, K-1,\dots, 1$. 
The details of the argument (which requires no estimate
of the value of $K$)
are provided in the next paragraph.

Let $1 \leq k \leq K$ and 
let $\g$ be an open path of $G^k$  satisfying (a)--(c) above.
Since $T_k$ acts between $s_N$ and $s_0$, 
it does not move $s_0$, and $S_k (\g)$ has the same endpoints as $\g$. 
Moreover the three faces of $(G^k)^\di$ on which $S_k$ acts 
are either all strictly below $s_N$, or two of them are part of $s_N$ and the third is above. 
In the first case $S_k(\g)$ may differ from $\g$ 
but is still contained between $s_0$ and $s_N$;
in the second case $S_k$ does not influence $\g$. 
In conclusion, $S_k(\g)$ is an open path on $G^{k-1}$ that satisfies (a)--(c).

Since the canonical  measure is conserved under a \stt, the remark above implies 
\begin{equation}\label{grid_transport}
  \PP_{G}(E_N) \geq \PP_{G^K}(E_N).
\end{equation}
 It remains to prove a lower bound for $\PP_{G^K}(E_N)$. 

Write $(r_i : i \in \ZZ)$ for 
the sequence of all tracks other than the $s_j$, 
indexed and oriented according to their intersections with $s_0$, with $r_0 = t_0$,
and including the $t_i$ in increasing order. 
Let $\b_j$ be the transverse angle of $s_j$, 
and $\pi+\alpha_i$ that of $r_i$. 
Since each $r_i$ intersects each $s_j$, 
the vectors $\balpha = (\alpha_i : i\in \ZZ)$,
$\bbeta = (\beta_j:j\in \ZZ)$ satisfy
\eqref{bounded_angles}, and hence $G_{\balpha, \bbeta}$ 
is an isoradial square lattice satisfying \BAC$(\eps)$.
By Corollary \ref{bxp_square},  there exists $\de'=\de'(\eps)>0$ such that
$G_{\balpha, \bbeta}$  satisfies the \bxp\ \BXP$(\de')$. 

The track-system of $G^K$ inside $\Dom(t_{-2N}, t_{2N}; s_0, s_N)$ is isomorphic
to a rectangle of $\ZZ^2$, and comprises the
horizontal tracks $s_0,s_1,\dots,s_N$, crossed in order
by those $r_i$ between (and including) $t_{-2N}$ and $t_{2N}$.
Thus, $G^K$ agrees with $G_{\balpha,\bbeta}$ inside this domain.

Consider the following boxes of $(G^K)^\di$:
\begin{align*}
  V_1 &= \Dom(t_{-2N}, t_{-N};s_0, s_N),\\
  V_2 &= \Dom(t_{N} , t_{2N}; s_0, s_N ),\\
  H   &= \Dom(t_{-2N}, t_{2N}; s_0, s_N).
\end{align*}
By the Harris--FKG inequality,
\begin{align}
  \PP_{G^K}(E_N) & \ge \PP_{G^K}\bigl[\Cv(V_1) \cap \Cv(V_2) \cap \Ch(H)\bigr]\nonumber\\
  &\ge \PP_{G^K}[\Cv(V_1)] \PP_{G^K}[\Cv(V_2)]\PP_{G^K}[\Ch(H)].\label{G625}
\end{align}  
The boxes $V_1$, $V_2$ in $(G^K)^\di$ may be regarded as boxes in
$G_{\balpha,\bbeta}^\di$, and
have height $N$ and width at least $N$.
Similarly, the box $H$ has height $N$ and width at most $4I N $.
By \BXP$(\de')$ and \eqref{G625}, there exists $\de= \delta(\eps, I) > 0 $ such that 
$\PP_{G^K}(E_N) \geq  \de$, 
and \eqref{grid_transport} is proved. 

\begin{proof}[Proof of Lemma \ref{grid_slide}]
  We shall prove 
  the existence of a finite sequence  $(T_k : 1 \leq k \leq K)$ of star--triangle transformations, 
  each acting between $s_1$ and $s_0$, such that, in $T_K\circ \dots \circ T_1(G)$, 
  the tracks $s_0, s_1$ are adjacent between $t_{-2N}$ and $t_{2N}$.
  The general claim follows by iteration.  

  In this proof we work with the graph $G$ only through its track-set $\sT$. 
  Tracks will be viewed as arcs in $\RR^2$.
  A \emph{point} of $\sT$ is the intersection of two tracks, and we 
  write $\sP$ for the set of points. 

  Let $\sN$ be the set of tracks that are not parallel to $s_0$.
  Any $r \in \sN$ intersects both $s_0$ and $s_1$ exactly once, and 
  we orient such $r$ in the direction from its intersection with $s_0$ to that with $s_1$. 
  
  An oriented path $\gamma$ on the track-set $\sT$ is called \emph{increasing} 
  if it uses only tracks in $\sN$
  and it conforms to their orientations. 
  For points $y_1,y_2\in \sP$, 
  we write $y_1\ge y_2$
  if there exists an increasing path $\gamma$ from $y_2$ to $y_1$.
  By the properties of $\sT$ given in Section \ref{sec:track}, 
  the  relation $\ge$ is reflexive, antisymmetric, and transitive, 
  and is thus a partial order on $\sP$. 
  
  Let $\sR_k$ be the closed region of $\RR^2$ delimited by
  $t_{-k}$, $t_k$, $s_0$, $s_1$, illustrated in Figure \ref{fig:grid_slide}.
  A  point $y \in \sP$ is coloured \emph{\prob} if it is strictly between $s_0$ and $s_1$, 
  and in addition $y \ge y'$ for some $y'$ in $\sR := \sR_{2N}$ or on its boundary. 
  In particular, any point in the interior of $\sR$  
  or of its left/right boundaries is \prob.  
  We shall see that the \prob\ points are precisely those
  to be `moved' above $s_1$ by the \stt s $T_k$.
  
  \begin{figure}[htb]
    \centering
    \includegraphics[width=0.8\textwidth]{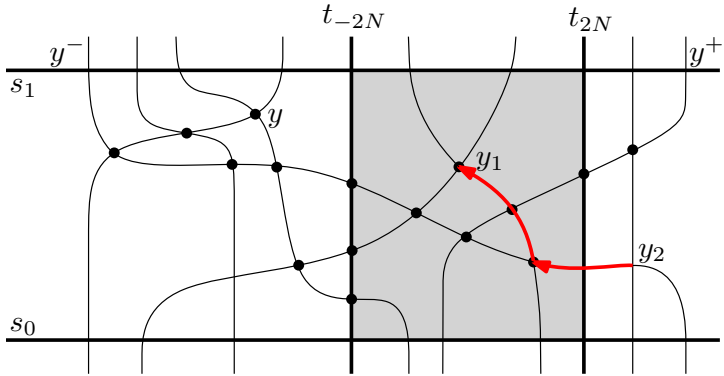}
    \caption{The \prob\ points are indicated. 
      The path $\gamma$ from $y_2$ to $y_1$ is drawn in red. 
      The points $y$ and $y_1$ are  maximal, 
      and are not comparable.
      The region $\sR$ is shaded.}
    \label{fig:grid_slide}
  \end{figure} 

We prove first that the number $B$ of black points is finite.
By \BAC$(\eps)$, the number of tracks intersecting $\sR$ is finite.
  Let $y^+$ (\resp, $y^-$) be the rightmost (\resp, leftmost) 
point on $s_1$ that is the intersection of $s_1$
  with a track  $r$ that intersects $\sR$. 
 We claim that 
  \begin{equation}\label{finite_prob}
    \mbox{if $r\in\sT$ has a \prob\ point, 
      it intersects $s_1$ between $y^-$ and $y^+$.}
  \end{equation} 
  Assume \eqref{finite_prob} for the moment.
  Since a \prob\ point is the unique intersection of two tracks, and 
  since  \eqref{finite_prob} implies that there are only finitely many 
  tracks with \prob\ points, we have that $B<\oo$.
  
 We prove \eqref{finite_prob} next.
  Let $y$ be a \prob\ point. 
  If $y \in \sR$, \eqref{finite_prob} follows immediately. 
  Thus we may suppose, without loss of generality,
  that $y$ is strictly to the left of $\sR$.
  There exists an increasing path $\gamma$, 
  starting at a point on the left boundary of $\sR$ and ending at $y$. 
  Take $\gamma$ to be the `highest' such path.
  Let $(r_l : 1 \leq l \leq L)$ be the tracks used by $\gamma$ in order, where $L < \infty$. 
  We will prove by induction that, for $l\ge 1$,
  \begin{equation}\label{G626}
    \mbox{$r_l$ intersects $s_1$ between $y^-$ and $y^+$.}
  \end{equation} 
  Clearly \eqref{G626} holds with $l=1$ since $r_1$ intersects $\sR$.
  
  \begin{figure}[htb]
    \centering
    \includegraphics[width=0.9\textwidth]{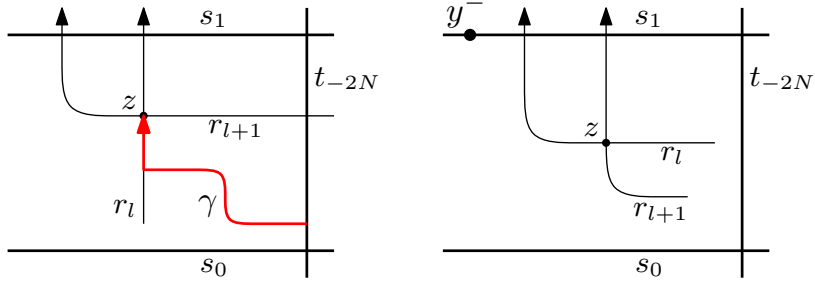}
    \caption{\emph{Left}: The oriented track $r_{l+1}$ crosses $r_l$ from right to left,
      in contradiction of the choice of $\g$ as highest.
    \emph{Right}:  The track $r_{l+1}$ crosses $r_l$ from left to right. }
    \label{fig:Lfinite}
  \end{figure}

  Suppose $1 \le l < L$ and \eqref{G626} holds for $r_l$, and let $z = r_l \cap r_{l+1}$.
  If $r_{l+1}$ intersects $\sR$ (before or after $z$), \eqref{G626} follows trivially. 
  Suppose $r_{l+1}$ does not intersect $\sR$.
  There are two possibilities:
  either $r_{l+1}$ crosses $r_l$ from right to left, or from left to right.
  The first case is easily seen to be impossible, 
  since it contradicts the choice of $\g$ as highest.  
  Hence, $r_{l+1}$ crosses $r_l$ from left to right (see Figure \ref{fig:Lfinite}).
  The part of the oriented track $r_{l+1}$ after $z$ 
  is therefore above the corresponding part of $r_l$. 
  Since $r_{l+1}$ intersects $s_1$ after $z$, 
  and does not intersect $\sR$,
  the intersection of $r_{l+1}$ and $s_1$ lies  between $y^-$ and $y^+$,
  and the induction step is complete. 

  In conclusion $r_L$ intersects $s_1$ between $y^+$ and $y^-$.
  Let $r$ denote the other track containing $y$. 
  By the same reasoning, $r$ intersects $r_L$ from left to right,
  whence it also intersects $s_1$ between $y^+$ and $y^-$.
  This concludes the proof of \eqref{finite_prob}, and we deduce that $B<\oo$.
  
  If $B = 0$,  there is no point in the interior of either $\sR$ or its left/right sides, 
  whence $s_0$, $s_1$ are adjacent  between $t_{-2N}$ and $t_{2N}$. 
  
  Suppose $B \ge 1$.  We will show that $B$ may be reduced by one 
  by a \stt\ acting between $s_0$ and $s_1$, 
  and the claim of the lemma will follow by iteration. 
  
  Since $B<\infty$, there exists a \prob\ point that is maximal in the partial order $\ge$,
  and we pick such a point $y = r_1\cap r_2$. 
  By the maximality of $y$, the tracks $r_1$, $r_2$, $s_1$ form a track-triangle.
  By applying the \stt\ to this track-triangle 
  as in Figure \ref{fig:isoradial_star_triangle}, the point $y$
  is moved above $s_1$, and the number of \prob\ points is decreased. 
  This concludes the proof of Lemma \ref{grid_slide}.
\end{proof}

\section{Arm exponents}\label{sec:arm-exps}

\subsection{Outline of proof}

We recall the isoradial embedding  $G_{0,\pi/2}$ of the 
homogeneous square lattice,
with associated measure denoted $\PP_{0,\pi/2}$. 

\begin{prop}\label{exp_transport}
  Let $k \in \{1,2,4,\dots\}$, $\eps>0$, and $I \in \NN$.
  There exist constants $c_i = c_i(k,\eps,I) > 0$ 
  and $N_0 = N_0(k,\eps,I)\in\NN$ such that, for $N \geq N_0$, 
  $n \geq c_0 N_0$, $G \in \sG(\eps,I)$, and any vertex $u$ of $G^\di$,
  \begin{align*}
    c_1 \PP_{0,\pi/2}[\Arm_k(N, n)] \leq \PP_G[\Arm_k^u(N, n)] \leq c_2 \PP_{0,\pi/2}[\Arm_k(N, n)].
  \end{align*}
\end{prop}

Part (a) of Theorem \ref{main3} is an immediate consequence;
part (b) is discussed in Section \ref{sec:kesten}.
 
Sections \ref{sec:alt_def}--\ref{sec:sep_thm} are devoted to the proof of 
Proposition \ref{exp_transport}.
In Section \ref{sec:alt_def}
is presented a modified definition of the arm-events, 
adapted to the  context of an isoradial graph. 
This is followed by Proposition \ref{exp_equiv}, which asserts
in particular the equivalence of the two types of arm-events.
The proof of Proposition \ref{exp_transport} follows, using the
techniques of the proof of Theorem \ref{main};
the proof for isoradial square lattices is in Section \ref{sec:proof_arm1},
and  for general graphs in Section \ref{sec:proof_arm2}. 
Section \ref{sec:sep_thm} contains the so-called separation theorem,
together with 
the proof of Proposition \ref{exp_equiv}.

For the remainder of this section, $\eps>0$ and 
$I\in\NN$ shall remain fixed. 
Unless otherwise stated, 
constants $c_i > 0$, $N_0\in\NN$ depend only on $\eps$, $I$, 
and on the number $k$ of arms in the event under study.
We use the expression \lq for $n > N$ sufficiently large' 
to mean: for $n \geq c_0 N$ and $N> N_0$.

\subsection{Arm-events}\label{sec:alt_def}

Let $G \in \sG(\eps,I)$, $k \in \{1,2,4,\dots\}$, and let 
$s$ be a track and $u$ be a vertex of $G^\di$, adjacent to $s$. 
For $n \geq N$, 
we define the \lq modified arm-event' $\wt{A}^{u,s}_k(N, n)$ as follows.
For simplicity of notation, we omit explicit reference
to $u$ and $s$ when no ambiguity results, but in such a case we say that
$\wt A_k(N,n)$ is `centred at $u$'.

A certain technical assumption will be useful in Section \ref{sec:proof_arm1}, 
when applying \stt s to isoradial square lattices. Recall the notation $\La_u^\di(n)$ from Section \ref{sec:metrics},
and the constant $\cd$ of \eqref{di_euclid}. 
We shall frequently require
a primal vertex $u$ and dual vertex $u^*$ of $G^\di$ to
satisfy the appropriate inclusion of the following:
\begin{equation}\label{G601}
  C_u \subseteq \La_u^\di(3 \cd^2 n), \qquad C^*_{u^*} \subseteq \La_{u^*}^\di(3 \cd^2 n).
\end{equation}

The modified arm-events $\wt{A}_k(N, n)=\wt{A}^{u,s}_k(N, n)$ are defined thus:
\begin{romlist}
\item For $k = 1$, $\wt{\Arm}_1(N, n)$ is the event that there exist vertices
  $x_1 \in \La_u^\di (N)$ and $y_1 \notin \La_u^\di (n)$, 
  both adjacent to $s$, on the same side of $s$ as $u$
  and satisfying \eqref{G601}, such that
  $x_1 \xlra{} y_1$. 

\item For $k =2$, $\wt{\Arm}_2(N, n)$ is the event that there exist vertices
  $x_1, x_1^* \in \La_u^\di(N) $ and $y_1, y_1^* \notin \La_u^\di(n)$, 
  all adjacent to $s$ and on the same side of $s$ as $u$, such that:
  \begin{letlist}
  \item $x_1$, $x_1^*$ and $y_1$, $y_1^*$ satisfy \eqref{G601}, 
  \item
    $x_1 \xlra{} y_1$ and $x_1^* \dxlra{} y_1^*$. 
  \end{letlist}
  
\item For $k =2j \geq 4$, $\wt{\Arm}_k(N, n)$ is the event that there exist vertices
  $x_1,\dots, x_j \in \La_u^\di(N)$ and $y_1, \dots,y_j \notin \La_u^\di(n)$, 
  all adjacent to $s$ and on the same side of $s$ as $u$, such that:
  \begin{letlist}
  \item each $x_i$ and $y_i$ satisfies \eqref{G601},
  \item
    $x_i \xlra{} y_i$ and $x_i \nxlra{} x_{i'}$ for $i \neq i'$. 
  \end{letlist}
\end{romlist}

The following proposition contains three statements, the third of which
relates the modified arm-events to those of
Section \ref{sec:main}.
All arm-events $A_k$ and $\wt A_k$ considered here 
are centred at the same vertex $u \in G^\di$.
The event $\wt{\Arm}_k(N, n)$ is to be interpreted
in terms of any of the tracks to which $u$ is adjacent.

\begin{prop}\label{exp_equiv}
  There exist constants $c_i > 0$ such that, for $n > N$ sufficiently large,
 \begin{align}
    \PP_G[\Arm_k(N, 2n)]  &\leq \PP_G[\Arm_k(N, n)] \leq c_1 \PP_G[\Arm_k(N, 2n)],  \label{exp_equiv_a}\\
    \PP_G[\Arm_k(N, n)]  &\leq \PP_G[\Arm_k(2N, n)] \leq  c_2  \PP_G[\Arm_k(N, n)], \label{exp_equiv_b}\\
    c_3 \PP_G[\Arm_k(N, n)]  &\leq \PP_G[\wt{\Arm}_k(N, n)] \leq c_4 \PP_G[\Arm_k(N, n)]. \label{exp_equiv_c}
  \end{align}
\end{prop}

By \eqref{exp_equiv_c}, for $n > N$ sufficiently large, 
there exist constants $c_5,c_6 >0$  such that, 
if $u$ is adjacent to the tracks $s$ and $t$,
\[
c_5 \PP_G[\wt{\Arm}_k^{u,s}(N, n)] 
\leq \PP_G[\wt{\Arm}_k^{u,t}(N, n)] 
\leq c_6 \PP_G[\wt{\Arm}_k^{u,s}(N, n)]. 
\]
The proof of Proposition \ref{exp_equiv} is deferred to Section \ref{sec:sep_thm}. 
It relies on the so-called separation theorem, an account of which may be found in that
section. 

\subsection{Proof of Proposition \ref{exp_transport}: Isoradial square lattices}\label{sec:proof_arm1}

Let $G$ be an isoradial square lattice satisfying 
the \bac\ \BAC$(\eps)$, and
let $u$ be a vertex of $G^\di$. As usual, the horizontal tracks  are labelled 
$(s_j: j \in \ZZ)$ and the vertical tracks $(t_i: i \in \ZZ)$.

As explained in Section \ref{sec:iso-sl},
$G = G_{\balpha, \bbeta}$ 
for angle-sequences $\balpha = (\a_i : i \in \ZZ)$, $\bbeta = (\b_j : j \in \ZZ)$
satisfying \eqref{bounded_angles}.
We label $\balpha$ and $\bbeta$ in such a way that $u = v_{0,0}$, 
whence  $u$ is adjacent to $t_0$ and $s_0$ 
(here we do not require $v_{0,0}$ to be primal). 
The horizontal track at level $0$, initially $s_0$, may change its label through track-exchanges.
Let $\bef $ be such that $\balpha$ 
and the constant sequence $(\bef)$ 
satisfy \BAC$(\eps)$, \eqref{bounded_angles}. 
All arm-events in the following are centred at $u=v_{0,0}$.

\begin{lemma}\label{exp_transport1}
  There exist constants $c_i>0$ such that, for $n > N$ sufficiently large,
  \begin{align*}
    c_1 \PP_{\balpha, \bef }[\Arm_k(N, n)] \leq 
    \PP_G[\Arm_k(N, n)] \leq 
    c_2 \PP_{\balpha, \bef }[\Arm_k(N, n)].
  \end{align*}
\end{lemma}
 
\begin{proof}

  Let $N,n \in \NN$ be picked (later) such that $N$ and $n/N$ are large,
  and write $M = \lceil 3 \cd^2 n \rceil$. 
  For $0\le m \le M$, 
  let $G^m$ be the isoradial square lattice with angle-sequences
  $\wt{\balpha} = (\a_i : - 4M \leq i \leq 4M)$ 
  and $\wt{\bbeta}^m$, with 
  \begin{equation}\label{G610}
    \wt{\b}^m_j =
    \begin{cases}
      \bef     &\text{ if } -m \leq j < m , \\
      \b_{j+m} &\text{ if } -(m + M) \leq j < -m,\\
      \b_{j-m} &\text{ if } m \leq j < m + M,\\
      \bef     &\text{ if } j < -(m + M) \text{ or }  j \geq m + M.\\
    \end{cases}
  \end{equation}
  Thus $G^m$ is obtained from $G$ by taking the horizontal
  tracks $s_j$, $-M \le j < M$,
  splitting them with a band of height $2m$, and filling the rest of space with 
  horizontal tracks having transverse angle $\bef$.
  By the choice of $\bef $, each $G^m$ satisfies \BAC$(\eps)$.
  Moreover, inside $\La_u^\di(M)$, 
  $G^0$ is identical to $G$ and $G^M$ is identical to $G_{\balpha, \bef }$. 
  
  For $0 \leq m < M$, let
  \begin{align*} 
    & U_m = 
    ( \Sigma_{m+1} \circ \dots \circ \Sigma_{m + M})  \circ
    ( \Sigma_{-(m+1)} \circ \dots \circ \Sigma_{-(m + M)}),
  \end{align*}
  where the $\Si_j$ are given in Section \ref{sec:xch}.
  Under 
  $U_m$, the track at level $m+ M$ is moved to the position 
  directly above that at level $m-1$,
  and the level $-(m+M+1)$ track below the level $-m$ track.
  We have that
  \begin{align*} 
    U_m (G^m) = G^{m+1}.
  \end{align*}

  Let $\omega^0$ be a configuration on $G^0$ such that $\wt{\Arm}^{u, s_0}_k(N, n)$ occurs.
  Set $j = 1$ when $k = 1$, and $j = k/2$ when $k \geq2$. 
  There exist vertices $x_1,\ldots, x_j$, $y_1, \ldots,y_j$ 
  and, when $k = 2$, $x_1^*$, $ y_1^*$,
  all lying in the set $\{v_{m,0}: m \in \ZZ\}$ of vertices of $G^\di$, such that:
  \begin{letlist}
  \item   $x_i \xlra{G^0, \omega^0} y_i$ and $x_i \nxlra{\ \ G^0, \omega^0} x_{i'}$ for $i \neq i'$,
  \item   $x^*_1 \dxlra{G^0, \omega^0} y^*_1$, when $k = 2$,
  \item   $d^\di(v_{0,0}, x_i) \leq N$, $d^\di(v_{0,0}, y_i) > n$, 
  \item 
    $d^\di(v_{0,0}, x_1^*) \leq N$, $d^\di(v_{0,0}, y_1^*) > n$, when $k=2$,
  \item   $C_{x_i} \subseteq \La_u^\di(M)$ and, when $k=2$, $C^*_{x_1^*} \subseteq \La_u^\di(M)$.
  \end{letlist}
  
  As we apply $U_{M-1} \circ \dots \circ U_{0}$ to $(G^0, \omega^0)$, 
  the images of paths from each of $x_i$, $y_i$, and $x^*_1, y^*_1$ 
  retain their starting points. 
  
Each $\La_u^\di(r)$ has a diamond shape. 
By an argument similar to that of Lemma \ref{D_control}, for $0 \le m\le M$,
  \begin{align*}
    C_{x_i}(\om^m)     \subseteq \La_u^\di(M + 2m), \quad
    C^*_{x_1^*}(\om^m) \subseteq \La_u^\di(M + 2m).
  \end{align*}
  Moreover, since $C_{x_i}(\om^m)$ and $C^*_{x_1^*}(\om^m)$ 
  do not extend to the left/right boundaries of $G^m$, 
  these clusters neither break nor merge with one another. 
  Therefore,
\begin{letlist}
  \item   $x_i \xlra{G^M, \omega^M} y_i$ and $x_i \nxlra{\ \ G^M, \omega^M} x_{i'}$ for $i \neq i'$,
  \item   $x_1^* \dxlra{G^M, \omega^M} y_1^*$, when $k = 2$,
  \end{letlist}
  so that  $\omega^M \in A_k(\cd N,\cd^{-1}n)$. 
(A related discussion may be found in \cite[Sect.\ 3.2]{GM2}.)
  In conclusion, 
  there exists $c_3>0$ such that
  \begin{align*}
    \PP_{G} [ \Arm_k (N,n) ]  
    &\leq c_3 \PP_{G^0} [ \wt{\Arm}_k (N,n) ]  \qquad\quad \text{ by \eqref{exp_equiv_c}}\\
    &\leq c_3 \PP_{G^M} [ \Arm_k (\cd N, \cd^{-1} n) ].
  \end{align*}

  Since the intersection of any $G^m$ with $\Ann(\cd N,\cd^{-1}n)$ is contained in
  $\La_u^\di(M)$, we have by the discussion after \eqref{G610} that there
exists $c_4>0$ with
  \begin{align*}
    \PP_{G} [ \Arm_k (N,n) ] 
    &\le c_3 \PP_{\balpha, \bef } [ \Arm_k (\cd N, \cd^{-1} n) ] \\
    &\leq c_3 c_4 \PP_{\balpha, \bef } [ \Arm_k (N, n) ],
  \end{align*}
  by \eqref{exp_equiv_a} and  \eqref{exp_equiv_b}, iterated.
  The second inequality of Lemma \ref{exp_transport1} is proved. 

  Turning to the first inequality, 
  let $\omega^M$ be a configuration on $G^M$ such that $\wt{\Arm}_k(N, n)$ occurs
  (the arm event is defined in terms of $v_{0,0}$ and the horizontal track at level $0$).
  It may be seen as above that 
   $\omega^0 = U_{M-1} \circ \dots \circ U_{0}(\omega^M)$ 
  is a configuration on $G^0$ contained in $\Arm_k(\cd N , \cd^{-1} n)$. 
 Furthermore,
  \begin{align*}
    \PP_{\balpha, \bef } [ \Arm_k (N,n) ]  
    \leq c_3 \PP_{G^M} [ \wt{\Arm}_k (N,n) ]
    \leq c_3 c_4 \PP_{G} [ \Arm_k (N, n) ].
  \end{align*}
  The proof is complete.
\end{proof}

\begin{cor}\label{exp_transport1_cor}
  There exist constants $c_i > 0$ such that, for $n > N$ sufficiently large
  and any isoradial square lattice $G_{\balpha, \bbeta} \in \sG(\eps, I)$, 
  \begin{align}
    c_1 \PP_{0,\pi/2}[\Arm_k(N, n)] \leq 
    \PP_{\balpha, \bbeta}[\Arm_k(N, n)] \leq 
    c_2 \PP_{0,\pi/2}[\Arm_k(N, n)].
    \label{exp_transport1_eq}
  \end{align}
\end{cor}

\begin{proof}
  If $\balpha$ is a constant vector $(\a_0)$, \eqref{exp_transport1_eq} follows 
  by Lemma \ref{exp_transport1} with $\bef  = \a_0 + \pi/2$. 
  
  For $\balpha$ non-constant, we apply  Lemma \ref{exp_transport1} 
  with $\bef  = \b_0$, thus bounding
  the arm-event probabilities for $G_{\balpha,\bbeta}$ by those 
  for $G_{\balpha,\b_0}$.
  Now, $G_{\balpha,\b_0}$ is of the type analysed above, and the conclusion follows. 
\end{proof}
  
\subsection{Proof of Proposition \ref{exp_transport}: The general case}\label{sec:proof_arm2}

Let $G \in \sG(\eps,I)$, and 
let $(s_j:j\in\ZZ)$  and $(t_i:i\in\ZZ)$ be two families of tracks forming a square grid of $G$,
duly oriented.
Write $(r_i : i \in \ZZ)$ for 
the sequence of all tracks other than the $s_j$, 
indexed and oriented according to their intersections with $s_0$, with $r_0 = t_0$,
and including the $t_i$ in increasing order. 
Let $\b_j$ be the transverse angle of $s_j$, 
and $\pi+\alpha_i$ that of $r_i$. 
Since each $r_i$ intersects each $s_j$, 
the vectors $\balpha = (\alpha_i : i\in \ZZ)$,
$\bbeta = (\beta_j:j\in \ZZ)$ satisfy
\eqref{bounded_angles}, and hence $G_{\balpha, \bbeta}$ 
is an isoradial square lattice satisfying \BAC$(\eps)$.
As in Lemma \ref{grid_slide}, we may retain the labelling of tracks throughout the proof.
Let $u$ be the vertex adjacent to $s_0$ and $t_0$, 
below and \resp\ left of these tracks. 
All arm-events in the following  are centred at the vertex $u$ 
and expressed in terms of the track $s_0$. 

\begin{lemma}\label{exp_transport2}
  There exist constants $c_1, c_2 >0$ such that, for $n > N$ sufficiently large, 
  \[
  c_1 \PP_{\balpha, \bbeta}[\Arm_k (N, n)] \leq 
  \PP_G[\Arm_k(N, n)] \leq 
  c_2 \PP_{\balpha, \bbeta}[\Arm_k(N, n)].
  \]
\end{lemma}

This lemma, together with Corollary \ref{exp_transport1_cor}, 
implies Proposition \ref{exp_transport} for arm events centred at $u$. 
By the \sgp, any vertex is within bounded distance of 
one of the tracks $(s_j : j \in \ZZ)$. 
This allows us to extend the conclusion to arm events centred at 
any vertex. 

\begin{proof}
  Let $n \in \NN$ and $M = \lceil \cd n \rceil $.
  By Lemma \ref{grid_slide}, applied in two stages above and below $s_0$, 
  there exists a finite sequence $R^+$ of star--triangle transformations
  such that, in $G^M := R^+(G)$, 
  the tracks $s_{-M},\dots, s_M$ are adjacent between $t_{-M}$ and $t_M$. 
  Moreover, no \stt\ in $R^+$ involves a rhombus lying in $s_0$. 
  The sequence $R^+$ has an inverse sequence denoted $R^-$.
  Note that $G^M$ agrees with $G_{\balpha, \bbeta}$ inside $B_n + u \subseteq \La_u^\di(M)$. 

  Let $\omega$ be a configuration on $G$ belonging to $\wt{\Arm}_k (N, n)$,
  and let vertices $x_i$, $y_i$ be given accordingly. 
  Consider the image configuration $\omega^M = R^+ (\omega^0)$ on $G^M$. 
  By considering the action of the transformation $R^+$, we may see that
  \begin{letlist}
  \item   $x_i   \xlra{G^M, \omega^M}   y_i$   and $x_i \nxlra{\ \ G^M, \omega^M} x_{i'}$ for $i \neq i'$,
  \item   $x_1^* \dxlra{G^M, \omega^M} y_1^*$, when $k = 2$.
  \end{letlist}
  Taken together with \eqref{di_euclid},  
  this implies that $\om^M \in \Arm_k(\cd N, \cd^{-1} n)$.
  Therefore, there exist $c_i>0$ such that
  \begin{align*}
    \PP_{G} [ \Arm_k (N,n) ]  
    &\leq c_3 \PP_{G} [ \wt{\Arm}_k (N, n) ] \quad && \text{ by \eqref{exp_equiv_c}}\\
    &\leq c_3 \PP_{G^M} [ \Arm_k (\cd N, \cd^{-1} n) ] \\
    &= c_3 \PP_{\balpha, \bbeta} [ \Arm_k (\cd N, \cd^{-1} n) ] \\
    &\leq c_3 c_4 \PP_{\balpha, \bbeta} [ \Arm_k (N, n) ] \quad && \text{ by \eqref{exp_equiv_a} and \eqref{exp_equiv_b}},
  \end{align*}
  and the second inequality of the lemma is proved.

  Conversely, let $\omega^M$ be a configuration on $G^M$ belonging to $\wt{\Arm}_k (N, n)$.
  By applying the inverse transformation, we obtain the configuration 
  $\omega =  R^- (\omega^M)$ on $G$. As above,
  \begin{align*}
    \PP_{\balpha, \bbeta}  [ \Arm_k (N,n) ]  
    &= \PP_{G^M}  [ \Arm_k (N,n) ]   \\
    &\leq c_3 \PP_{G^M} [ \wt{\Arm}_k (N, n) ] \quad && \text{ by \eqref{exp_equiv_c}}\\
    &\leq c_3 \PP_{G} [ \Arm_k (\cd N, \cd^{-1} n) ] \\
    &\leq c_3 c_4 \PP_{G} [ \Arm_k (N, n) ] \quad && \text{ by \eqref{exp_equiv_a} and \eqref{exp_equiv_b}}.
  \end{align*}
  This concludes the proof of the first inequality of the lemma.
\end{proof}

\subsection{Proof of Proposition \ref{exp_equiv}}\label{sec:sep_thm}

This section is devoted to the proof of Proposition \ref{exp_equiv},
and is not otherwise relevant to the rest of the paper. 
The two main ingredients of the proof are the separation theorem (Theorem \ref{separation}) and 
the equivalence of metrics, \eqref{di_euclid}.

\subsubsection{Separation theorem}
The so-called separation theorem was proved by Kesten in \cite{Kesten87}
(and elaborated  in \cite{Nolin}) in the context of homogeneous site percolation. 
The theorem is set more generally in \cite[Thm 3.5]{GM2}, and this gives rise
to the current version, Theorem \ref{separation}, to be used in the proof of Proposition \ref{exp_equiv}.
The formal statement of Theorem \ref{separation} requires some notation from 
\cite{GM2,Kesten87,Nolin} which, for completeness, is presented below.

Let $G \in \sG(\eps,I)$.
By Theorem \ref{main} (see also \eqref{G2231}), $\PP_G$ satisfies the \bxp\ 
\BXP$(\de)$ with $\de=\de(\eps,I)$.
Let $\sigma$ be  a colour sequence of length $k$. The event  $\Arm_\sigma(N,n)$
requires the existence of a number of `arms' crossing an annulus.
Roughly speaking, the separation theorem implies that
the extremities of these arms may be taken to be distant from one another.  

\begin{figure}[htb]
  \begin{center}
    \includegraphics[width=0.8\textwidth]{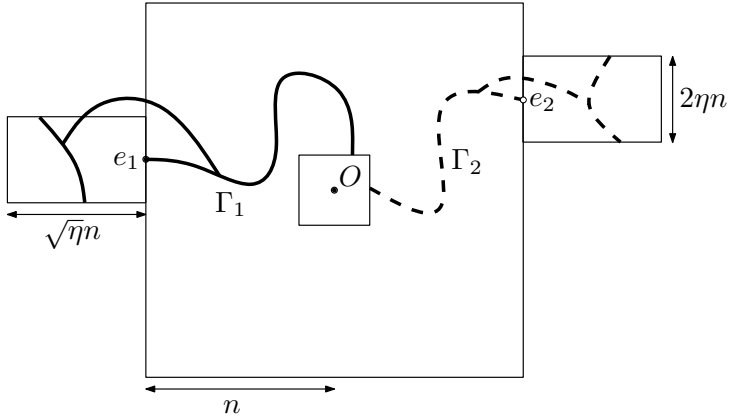}
  \end{center}
  \caption{A primal $\eta$-exterior-fence $\Ga_1$ with exterior endpoint $e_1$,
    and a dual $\eta$-exterior-fence $\Ga_2$.}
  \label{fig:fence}
\end{figure}

We shall consider open and (dual) open$^*$ crossings between 
the interior and exterior boundaries of $\Ann(N,n)$. 
For clarity, we concentrate first on the behaviour of crossings at their exterior endpoints.
Let $\eta\in(0,1)$. A primal (\resp, dual) \emph{$\eta$-exterior-fence}
is a set $\Ga$ of connected open (\resp, open$^*$) paths comprising the union of:
\begin{romlist}
\item a crossing of $\Ann(N,n)$ from its interior to its exterior
  boundary, with exterior endpoint denoted 
  $\ext(\Ga)$,
\end{romlist}
together with certain further paths which we describe thus under the assumption that
$\ext(\Ga)=(n,y)$ is on the right side of $\pd B_n$:
\begin{romlist}
\item[(ii)] a vertical crossing of the box 
  $[n , (1+\sqrt{\eta})n] \times [y - \eta n, y + \eta n]$,
\item[(iii)] a connection between the above two crossings, 
  contained in $\ext(\Ga) +B_{\sqrt{\eta}n}$.
\end{romlist}
If $\ext(\Ga)$ is on a different side of $\pd B_n$, the event of condition (ii) is replaced by
an appropriately rotated and translated event.  
This definition is illustrated in Figure \ref{fig:fence}.

\begin{figure}[htb]
  \begin{center}
    \includegraphics[width=0.7\textwidth]{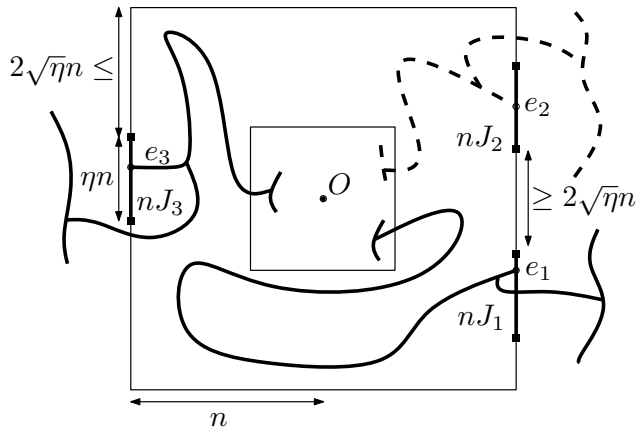}
  \end{center}
  \caption{The event $\Arm_{\sigma}^{J,J} (N,n)$ with 
    $\sigma = (1,0,1)$ and $\eta$-landing-sequence $J$. 
    Each crossing $\Gamma_i$ is an $\eta$-fence 
    with exterior endpoint $e_i \in n J_i$.}
  \label{fig:separated_fences}
\end{figure}

One may similarly define an $\eta$-\emph{interior}-fence by considering the
behaviour of the crossing near its interior endpoint.
We introduce also the concept of a primal (\resp, dual)
$\eta$-\emph{fence}; this is a union of
an open (\resp, open$^*$)  crossing
of $\Ann(N,n)$ together with further paths in the vicinities
of both interior and exterior endpoints along the lines of the above definitions. 

An \emph{$\eta$-landing-sequence} is a sequence of closed sub-intervals 
$J =\linebreak  (J_i : i =1,2, \dots, k)$ of $\pd B_1$,
taken in anticlockwise order,
such that each $J_i$ has length $\eta$, and
the minimal distance between any two intervals, and 
between any interval and a corner of $B_1$, 
is greater than $2\sqrt\eta$. 
We shall assume that
\begin{equation}\label{eta}
  0 < k(\eta+2\sqrt\eta) < 8,
\end{equation}
so that $\eta$-landing-sequences exist.

Let $\eta,\eta'$ satisfy \eqref{eta}, and let $J$ (\resp, $J'$) be an $\eta$-landing-sequence
(\resp, $\eta'$-landing-sequence). 
Write  $\Arm_{\sigma}^{J,J'} (N,n)$
for the event that there exists a sequence of 
$\eta$-fences  $(\Gamma_i : i =1,2, \dots, k)$
in the annulus $\Ann(N,n)$, 
with colours  prescribed by $\sigma$, such that, for all $i$,
the interior (\resp, exterior) endpoint  of $\Ga_i$ lies in $N J_i$ (\resp, $n J'_i$).
These definitions are illustrated in Figure \ref{fig:separated_fences}.

We now state the separation theorem. 
The proof is omitted, and may be constructed via careful readings 
of the appropriate sections of \cite{Kesten87,Nolin}.  

\begin{thm}[Separation theorem]\label{separation}
  For $\eta_0>0$ and a colour sequence $\sigma$,
  there exist constants $c > 0$ and $N_1 \in \NN$ 
  depending only on $\eta_0$, $\sigma$, $\eps$, and $I$ such that:
  for all $\eta,\eta' \ge \eta_0$ satisfying \eqref{eta}, 
  all $\eta$-landing-sequences $J$ and $\eta'$-landing-sequences $J'$,
  and all $N \geq N_1$ and $n \geq 2N$, we have
  \begin{align*}
    \PP_G [ \Arm_{\sigma}^{J,J'} (N,n) ] \geq 
    c \PP_G [ \Arm_{\sigma} (N,n) ].
  \end{align*}
\end{thm}

\subsubsection{Proof of Proposition \ref{exp_equiv}}

Inequalities \eqref{exp_equiv_a} and \eqref{exp_equiv_b} 
follow from Theorem \ref{separation} and the \bxp. 
The proof, omitted here, is essentially that of \cite[Prop.\ 12]{Nolin}, and
uses the extension of paths by judiciously positioned box-crossings.
We turn therefore to the proof of \eqref{exp_equiv_c}. 

Consider the first inequality of \eqref{exp_equiv_c} (the second is easier to prove).
The idea is as follows.
Suppose that $A_k(N,n)$ occurs (together with some additional assumptions). 
One may construct a bounded
number of open or open$^*$ box-crossings in order
to obtain $\wt A_k(N,n)$. These two arm-events are given in terms of annuli
defined via different metrics --- the Euclidean metric and $d^\di$, \resp\ ---
but the radii of these annuli are comparable by \eqref{di_euclid}.

Assume $n/N \ge 2$, and let 
\begin{equation}\label{G604}
  M = \cd^{-1} N,\quad m = \cd n,
\end{equation}
with $\cd$ as in \eqref{di_euclid}.  
Let $k \in \{1,2,4,\dots\}$,  $\sigma = (1,0,1,0, \ldots)$,
and consider the corresponding arm-event $\Arm_k(M,m)$. 
All constants in the following proof may depend on $k$, $\eps$, and $I$
but, unless otherwise specified, on nothing else.   
All arm-events that follow are assumed centred at the vertex $u$ adjacent to a track $s$.
By translation, we may assume that $u$ is the origin of $\RR^2$.
In order to gain some control over the geometry of $s$,
we may assume, without loss of generality,  that its transverse angle $\b$  
satisfies $\b\in[\frac14 \pi, \frac34 \pi]$.
    
Let $\eta=\eta(k)>0$ satisfy \eqref{eta}, and let
$J$ be an $\eta$-landing sequence of length $k$, 
entirely contained in $\{1\} \times [0,1]$, 
with $J_1$ being the lowest interval. 
Henceforth assume $M  \geq N_1$, where $N_1$ is given 
in Theorem \ref{separation} with $\eta_0=\eta$. 
By that theorem, there exists $c_0>0$ such that
\begin{align}\label{G603}
  \PP_G[ \Arm_k^{J,J} (M,m) ] \geq c_0 \PP_G \left[ \Arm_k (M,m) \right].
\end{align}

Let $(M, v_i)$ be the lower endpoint of $MJ_i$, and $(M, w_i)$ the upper. 
Let $H_M$ be the event that,
for  $i \in \{1, 2,\dots,k\}$,  the following crossings of colour $\sigma_i$ exist:
\begin{letlist}
\item a horizontal crossing of
  $[-w_i,M]\times [v_i,w_i] $,
\item a vertical crossing of 
  $[-w_i,-v_i]\times [-w_i, w_i] $, 
\item for $i$ odd,
  a horizontal crossing of
  $[-w_i,w_i]\times [-w_i,-v_i] $,
\item for $i$ even, 
  a horizontal crossing of
  $[-w_i,w_k]\times [-w_i,-v_i] $.
\end{letlist}  
If $k \geq 4$, 
we require also an open$^*$ vertical crossing of $[v_k,w_k]\times [-w_k,0] $.
The event $H_M$  depends only on the configuration inside $B_M$,
and is illustrated 
in Figure \ref{fig:arm_equiv}.  
 
\begin{figure}[htb]
  \centering
  \includegraphics[width=1.0\textwidth]{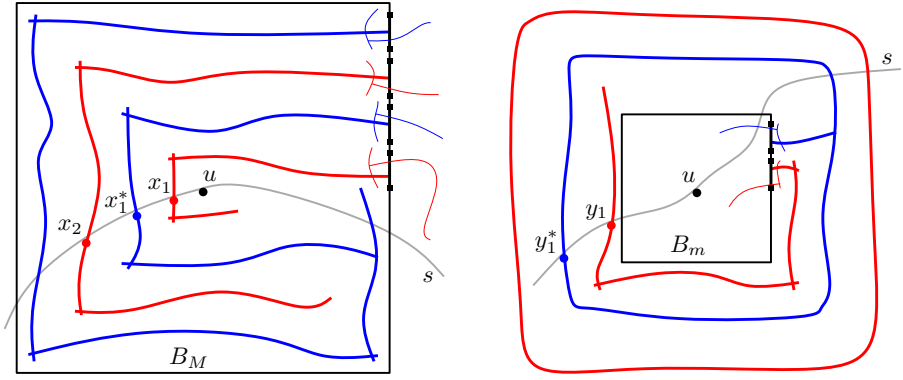}
  \caption{\emph{Left}: The event $H_M$ for $k=4$. 
    The red paths are open, the blue paths are open$^*$. 
    The  thin coloured paths are parts of the interior fences of $\Arm_k^{J,J} (M,m)$.
    \emph{Right}: The event $K_m$ for $k =2$, together with parts of
    the exterior fences of the arm-event. 
    The track $s$ intersects  the open/open$^*$ crossings  
    just above the points labelled $x_i$ and $y_i$.}
  \label{fig:arm_equiv}
\end{figure}

Let $(m, v_i)$ be the lower endpoint of $mJ_i$, and $(m, w_i)$  the upper. 
Let $K_m$ be the event that, for $i \in \{1, 2,\dots,k\}$,  
the following crossings of colour $\sigma_i$ exist:
\begin{letlist}
\item a horizontal crossing of 
  $[m, (m+w_i)]\times[v_i, w_i]$,
\item a vertical crossing of 
  $[ (m+v_i), m+w_i]\times[-(m + w_i), w_i]$,
\item a horizontal crossing of
  $[ -(m+w_i), m+w_i]\times[-(m + w_i), -(m + v_i)]$,
\item if $i$ is odd, a vertical crossing of 
  $[ -(m+w_i), -(m+v_i)]\times[-(m + w_i), m + w_i]$,
\item if $i$ is even, a vertical crossing of 
  $[ -(m+w_i), -(m+v_i)]\times[-(m + w_i), m + w_k]$.
\end{letlist}
We require in addition the following:
\begin{letlist}
\item[(f)] when $k = 1$, an open$^*$ circuit in $\Ann(2m, 3m)$, 
\item[(g)] when $k = 2$, an open circuit in $\Ann(2m, 3m)$,
\item[(h)] when $k \geq 2$, an open$^*$ circuit in $\Ann(m+v_k, m+w_k)$.
\end{letlist}
The event $K_m$ depends only on the configuration inside $\Ann(m, 3m)$, 
and is illustrated in Figure \ref{fig:arm_equiv}.  

Set $j=1$ when $k=1$, and $j=k/2$ when $k \ge 2$.
We claim that, on $H_M \cap K_m \cap \Arm_k^{J,J} (M,m)$,
there exist vertices $x_1,\dots, x_j$, $y_1, \dots,y_j$ and, 
when $k=2$, $x_1^*$, $y_1^*$, 
all adjacent to $s$ and on the same side of $s$ as $u$, such that:
\begin{letlist}
\item   $ x_i \in B_M$, $ y_i \notin B_m$ and, when $k=2$, $ x_1^* \in B_M$, $ y_1^* \notin B_M$,
\item   $x_i \xlra{} y_i$ and $x_i \nxlra{} x_{i'}$ for $i \neq i'$,
\item   $x_1^* \dxlra{} y_1^*$ when $k =2$,
\item   $C_{x_i} \subseteq B_{3m}$ and, when $k=2$, $C^*_{x_1^*} \subseteq B_{3m}$.
\end{letlist}
This claim holds as follows. 
The crossings in the definition of $H_M$ (\resp, $K_m$) may be regarded as extensions 
of the arms of $\Arm_k^{J,J} (M,m) $ inside $B_M$ (\resp, outside $B_m$). 
Let $\la$ be the straight line with inclination 
$\b\in[\frac14 \pi,\frac34 \pi]$, passing through $u$.
Since $s$ corresponds to a chain of rhombi with common sides parallel to $\la$,
it intersects $\la$ 
only in the edge of $G^\di$ crossing $s$ and containing $u$.
Therefore, the part of $s$ to the left of $\la$ necessarily
intersects all the above extensions. 
These intersections provide the $x_i$, $y_i$ and, when $k=2$, $x_1^*$, $y_1^*$. 
The remaining statements above are implied 
by the definitions of the relevant events.

By \eqref{di_euclid}, \eqref{G601}, and \eqref{G604}, 
  \begin{align*}
    H_M \cap K_m \cap \Arm_k^{J,J} (M,m) \subseteq \wt{\Arm}_k(\cd M,\cd^{-1} m).
  \end{align*}
  By the extended Harris--FKG inequality of \cite[Lemma 3]{Kesten87} (see also 
\cite[Lemma 12]{Nolin}), 
  \begin{align*}
    \PP_G \bigl( H_M \cap K_m \cap \Arm_k^{J,J} (M,m) \bigr) \geq 
    \PP_G (H_M) 
    \PP_G ( K_m ) 
    \PP_G [ \Arm_k^{J,J} (M,m) ].
  \end{align*}
  The events $H_M$ and  $K_m$ are given in terms of crossings of boxes with aspect-ratios 
  independent of $M$ and $m$. 
  Therefore, there exists $c_1 > 0$  such that, for 
  $m$ and $M$ sufficiently large,
  $\PP_G ( H_M ) \geq c_1$ and $\PP_G ( K_m ) \geq c_1$.
  In conclusion, by \eqref{G604},
there exists $c_5>0$ such that, for $n/N \ge 2$, 
  \begin{alignat*}{2}
      \PP_G [ \wt{\Arm}_k(N,n)  ]
      &\geq     \PP_G (H_M )   \PP_G ( K_m )  \PP_G [ \Arm_k^{J,J} (M,m) ] \\
      &\geq     c_1^2 c_0 \PP_G [ \Arm_k (M,m) ] &\quad &\text{by \eqref{G603}}\\
      &\geq     c_1^2 c_0 c_5 \PP_G [ \Arm_k (\cd M, \cd^{-1} m) ]\\
&=c_1^2 c_0 c_5 \PP_G [ \Arm_k (N,n) ] & &\text{by \eqref{G604}}
  \end{alignat*}
where the third inequality holds by iteration of \eqref{exp_equiv_a}--\eqref{exp_equiv_b}.
The first inequality of \eqref{exp_equiv_c} follows. 

  The second inequality is simpler. Set $M= \cd N$ and $m= \cd^{-1} n$.
  By  \eqref{di_euclid},
 $\wt{\Arm}_k(\cd^{-1}M,\cd m) \subseteq \Arm_k (M,m)$.
  By iteration of \eqref{exp_equiv_a}--\eqref{exp_equiv_b},
  there exists $c_6 > 0$ such that, for $m>M$ sufficiently large,
  \begin{align*} 
    \PP_G [ \wt{\Arm}_k(\cd^{-1}M,\cd m)  ]
    &\leq \PP_G [  \Arm_k (M,m) ]\\
&\leq c_6 \PP_G [ \Arm_k (\cd^{-1}M,\cd m)].
  \end{align*}
This concludes the proof of Proposition \ref{exp_equiv}.

\subsection{Proof of Theorem \ref{main3}(b)}\label{sec:kesten}

Theorem \ref{main3}(b) follows from 
Theorem \ref{main3}(a) and the following proposition. 

\begin{prop}\label{kesten_scaling}
  Let $G \in \sG$. 
  If either $\rho$ or $\eta$ exists for $G$, 
  then $\rho$, $\eta$, and $\de$ exist for $G$, and they satisfy 
 $\eta\rho=2$ and $2\rho = \de + 1$.
\end{prop}

Kesten \cite{Kes87a} proved this statement for the special case
of homogeneous bond percolation on the square lattice, 
using arguments that rely heavily on estimates from the earlier \cite{Kes86}.
As noted in these papers, the conclusions 
may be extended to periodic models satisfying the \bxp.
They may in fact be extended still further,
to planar percolation processes $(G, \PP)$ embedded in $\RR^2$
that satisfy the following:
\begin{letlist}
\item a uniform upper bound on the lengths of both primal and dual edges,
\item a (strictly positive) lower and upper bound on the density of vertices of $G$, 
\item$\PP$ and the measure $\PP^*$ of the dual process have the \bxp,
\item the existence of a constant $c >0$ such that, for $n$ sufficiently large and any vertices $u$, $v$,
  \begin{align*}
    c^{-1} \PP\bigl(\rad(C_u) > n\bigr) \leq \PP\bigl(\rad(C_v) > n\bigr) 
\leq c\PP\bigl(\rad(C_u) > n\bigr).
  \end{align*}
\end{letlist}
Conditions (a) and (b) are satisfied by any isoradial graph with the \bac.
By Theorem \ref{main} and Proposition \ref{exp_transport}, 
conditions (c) and (d) are satisfied by isoradial graphs with the \bac\ and the \sgp.
Proposition \ref{kesten_scaling} follows for general $G \in \sG$. 

No major changes are required in adapting the proofs of \cite{Kes86,Kes87a},
and we include no further details here. 

\section*{Acknowledgements}
The authors thank C\'edric Boutillier, B\'eatrice de Tili\`ere, and Richard Kenyon
for their suggestions concerning this work.
GRG was supported in part by the EPSRC under grant
EP/103372X/1.
IM acknowledges financial support from 
the EPSRC, the Cambridge European Trust,  DPMMS,
and the Peter Whittle Fund, Cambridge University,

\bibliographystyle{amsplain}
\bibliography{iso-final4}

\end{document}